\documentclass[11pt]{amsart}

\usepackage{fullpage}
\usepackage{xcolor,cancel}

\usepackage{hyperref}

\usepackage{amsmath,amssymb,amsfonts,mathabx,wrapfig,enumerate,multicol,tikz-cd,tikz,caption,bm,relsize} 
\usepackage{subcaption}
\usepackage[numbers]{natbib}

\tikzset{>=latex}

\newtheorem{theorem}{Theorem}[section]

\newtheorem{lemma}[theorem]{Lemma}

\newtheorem{prop}[theorem]{Proposition}
\newtheorem{corollary}[theorem]{Corollary}

\newtheorem{definition}[theorem]{Definition}

\newtheorem{innercustomthm}{Theorem}
\newenvironment{customthm}[1]
  {\renewcommand\theinnercustomthm{#1}\innercustomthm}
  {\endinnercustomthm}

\newtheorem{innercustomcor}{Corollary}
\newenvironment{customcor}[1]
  {\renewcommand\theinnercustomcor{#1}\innercustomcor}
  {\endinnercustomcor}

\theoremstyle{definition}

\DeclareMathOperator{\Ric}{Ric}
\DeclareMathOperator{\K}{K}
\DeclareMathOperator{\2}{II}
\DeclareMathOperator{\R}{R}

\DeclareMathOperator{\Bee}{B}
\DeclareMathOperator{\Ee}{E}
\DeclareMathOperator{\Zee}{Z}
\DeclareMathOperator{\Ex}{X}
\DeclareMathOperator{\Wy}{Y}

\DeclareMathOperator{\En}{N}
\DeclareMathOperator{\Em}{M}
\DeclareMathOperator{\Sphere}{S}

\DeclareMathOperator{\CP}{\mathbf{C}P}
\DeclareMathOperator{\HP}{\mathbf{H}P}
\DeclareMathOperator{\OP}{\mathbf{O}P}
\DeclareMathOperator{\Disk}{D}

\title{Ricci-positive metrics on connected sums of products with arbitrarily many spheres }
\author{Bradley Lewis Burdick}

\begin{document}
\maketitle

\begin{abstract} In this paper we construct Ricci-positive metrics on the connected sum of products of arbitrarily many spheres provided the dimensions of all but one sphere in each summand are at least 3. There are two new technical theorems required to extend previous results on sums of products of two spheres. The first theorem is a gluing construction for Ricci-positive manifolds with corners generalizing a gluing construction of Perelman for Ricci-positive manifolds with boundaries. Our construction gives a sufficient condition for gluing together two Ricci-positive manifolds with corners along isometric faces so that the resulting smooth manifold with boundary will be Ricci-positive and have convex boundary. The second theorem claims that one can deform the boundary of a Ricci-positive Riemannian manifold with convex boundary along a Ricci-positive isotopy while preserving Ricci-positivity and boundary convexity. 
\end{abstract}

\tableofcontents

\section{Introduction}\label{intro} 

\subsection{Background and Main Results}
The existence of a metric of positive scalar curvature was shown to be invariant under surgery in codimension at least $3$ by Gromov-Lawson \cite{GL} and Schoen-Yau \cite{SchY}. In particular, the existence of a metric of positive scalar curvature is always invariant under connected sum. Under rigid metric conditions, the existence of a metric of positive Ricci curvature was shown to be invariant in surgery in codimension at least $3$ and dimension at least $2$ by Sha-Yang \cite{SY2} and Wraith \cite{Wraith1}. The techniques of \cite{SchY,GL} fundamentally cannot work to preserve positive Ricci curvature,\footnote{See \cite{Wolf,Heol} for more general results using the techniques of \cite{GL}. Both papers make it clear why positive Ricci curvature is excluded.} while the Ricci-positive surgery constructions of \cite{SY2,Wraith1} also cannot be extended to connected sum. 

Motivated by the case of positive scalar curvature, in this paper we will consider under what conditions is the existence of a Ricci-positive metric preserved under connected sum. If we restrict ourselves to simply connected manifolds,\footnote{By Myers' Theorem and the Seifert-van Kampen Theorem, there is no Ricci-positive metric on the $\Em^n_1 \# \Em_2^n$, where both of the $\Em_i^n$'s are not simply connected.} there is no known or expected topological obstruction to the existence of a metric of positive Ricci curvature on the connected sum. As explained, we need to take an entirely different approach from the existing curvature stable surgery results. 
 
While there are few general results, there are number of interesting examples of Ricci-positive connected sums. The earliest nontrivial example is due to Cheeger \cite{Chee}, using so-called Cheeger deformations, to construct a metric of positive Ricci curvature\footnote{Technically the metrics have nonnegative sectional curvature, but by \cite{Ehr} these metrics can be deformed slightly to have positive Ricci curvature.} on the connected sum of any two of the following manifolds: $\CP^n$, $\HP^n$, and $\OP^2$. Using a similar technique, Sha and Yang \cite{SY1} constructed a Ricci-positive metric on $\#_k \left( \Sphere^3\times \CP^2\right)$ for each $k$.\footnote{Technically the space considered in \cite{SY1} has a slightly different topology, but as illustrated by \cite[Proposition 5.2]{BLB1}, a slight modification of this argument produces a Ricci-positive metric on the topology claimed here.} Using the Ricci-positive surgery result alluded to above, Sha and Yang \cite{SY2} constructed Ricci-positive metrics on $\#_k\left( \Sphere^n\times \Sphere^m\right)$ for each $k$ after identifying it with iterated surgery on $\Sphere^{n-1}\times \Sphere^{m+1}$. Later in \cite{Wraith4}, Wraith gave a careful generalization of the construction of \cite{SY2} that allowed for metrics of positive Ricci curvature on $\#_i \left(\Sphere^{n_i}\times \Sphere^{m_i}\right)$ provided $n_i,m_i\ge 3$. 

The main result of this paper claims the existence of a metric of positive Ricci curvature on connected sums where each summand is the total space of iterated sphere bundles. 

\begin{customthm}{A}\label{main} Given a sequence of manifolds $\Bee^n_i \in \{ \Sphere^n, \CP^{n/2}, \HP^{n/4}, \OP^{(n=16)/8} : n\ge 2\}$, if $\Ex_i^{N}$ be the total space of iterated linear sphere bundles over $\Bee_i^{n_i}$ with fibers of dimension at least $3$, then there is a Ricci-positive metric on the space
$$ \left(\Sphere^N/G\right) \# \left( \#_i \Ex_i^N\right),$$
where $G$ is any cyclic group\footnote{Originally, \cite[Corollary 4.5]{BLB1} was stated in terms of an algebraic condition on the action of $G$. The MathSciNet reviewer for \cite{BLB1}, Lee Kennard, pointed out that the work of \cite{Wolf2} classified such groups. See \cite[Corollary 4.5]{BLB1} for the original statement and \cite[Lemma 1.2.9]{BLB2} for the argument that the only such $G$ are cyclic.} acting freely by diffeomorphisms on $\Sphere^N$.

\end{customthm}

\noindent Here we say that $\Ex^N$ is the total space of an iterated linear sphere bundle over $\Bee^{n}$ with fibers of dimension at least $3$ if there is a sequence $\Ex_i$ with $0\le i\le k$ such that $\Ex^N = \Ex_k$, $\Ex_0 = \Bee^{n}$, and $\Ex_i = \Sphere(\Ee_i)$ where $\Ee_i$ is a rank $r_i+1$ vector bundle over $\Ex_{i-1}$ for $ 1\le i\le k$ with $r_i\ge 3$. Note that Theorem \ref{main} produces many new examples even when we take trivial bundles, see Corollary \ref{productmain} below for details.

The perspective we will take follows the work of \cite{Per1,Per2}. In this pair of papers, Perelman constructed a Ricci-positive metric on $\#_k \CP^2$ by taking metrics on $\CP^2\setminus \Disk^4$ in \cite{Per2} similar to those introduced in Cheeger and by viewing the space $\#_k\CP^2$ in \cite{Per1} as $\Sphere^4$ blown up at $k$ disjoint points. The key features of the metrics constructed on $\CP^2\setminus \Disk^4$ in \cite{Per1} are that they positive Ricci curvature and the boundary is round and convex. Following Perelman we define \emph{a core metric} for a closed manifold $\Em^n$ to a be a metric of positive Ricci curvature so that $\Em^n\setminus \Disk^n$ has convex boundary isometric to the unit, round sphere for some embedded disk. Other than the construction of core metrics for $\CP^2$, we summarize the main contribution of \cite{Per1} in two theorems: Theorems \ref{glue} and Theorem \ref{docking} below. The first is a gluing theorem for Riemannian manifolds of positive Ricci curvature, which removes the need for the careful bending arguments that appear in \cite{Chee,SY1,SY2,Wraith1,Wraith4}, and the second is a construction of a very particular family of Ricci-positive metrics on the punctured spheres, which we call \emph{the docking station}. The Ricci-positive metric on $\#_k\CP^2$  is then constructed by using Theorem \ref{glue} to glue $k$-copies of the core metric of $\CP^2$ to the docking station. 

The reason we take this perspective is that there is nothing particular to the topology of $\CP^2$ used in the construction of \cite{Per1} other than constructing the core metric. It follows immediately from \cite{Per1} that $\#_k \Em_i^n$ will admit a metric of positive Ricci curvature provided each $\Em_i^n$ admits a core metric.\footnote{See the proof of \cite[Theorem B]{BLB1} for further details} This perspective gives us an intrinsic approach to studying Ricci-positive connected sums by asking: \emph{which Riemannian manifolds of positive Ricci curvature admit core metrics?} Not every Ricci-positive Riemannian manifold will admit a core metric. For instance any manifold with nontrivial fundamental group cannot admit core metrics (see \cite[Proposition 2.8]{FL}). In \cite{BLB1}, the author generalized the construction of core metrics found in \cite{Per2} to all complex, quaternionic, and octonionic projective spaces.

One of the chief sources of examples of metrics of positive Ricci curvature comes from fiber bundles. Nash \cite{Nash} showed that the total space of a fiber bundle will admit a metric of positive Ricci curvature provided both base and fiber admit Ricci-positive metrics and that the structure group acts by isometries on this fiber metric. Our main result claims that there is an analogous statement for core metrics if we restrict our fibers to round spheres. 

\begin{customthm}{B}{\emph{\cite[Theorem B]{BLB2}}}\label{iterate} Let $\Ee\rightarrow \Bee^n$ be a rank $m+1\ge4$ vector bundle over $\Bee^n$. If $\Bee^n$ admits a core metric, then $\Sphere( \Ee)$ also admits a core metric. 
\end{customthm}

\noindent Note by \cite[Theorem C]{BLB1} we may take $\Bee^n$ in Theorem \ref{iterate} to be $\Sphere^n$, $\CP^{n/2}$, $\HP^{n/4}$, and $\OP^{(n=16)/8}$. Combining this observation with \cite[Corollary 4.5]{BLB1}, proves Theorem \ref{main}.

%%%%%%%%%%%%%%%%%%%%%%%%
\subsection{Applications}\label{intro:corollaries}
%%%%%%%%%%%%%%%%%%%%%%%%

The original motivation for stating Theorem \ref{iterate} was a question posed by David Wraith whether it is possible to construct a metric of positive Ricci curvature on $\#_k\left( \Sphere^n\times \Sphere^m \times \Sphere^l\right)$. In the author's previous paper \cite[Proposition 5.5]{BLB1}, it was shown that to construct such a metric it is sufficient to construct a core metric for $\Sphere^n\times \Sphere^m$. Theorem \ref{iterate} can be viewed as a generalization of this claim allowing for one factor to be any manifold admitting a core metric and allowing for a nontrivial sphere bundle. When we take trivial bundles in Theorem \ref{main} we see have gone much further than answering this original question.

The following corollary lists the result of Theorem \ref{main} when restrict to products and to base spaces $\Bee^n$ that are spheres or projective spaces. In order to succinctly list all such spaces we must introduce some notation. For a multi-index $\alpha=(\alpha_1,\dots, \alpha_{|\alpha|})$ we define the corresponding product of spheres 
$$\mathbf{S}^\alpha = \prod_{1\le i\le |\alpha| } \Sphere^{\alpha_i} .$$
\noindent The set of $n$-dimensional products that occur in Theorem \ref{main} can be defined as follows.
$$\mathbf{Prod}_n :=\left\{ \Sphere^k \times \mathbf{S}^\alpha, \CP^{k/2} \times \mathbf{S}^\alpha, \HP^{k/4} \times \mathbf{S}^\alpha ,  \OP^{(k=16)/8} \times \mathbf{S}^\alpha    :  k+|\alpha| = n, k\ge 2, \text{ and } \alpha_i\ge 3 \right\} $$

\noindent Note that the empty multi-index is always allowed. Theorem \ref{iterate} implies that every manifold in $\mathbf{Prod}_n$ admits a core metric, and so by \cite[Corollary 4.5]{BLB1} we have the following.  

\begin{customcor}{A}\label{productmain} For any sequence $\Em_i^n\in \mathbf{Prod}_n$, the connected sum $\#_i \Em_i^n$ admits a metric of positive Ricci curvature. 
\end{customcor}
\noindent For example when $n=8$ there is a Ricci-positive metric on the connected sum on any combination of the following manifolds. 
$$ \mathbf{Prod}_8 = \left\{\left( \Sphere^2\times \Sphere^3\times \Sphere^3\right) ,\left( \Sphere^2\times \Sphere^6\right),\left(  \Sphere^3\times \Sphere^5\right), \left(\Sphere^4\times \Sphere^4\right),\left( \CP^2 \times \Sphere^4\right), \left(\CP^3\times \Sphere^2\right), \CP^4, \HP^2 \right\}.$$

\noindent Corollary \ref{productmain} leaves one interesting question about products of spheres unanswered: \emph{is it possible to construct a metric of positive Ricci curvature on the connected sums of product of spheres where arbitrarily many factors are allowed to be $2$-spheres?} By \cite{SY1}, there is a Ricci-positive metric on $\#_k \left(\Sphere^2\times \Sphere^2\right)$, yet we still do not have a technique for constructing Ricci-positive metrics on $\#_k\left( \Sphere^2\times \Sphere^2\times \Sphere^2\right)$.

One particular application of Corollary \ref{productmain} is to the study of the space of all Ricci-positive introduced by Botvinnik-Ebert-Wraith \cite{BEW}. The main tool used in \cite{BEW} is a detection theorem based on the work of the first two named authors and {Randal-Williams} \cite{BER}, that probes the space of positive scalar metrics using the structure of the cohomology ring of the space of stable diffeomorphism. Here ``stable'' refers to taking $\Em_g^{2n}$ as $g\rightarrow \infty$, where $\Em_g^{2n}$ is $\Em^{2n}$ stabilized by taking connected sums with $\Sphere^n\times\Sphere^n$, i.e. 
$$\Em_g^{2n}= \Em^n \# \left( \#_g  \left(\Sphere^n \times \Sphere^n\right)\right).$$
The main result of \cite{BEW} claims that the space of Ricci-positive metrics on $\Em_g^{2n}$ has nontrivial rational homology for sufficiently large $g$ under certain hypotheses on $\Em^{2n}$ and granted that the space of Ricci-positive metrics on $\Em_g^{2n}$ is nonempty. As there is no general result about the stability of Ricci-positivity under connected sum, one needs to have already constructed Ricci-positive metrics on $\Em_g^{2n}$ to apply this result. In particular, Johannes Ebert has asked whether it was known that $\CP^n\#(\#_g \Sphere^n\times \Sphere^n)$ and $\HP^n \# (\#_g \Sphere^{2n}\times \Sphere^{2n})$ admit metrics of positive Ricci curvature. Corollary \ref{productmain} answers this in the affirmative. Combining Theorem \ref{iterate} with \cite[Theorem A]{BEW} we have the following. 

\begin{corollary}\label{rationalhomology} For $ n\not\equiv 3 \mod 4$ and $n\ge 10$, if $\Em^{2n} = \#_i \En_i^{2n}$ is spin and each of the $\En_i^{2n}$ admit core metrics, then $H^j( \mathcal{R}^\text{pRc} (\Em_g^{2n}; \mathbf{Q})$ for some $1\le j\le 5$ for all $g$ sufficiently large. 

In particular we may take $\En_i^{2n}$ to be any of the spin\footnote{$\CP^{2k}$ is not spin, so any manifold $\CP^{2k}\times \mathbf{S}^\alpha\in \mathbf{Prod}_n$ is excluded.} manifolds listed in $\mathbf{Prod}_n$ or any manifold described in Theorem \ref{main}. 
\end{corollary}

In general it is anticipated that the topology of $\mathcal{R}^\text{pRc}(\Em^n)$ is rather complicated provided that it is nonempty. Another result in this direction is \cite[Corollary 1.9]{CSS}, which claims that $\mathcal{R}^\text{pRc}(\Em^n)$ has infinitely many nontrivial homotopy groups provided that $\Em^n$ is spin of dimension $n=8k+8,8k+9$ and $\Em^n$ admits a single metric of positive Ricci curvature. Theorem \ref{main} provides infinitely many new examples of manifolds satisfying these hypotheses.

%%%%%%%%%%%%%%%%%%%%%%%%

\section{Outline}\label{outline}

%%%%%%%%%%%%%%%%%%%%%%%%

%%%%%%%%%%%%%%%%%%%%%%%%

\subsection{The Work of Perelman}\label{outline:perelman}

%%%%%%%%%%%%%%%%%%%%%%%%

Perelman was able to simplify the description of his construction of a Ricci-positive metric on $\#_k \CP^2$ by giving a general argument  that it is possible to glue together two Ricci-positive Riemannian manifolds along isometric boundaries provided that the sum of corresponding principal curvatures of the boundaries is everywhere positive. 

\begin{theorem}{\emph{\cite[Section 4]{Per1}}}\label{glue} Given two Ricci-positive Riemannian manifolds $(\Em_i^n,g_i)$ with isometric boundaries $\Phi: \En_1^{n-1} \rightarrow \En_2^{n-1}$. If $\2_1 +\Phi^* \2_2$ is positive definite where $\2_i$ is the second fundamental form of $\En_i^n$ with respect to the outward unit normal, then there is a Ricci-positive metric $g$ on the smooth manifold $\Em^n = \Em_1^n \cup_\Phi \Em_2^n$ that agrees with the $g_i$ outside of an arbitrarily small neighborhood of $\En^{n-1}$. 
\end{theorem} 

\noindent There is a similar statement to Theorem \ref{glue} for psc Riemannian manifolds with isometric boundaries provided that the sum of the mean curvatures of the boundaries is everywhere positive (see \cite[Section 11.5]{Grom2} or \cite[Corollary 4.8]{ST}).

 Of the ideas presented in \cite{Per1}, Theorem \ref{glue} has attracted the most interest. In \cite{Wang2} it was used to study the topology of Ricci-positive manifolds with convex boundary, in \cite{AMW} it was used to study the space of all Ricci-positive metrics on the $3$-disk with convex boundary, and in \cite{BWW} it was used to study the observer moduli space of Ricci-positive metrics on the sphere. Along with this attention it has attracted detailed proofs; see for example \cite[Appendix 2.3]{Wang1}, \cite[Appendix 5]{AMW}, and \cite[Section 2]{BWW}.

The other key ingredient to the construction of a Ricci-positive metric on $\#_k\CP^2$ in \cite{Per1} is the construction of a particular family of Ricci-positive metrics on the punctured sphere. We refer to these metrics as docking stations, because they are designed with the goal of attaching many cores using Theorem \ref{glue}. 

\begin{theorem}{\emph{\cite[Proposition 1.3]{BLB1} \& \cite{Per1}}} \label{docking} For all $k$ and $1>\nu>0$, there is a Ricci-positive metric $g_\text{docking}(\nu)$ on $\Sphere^n \setminus \bigsqcup_k \Disk^n$ so that each boundary compoment is isometric to $(\Sphere^{n-1},ds_{n-1}^2)$ and the principal curvatures of the boundary are all at least $-\nu$. 
\end{theorem} 
\noindent It is now clear how Theorems \ref{glue} and \ref{docking} reduce the construction of Ricci-positive metric on $\#_i \Em_i^n$ to the construction of core metrics on $\Em_i^n$. Note that the round metric is very far from satisfying the claims of Theorem \ref{docking}. If we delete $k$ geodesic balls of radius $r<\pi/k$ from the round sphere and rescale so that the boundaries have radius $1$, then the principal curvatures will be $- \cos r$. In fact by \cite[Theorem 1]{Wang1}, the round metric cannot accommodate arbitrarily many Ricci-positive manifolds with nontrivial topology in this manner. Thus Theorem \ref{docking} is an essential ingredient to the constructions of \cite{Per1}. 

Not only do Theorems \ref{glue} and \ref{docking} reduce the construction of Ricci-positive connected sums to the construction of core metrics, they will also be instrumental to the construction of the core metrics hypothesized to exist in Theorem \ref{iterate}. This appraoch was anticipated by \cite[Proposition 5.5]{BLB1}, which claimed the existence of a core metric on $\Em^n_i$ implied the existence of a Ricci-positive metric on $\#_k \left( \Em_i^n\times \Sphere^m\right)$. We will explain in Section \ref{outline:top} how the proof of Theorem \ref{iterate} relies on Theorems \ref{glue} and \ref{docking} in the same way as \cite[Proposition 5.5]{BLB1}. This discussion will also make clear that, in order to prove Theorem \ref{iterate} in this way, we will need to consider a version of Theorem \ref{glue} for manifolds with corners, which we state as Theorem \ref{gluecorners} below in Section \ref{outline:corners}. Theorem \ref{gluecorners} represents one of the main technical contributions of this paper.

%%%%%%%%%%%%%%%%%%%%%%%%

\subsection{Topological Decompositions}\label{outline:top}

%%%%%%%%%%%%%%%%%%%%%%%%

In this section we will discuss some elementary topological decompositions that will allow us to reduce Theorem \ref{iterate} to concrete constructions in terms of Theorems \ref{glue} and \ref{docking}. 

As we are interested in constructing a core metric on $\Sphere(\Ee)$ using a hypothesized core metric on $\Bee^n\setminus \Disk^n$, we must consider a decomposition of $\Sphere(\Ee)\setminus \Disk^{n+m}$ in terms of $\Bee^n\setminus \Disk^n$. By trivializing $\Sphere(\Ee)$ over $\Disk^n\hookrightarrow \Bee^n$, we have the following decomposition. 
\begin{equation}\label{trivialized} \Sphere(\Ee) \setminus \Disk^{n+m} =\left[ \Sphere(\Ee)|_{\left( \Bee^n \setminus \Disk^n\right)} \right]\cup_{\left(\Sphere^{n-1} \times \Sphere^m\right)}  \left[ \left(\Disk^n \times \Sphere^m\right) \setminus \Disk^{n+m}\right]
\end{equation}
\noindent Equation (\ref{trivialized}) reduces the topological particulars of the bundle $\Sphere(\Ee)$ entirely to the bundle's restriction over $\Bee^n\setminus \Disk^n$. Using (\ref{trivialized}, we will reduce Theorem \ref{iterate} to the construction of a family of Ricci-positive metrics on $ \left(\Disk^n \times \Sphere^m\right) \setminus \Disk^{n+m}$ after constructing a suitable Ricci-positive metric $g_\text{piece}$ on $\Sphere(\Ee)|_{\Bee^n\setminus \Disk^n}$. 

Before we can describe the metric $g_\text{piece}$, we must mention a deformation that we will use a number of times throughout this paper. In many instances where we want to apply Theorem \ref{glue}, we will need the boundary of one of the Ricci-positive manifolds to be strictly convex. It often happens that the most straightforward construction of Ricci-positive metric will only have weakly convex boundary. Rather than repeat ourselves every time, we make the following observation: that it is always possible to deform the Ricci-positive metric with weakly convex boundary so that the boundary becomes strictly convex, but the boundary metric is unchanged. 

\begin{prop}{\emph{\cite[Proposition 1.2.11]{BLB2}}}\label{wecanwarp} Given a Ricci-positive compact Riemannian manifold $(\Em^{n+1},g)$ with boundary such that $\2\ge 0$, there is a Ricci-positive metric $\tilde{g}$ on $\Em^{n+1}$ so that $\widetilde{\2}>0$ and $g|_{\partial \Em^n} = \tilde{g}|_{\partial \Em^n}$. 
\end{prop}

\noindent  Proposition \ref{wecanwarp} is achieved by making a sufficiently small conformal change supported on a neighborhood of the boundary (see \cite[Proposition 1.2.11]{BLB2} for details). 

We now turn to describing the metric $g_\text{piece}$ on $\Sphere(\Ee)|_{\Bee^n\setminus \Disk^n}$. The goal is to produce a Ricci-positive metric with convex boundary so that it may be glued using Theorem \ref{glue} to a complimentary Ricci-positive metric on $(\Disk^n\times \Sphere^m)\setminus \Disk^{n+m}$ as in (\ref{trivialized}). Using the core metric on $\Bee^n$, hypothesized to exist in Theorem \ref{iterate}, we can use the theory of Riemannian submersions to construct metrics on $\Sphere(\Ee)|_{\Bee^n\setminus \Disk^n}$. In \cite{Nash}, Ricci-positive metrics are constructed in this way using the so-called \emph{canonical variation} \cite[Chapter 9.G]{Besse}. The following lemma summarizes what can be achieved using these same techniques in our situation. 

\begin{lemma}\label{pasta} Let $\Ee\rightarrow \Bee^n$ be a rank $m+1\ge 4$ vector bundle. If $\Bee^n$ admits a core metric, then for $R>1$ sufficiently large there is a Ricci-positive metric $g_\text{piece}$ on $\Sphere(\Ee)|_{ \left(\Bee^n \setminus \Disk^n \right)}$ with convex boundary isometric to $(\Sphere^{n-1}\times \Sphere^m,R^2ds_{n-1}^2 + ds_{m}^2)$. 
\end{lemma}

\begin{proof} Let $\check{g}$ be the hypothesized core metric on $\Bee^n\setminus \Disk^n$. By \cite[Theorem 9.59]{Besse}, for any $t>0$ and any specified horizontal distribution of $\Sphere(\Ee)\rightarrow \Bee^n\setminus \Disk^n$ there is metric $g$ on $\Sphere(\Ee)$ so that $(\Sphere(\Ee),g)\rightarrow (\Bee^n\setminus \Disk^n, \check{g})$ is a Riemannian submersion with totally geodesic fibers isometric to $t^2 ds_m^2$. After trivializing $S(\Ee)$ near the boundary of $\Bee^n\setminus \Disk^n$, we may choose the corresponding flat horizontal distribution on a neighborhood of the boundary and extend this to a horizontal distribution on the entirety of $S(\Ee)|_{\Bee^n\setminus \Disk^n}$. With such a choice of distribution, the metric will be isometric to $\check{g} + t^2ds_m^2$ on a neighborhood of the boundary, and because $\hat{g}$ is assumed to be a core metric, the boundary will be isometric to $ds_{n-1}^2 + t^2 ds_m^2$ and will satisfy $\2\ge 0$. By \cite[Proposition 9.70]{Besse}, if $t$ is taken sufficiently small the metric $g$ will have positive Ricci-curvature because both base and fiber metrics are Ricci-positive. Scaling the resulting metric by $R=1/t$, produces a Ricci-positive metric on $S(\Ee)|_{\left(\Bee^n\setminus \Disk^n\right)}$ with boundary isometric to $R^2 ds_{n-1}^2 + ds_m^2$ that still satisfies $\2\ge0$. Applying Proposition \ref{wecanwarp} produces the desired metric $g_\text{piece}(R)$ on  $S(\Ee)|_{\left(\Bee^n\setminus \Disk^n\right)}$. 
\end{proof} 

In order to describe the complimentary metric on $(\Disk^n\times \Sphere^m)\setminus \Disk^{n+m}$, we extract two important quantities from the construction of Lemma \ref{pasta}: the radius $R>1$ of the larger boundary sphere and $\nu>0$ the infimum of the principal curvatures of the boundary. We have no control over these quantities as they depend on both the arbitrary core metric on $\Bee^n$ and the bundle $\Sphere(\Ee)$. Nonetheless, Lemma \ref{pasta} allows us to trade the topological complexity of the bundle for metric complexity on $\left(\Disk^n\times \Sphere^m\right) \setminus \Disk^{n+m}$, by requiring we construct a family of metrics that can be glued to $g_\text{piece}$ for any such $R>1$ and $\nu>0$ using Theorem \ref{glue}. 

\begin{theorem}\label{transition} Let $n\ge2$, $m\ge 3$, and $R>1$. For all $\nu>0$ sufficiently small there is a Ricci-positive metric $g_\text{transition}(R,\nu)$ on $\left(\Disk^n\times \Sphere^m\right) \setminus \Disk^{n+m}$ such that 
\begin{enumerate}
\item\label{trans:1} The boundary $\Sphere^{n+m-1}$ is round,
\item\label{trans:2} The the boundary $\Sphere^{n+m-1}$ is convex,
\item \label{trans:3} The metric restricted to the boundary $\Sphere^{n-1} \times \Sphere^m$ is isometric to $R^2ds_{n-1}^2 + ds_m^2$,
\item\label{trans:4} The the principal curvatures of the boundary $\Sphere^{n-1} \times \Sphere^m$ are all at least $-\nu$. 
\end{enumerate}
\end{theorem}

\noindent Simply put, $g_\text{transition}(R,\nu)$ is a Ricci-positive metric that transitions from the product of round metrics to a round metric while preserving convexity. Assuming for a moment we have constructed the family of metrics in Theorem \ref{transition}, we can prove our main theorem. 

\begin{proof}[Proof of Theorem \ref{iterate}] By Lemma \ref{pasta} there is a Ricci-positive metric $g_\text{piece}$ on $\Sphere(\Ee)|_{\left( \Bee^n\setminus \Disk^n\right)}$ with boundary isometric to $(\Sphere^{n-1}\times \Sphere^m, R^2ds_{n-1}^2 + ds_m^2)$ and principal curvatures greater than $\nu>0$. By Theorem \ref{glue} there is a Ricci-positive metric $g_\text{transition}(R,\nu)$ on $\Disk^n\setminus \left( \Sphere^{n-1}\times \Disk^{m+1}\right)$ that may be glued to the metric $g_\text{piece}$ along the boundary $\Sphere^{n-1}\times \Sphere^m$ using Theorem \ref{glue}. By equation (\ref{trivialized}), this defines a Ricci-positive metric on $\Sphere(\Ee)\setminus \Disk^{n+m}$ with round and convex boundary. 
\end{proof}

While the desired properties of $g_\text{transition}(R,\nu)$ are easy to describe, the main body of this paper is dedicated to its construction. The approach we take to this construction uses the same idea as used to prove \cite[Proposition 5.5]{BLB1} in the author's previous paper. Morally, the idea is to combine the surgery perspective of \cite{SY2} with the technical metric constructions of \cite{Per1} discussed above in Section \ref{outline:perelman}. We will first explain first how to think of $\left(\Disk^n\times \Sphere^m\right) \setminus \Disk^{n+m}$ as the result performing surgery, which will then allow us to explain how the work of \cite{Per1} can be used in the proof of Theorem \ref{iterate}.

Our first step is to take a slightly different perspective on $\left(\Disk^n\times \Sphere^m\right) \setminus \Disk^{n+m}$. 
\begin{equation}\label{pants} \left( \Disk^n \times \Sphere^m \right) \setminus \Disk^{n+m} \cong \Disk^{n+m} \setminus \left(\Sphere^{n-1} \times \Disk^{m+1}\right) \end{equation}
To see the validity of (\ref{pants}), take the union of both sides of the equation with $\Sphere^{n-1}\times \Disk^{m+1}$ along the common boundary. The lefthand side we have constructed $\Sphere^{n+m}\setminus \Disk^{n+m} \cong \Disk^{n+m}$ using the standard handle decomposition of $\Sphere^{n+m}$, while the righthand side is self-evidently $\Disk^{n+m}$. 

The utility of equation (\ref{pants}) is that it allows us to think of $\Disk^{n+m}\setminus \left(\Sphere^{n-1}\times  \Disk^{m+1}\right)$ as embedded within $\Sphere^{n+m}\setminus \left(\Sphere^{n-1}\times  \Disk^{m+1}\right)$, which in turn allows us to take the surgery perspective of \cite{SY2}. This surgery approach is to think of $\Sphere^{n+m}\setminus \left(\Sphere^{n-1}\times  \Disk^{m+1}\right)$ as the result of performing surgery on $\left( \Sphere^{n-1} \times \Sphere^{m+1}\right) \setminus \left(\Sphere^{n-1}\times  \Disk^{m+1}\right)$.
\begin{equation}\label{handle} \Sphere^{n+m} \setminus \left( \Sphere^{n-1}\times \Disk^{m+1}\right) =  \left[\Disk^n \times \Sphere^m\right]
 \cup_{\left(\Sphere^{n-1}\times \Sphere^m\right)}  \left[\Sphere^{n-1} \times \left( \Sphere^{m+1} \setminus \left( \Disk^{m+1} \sqcup \Disk^{m+1} \right)\right)\right]  .\end{equation}
 Equation (\ref{handle}) follows from the standard handle decomposition of $\Sphere^{n+m}$. The perspective of (\ref{handle}) is useful to prove Theorem \ref{transition} as the boundary of $\Sphere^{n+m}\setminus \left( \Sphere^{n-1}\times \Disk^{m+1}\right)$ will have a product metric provided the metric we begin with on $\left( \Sphere^{n-1} \times \Sphere^{m+1}\right) \setminus \left(\Sphere^{n-1}\times  \Disk^{m+1}\right)$ is a product metric. 
 
Before we return to discuss the metric on $\Disk^{n+m}\setminus \left(\Sphere^{n-1}\times \Disk^{m+1}\right)$, we will explain how to use the work of \cite{Per1} and (\ref{handle}) to construct a metric on $\Sphere^{n+m}\setminus \left( \Sphere^{n-1}\times \Disk^{m+1}\right)$ so that the boundary automatically satisfies (\ref{trans:3}) and (\ref{trans:4}) of Theorem \ref{transition}. By Theorem \ref{docking}, we can consider the Ricci-positive metric $R^2ds_{n-1}^2 + g_\text{docking}(\nu)$ on the second term of the righthand side of (\ref{handle}), so that both boundary components are isometric to $R^2ds_{n-1}^2+ds_m^2$ and the principal curvatures are all at least $-\nu$. We now define a complimentary Ricci-positive metric on the first term on the righthand side of (\ref{handle}). 
\begin{definition}\label{handlemetric} For $R>1$ $f_\nu(x)>0$ be a family of functions with $1>\nu>0$ defined on $[0, R\pi/3]$ such that 
\begin{enumerate}
\item $f_\nu^\text{(odd)}(0)=0$,
\item $f_\nu(0)=1$,
\item $f_\nu'(0)>\nu$,
\item $f_\nu (x)$ converges uniformly to $1$ with respect to $\nu$.
\end{enumerate}
Define \textbf{the handle metric} on $\Disk^{n}\times \Sphere^m$ as the doubly warped product metric 
$$g_\text{handle}(R,\nu) = dx^2 + (2R)^2 \cos^2(x/2R) ds_{n-1}^2 + f_\nu^2(x) ds_m^2.$$
\end{definition}
\noindent Note that the boundary of $\Disk^n\times \Sphere^m$ has principal curvatures all greater than $\nu$ with respect to $g_\text{handle}(R,\nu)$. By construction, $g_\text{handle}(R,\nu)$ converges uniformly to $R^2 ds_n^2 + ds_m^2$ as $\nu\rightarrow 0$ and hence will have positive Ricci curvature for $\nu$ sufficiently small. We may now define the desired metric on $\Sphere^{n+m}\setminus \left(\Sphere^{n-1}\times \Disk^{m+1}\right)$. 

\begin{definition}\label{spheremetric}  For all $R>1$ and $1>\nu>0$ sufficiently small we have two Ricci-positive Riemannian manifolds:
\begin{align*}
(\Em_1^{n+m},g_1) & := \left(\Sphere^{n-1}\times \left(\Sphere^{m+1}\setminus \left( \Disk^{m+1} \sqcup \Disk^{m+1}\right) \right), R^2 ds_{n-1}^2 + g_\text{docking}(\nu)\right)\\
 (\Em_2^{n+m},g_2) & := (\Disk^n \times \Sphere^m, g_\text{handle}(R,\nu)).
 \end{align*}
Applying Theorem \ref{glue} to these manifolds produces a Ricci-positive metric on $\Sphere^{n+m}\setminus \left(\Sphere^{n-1}\times \Disk^{m+1}\right)$ by equation (\ref{handle}). We call this metric \textbf{the sphere metric} and denote it by $g_\text{sphere}(R,\nu)$. 
\end{definition}
\noindent As promised, the Ricci-positive metric $g_\text{sphere}(R,\nu)$ on $\Sphere^{n+m}\setminus \left(\Sphere^{n-1}\times \Disk^{m+1}\right)$ satisfies the boundary conditions (\ref{trans:3}) and (\ref{trans:4}) of Theorem \ref{transition}.

In order to prove Theorem \ref{iterate}, we start by constructing an embedding of $\Disk^{n+m}\setminus \left(\Sphere^{n-1}\times \Disk^{m+1}\right)$ into $\Sphere^{n+m}\setminus \left(\Sphere^{n-1}\times \Disk^{m+1}\right)$, so that (\ref{trans:3}) and (\ref{trans:4}) of Theorem \ref{transition} are immediately satisfied. While we cannot claim that the embedding will satisfy both (\ref{trans:1}) and (\ref{trans:2}) of Theorem \ref{iterate}, we have the following. 

\begin{lemma}\label{disk} There is an embedding $\iota_\text{disk}:\Disk^{n+m}\setminus \left(\Sphere^{n-1}\times \Disk^{m+1}\right) \hookrightarrow \Sphere^{n+m} \setminus \left(\Sphere^{n-1}\times \Disk^{m+1}\right)$ so that the boundary $\Sphere^{n+m-1}$ has nonnegative definite second fundamental form with respect to $g_\text{disk}(R,\nu) = \iota_\text{disk}^*g_\text{sphere}(R,\nu)$. 
\end{lemma}

\noindent The proof of Lemma \ref{disk} carried out in Section \ref{body:ball} turns out to be the main obstacle to proving Theorem \ref{transition}. While the construction of $g_\text{sphere}(R,\nu)$ was easy enough to describe, to fully understand the metric one must understand the proof of Theorems \ref{glue} and \ref{docking} given in \cite{Per1} as well as understand the interaction of these two proofs with an embedded submanifold. In short, it necessitates considering a version of Theorem \ref{glue} for manifolds with corners, which we will discuss as Theorem \ref{gluecorners} below in Section \ref{outline:corners}. 

By applying Proposition \ref{wecanwarp} to the spherical boundary of $(\Disk^{n+m}\setminus \left(\Sphere^{n-1}\times \Disk^{m+1}\right), g_\text{disk}(R,\nu))$, we produce a Ricci-positive metric that satisfies (\ref{trans:2}), (\ref{trans:3}), and (\ref{trans:4}) of Theorem \ref{transition}. The metric $g_\text{disk}(R,\nu)$ restricted to the boundary $\Sphere^{n+m-1}$ will not be round, and so (\ref{trans:1}) is not satisfied. In order to complete the proof of Theorem \ref{transition}, we therefore need a technique for altering Ricci-positive metrics on a collar neighborhood of their boundary while preserving convexity. As far as the author knows, there is no general technique for deformations of this kind. The following theorem claims it is possible to deform the boundary along a Ricci-positive isotopy while preserving ambient Ricci-positivity and boundary convexity. 

\begin{customthm}{C}{\emph{\cite[Corollary 1.3.4]{BLB2}}} \label{bendboundary} Given a Ricci-positive Riemannian manifold $(\Em^{n+1},g)$ with convex boundary isometric to $(\En^n,g_0)$, if $g_0$ and $g_1$ are path connected in the space of Ricci-positive metrics on $\En^n$, then there is a Ricci-positive metric $\tilde{g}$ on $\Em^{n+1}$ with convex boundary isometric to $(\En^n,g_1)$. 
\end{customthm}

The proof of Theorem \ref{bendboundary} is included below in Appendix \ref{bend}. The proof is a variation of a construction in \cite{Per1} (quoted as Proposition \ref{neck} below), which is used in the proof of Theorem \ref{docking}. Our statement of Theorem \ref{bendboundary} is motivated by an analogy between the study of Ricci-positive Riemannian manifolds with convex boundaries and the study of Riemannian manifolds of positive scalar curvature with minimal boundaries. One construction used extensively in the theory of psc manifolds with minimal boundary is the fact that isotopy implies concordance (see \cite[Lemma II.1]{Walsh}). We will reserve our discussion of what is an adequate replacement for ``concordance'' in the theory of Ricci-positive manifolds with convex boundaries for Appendix \ref{bend} below, but the key ingredient to the proof of Theorem \ref{bendboundary} is our version of a Ricci-positive ``isotopy implies concordance'' theorem, namely Theorem \ref{isotopyconcordance} below. 

Our final Lemma claims precisely that Theorem \ref{bendboundary} can be used to modify the metric produced by Lemma \ref{disk} and Proposition \ref{wecanwarp} and prove Theorem \ref{iterate}. 

\begin{lemma}\label{paths} The metric $g_\text{disk}(R,\nu)$ restricted to the boundary $\Sphere^{n+m-1}$ of $\Disk^{n+m}\setminus \left(\Sphere^{n-1}\times \Sphere^{m+1}\right)$ is connected via a path of Ricci-positive metrics to the round metric. 
\end{lemma}

\noindent  The proof of Lemma \ref{paths} below in Section \ref{body:paths} relies on the fact that the metric $g_\text{disk}(R,\nu)$ restricted to the $\Sphere^{n+m-1}$ boundary of $\Disk^{n+m}\setminus \left( \Sphere^{n-1}\times \Disk^{m+1}\right)$ is a doubly warped product metric, and can be given a fairly detailed description in terms of the various choices made in the construction of $g_\text{sphere}(R,\nu)$ aided by Corollary \ref{glueconcave} discussed in Section \ref{outline:corners}.  

Assuming we have established Lemmas \ref{disk} and \ref{paths} and Theorem \ref{bendboundary}, we have completed the proof of Theorem \ref{transition} and hence Theorem \ref{iterate}. 

\begin{proof}[Proof of Theorem \ref{transition}] By Lemma \ref{disk} there is a Ricci-positive metric $g_\text{disk}(R,\nu)$ on $\left(\Disk^n\times \Sphere^m\right)\setminus \Disk^{n+m}$ so that the boundary $\Sphere^{n-1}\times\Sphere^m$ satisfies (\ref{trans:3}) and (\ref{trans:4}) of Theorem \ref{transition} and so that the boundary $\Sphere^{n+m-1}$ has nonnegative principal curvatures. Applying Proposition \ref{wecanwarp} produces a Ricci-positive metric on $\left(\Disk^n\times \Sphere^m\right)\setminus \Disk^{n+m}$ that satisfies (\ref{trans:2}), (\ref{trans:3}), and (\ref{trans:4}) of Theorem \ref{transition} so that the spherical boundary is still isometric to $g_\text{disk}(R,\nu)$. By Lemma \ref{paths}, this metric is connected via a path of Ricci-positive metrics to the round metric. Applying Theorem \ref{bendboundary} produces the desired metric $g_\text{transition}(R,\nu)$. 
\end{proof}

%%%%%%%%%%%%%%%%%%%%%%%%

\subsection{Gluing and Smoothing for Manifolds with Corners}\label{outline:corners}

%%%%%%%%%%%%%%%%%%%%%%%%

 As explained in Section \ref{outline:top}, our proof of Theorem \ref{iterate} amounts to constructing a particular embedding in Lemma \ref{disk} that behaves well with respect to the metric $g_\text{sphere}(R,\nu)$. The implicit difficulty with this statement is that the metric $g_\text{sphere}(R,\nu)$ is constructed using Theorem \ref{glue} and the decomposition (\ref{handle}). The image of $\iota_\text{disk}$ will necessarily intersect the two manifolds with boundary on the righthand side of (\ref{handle}), and so the boundary of the image of $\iota_\text{disk}$ will cross over the neighborhood where Theorem \ref{glue} has been applied. We therefore must revisit Theorem \ref{glue} while also considering the affect it has on the principal curvatures of an embedded hypersurface. 
 
  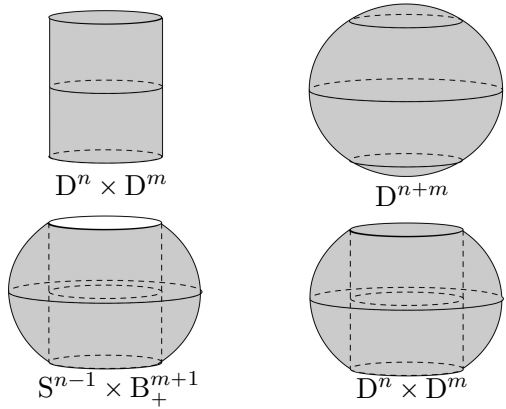
\begin{figure}
\centering
  
 \begin{tikzpicture}[scale=.1]
 \node (img)  {\includegraphics[scale=0.1]{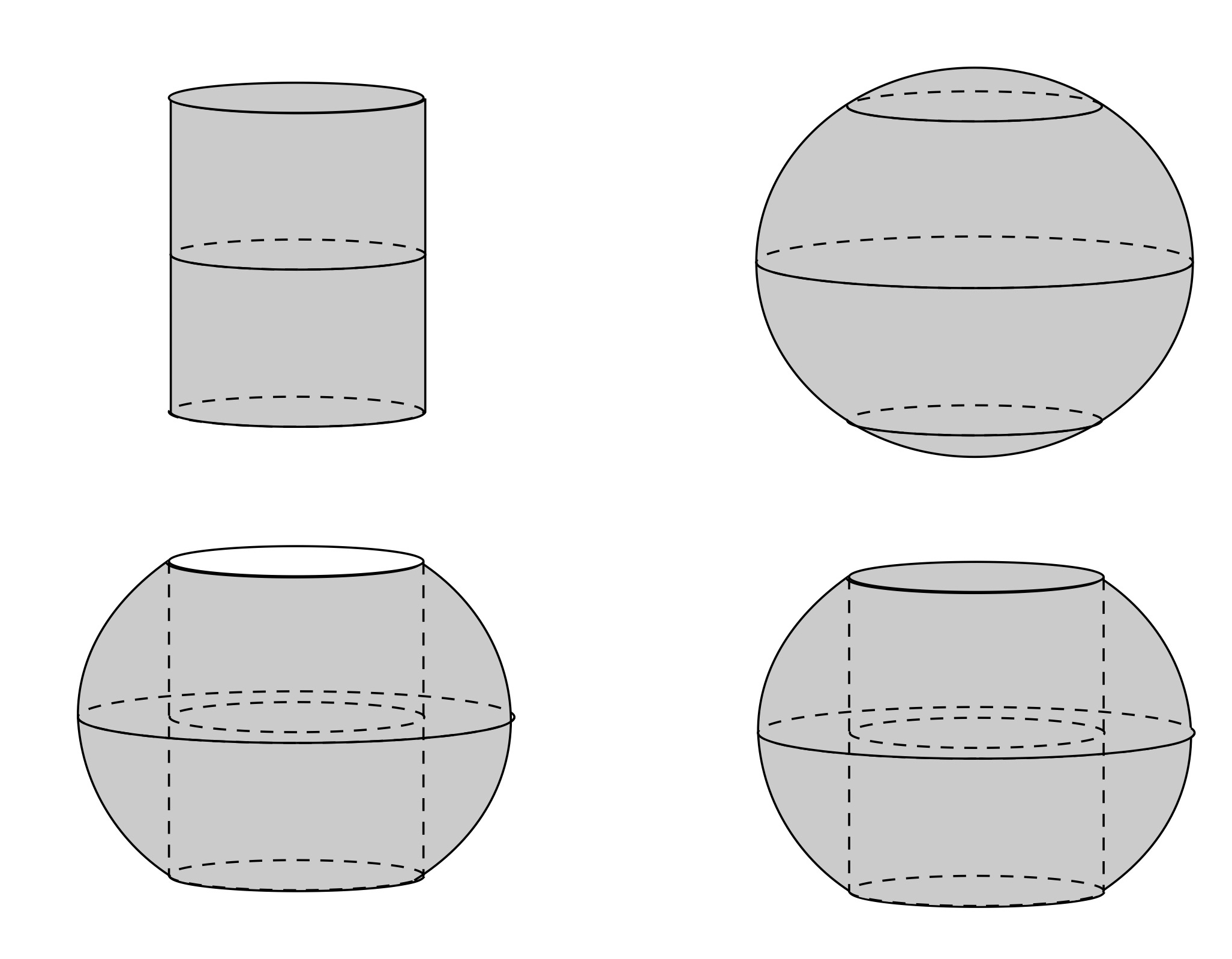}};
\node at (-18,0) {$\Disk^n\times \Disk^m$};
\node at (-17,-27) {$\Sphere^{n-1} \times \Bee_+^{m+1}$};
\node at (22,-1) {$\Disk^{n+m}$};
\node at (22,-27) {$\Disk^n\times \Disk^m$};
 \end{tikzpicture}

 \caption{Gluing two manifolds with corners along a common face can produce a smooth manifold with smooth boundary or a manifold with corners. }
  \label{fig:gluingcorners}
\end{figure}
 
We will take the perspective that the decomposition (\ref{handle}) induces a decomposition of the image of $\iota_\text{disk}$ as a union of two manifolds with corners in the following way:
 \begin{equation}\label{maindecomp} \Disk^{n+m}\setminus \left( \Sphere^{n-1}\times \Disk^{m+1}\right) = \left( \Disk^{n-1} \times \Disk^{m+1} \right) \cup_{\left(\Sphere^{n-1} \times \Disk^{m+1} \right)} \left( \Sphere^{n-1} \times \left( \Bee_+^{m+1} \setminus \Disk^{m+1}\right) \right).
\end{equation}
Where $\Bee_+^{m+1}$ is the smooth upper half-ball: $\Bee_+^{m+1}:= \Bee^{m+1} \cap \left( [0,\infty)\times \mathbf{R}^m\right)$. Figure (\ref{fig:gluingcorners}) illustrates this decomposition of $\Disk^{n+m}$. Note that two terms on the righthand side of (\ref{maindecomp}) are naturally embedded in the corresponding terms on the righthand side of (\ref{handle}). 

Motivated by this example, we wish to state a version of Theorem \ref{glue} for manifolds with corners embedded within manifolds with boundaries. Before phrasing the theorem, we must make a few definitions clear. \emph{A manifold with corners} is a topological space locally modeled on $(-\infty,0]^c \times \mathbf{R}^{n-c}$. The \emph{corners of codimension $c$} are the closure of the connected components of the preimage of $\{0\}^c\times \mathbf{R}^{n-c}$ under these charts. In general, the corners of a manifold with corners will themselves be manifold with corners. The codimension 1 corners are called \emph{faces}, and we call a manifold with corners \emph{a manifold with faces} if each corner of codimension $c$ is the intersection of $c$ distinct faces. 

Suppose we have a fixed decomposition of a closed manifold $\Em^n$ as a union of two manifolds with boundary $\Em^n=\Em_1^n\cup_\partial \Em_2^n$ and that $\Ex^n$ a manifold with boundary is embedded generically within $\Em^n$ with respect to this decomposition. Then $\Ex_i^n:= \Ex^n \cap \Em_i^n$ is a manifold with faces with corners of codimension at most $2$. If $\widetilde{\Wy}^{n-1}_i = \Ex_i^n\cap \partial \Em_i^n$, then $\Ex^n$ is naturally decomposed as a union of two manifolds with faces along a common face $\Ex^n = \Ex_1^n\cup_{\widetilde{\Wy}} \Ex_2^n$. We say that $\Ex_i^n$ is embedded in $\Em_i^n$ with respect to its face $\widetilde{\Wy}^{n-1}_i$ when $\widetilde{\Wy}^{n-1}_i = \Ex_i^n\cap \partial \Em_i^n$ and all other faces of $\Ex_i^n$ intersect $\partial \Em^n$ transversely. 

We now change our perspective from decomposing an existing manifold, to instead considering two manifolds with boundaries $\Em_i^n$ and an abstract orientation reversing diffeomorphism between their boundaries $\Phi:\partial \Em^n_1\rightarrow \partial \Em_2^n$. We may use this diffeomorphism to identify the boundaries and produce a smooth closed manifold $\Em^n=\Em_1^n\cup_\Phi \Em_2^n$ (see \cite[Lemma A.1]{Mil3}). The construction of the smooth structure on $\Em^n$ relies on a choice of collar neighborhood of $\partial \Em_i^n$, but by \cite[Lemma A.2]{Mil3}, any two choices produce diffeomorphic manifolds. This is the perspective taken in the statement of Theorem \ref{glue}, and it is crucial in the proof of Theorem \ref{glue} that the choice of collar neighborhoods are given by normal coordinates of the hypothesized Ricci-positive metrics $g_i$.

 \begin{figure}
\centering
  
 \begin{tikzpicture}[scale=.1]
 \node (img)  {\includegraphics[scale=0.1]{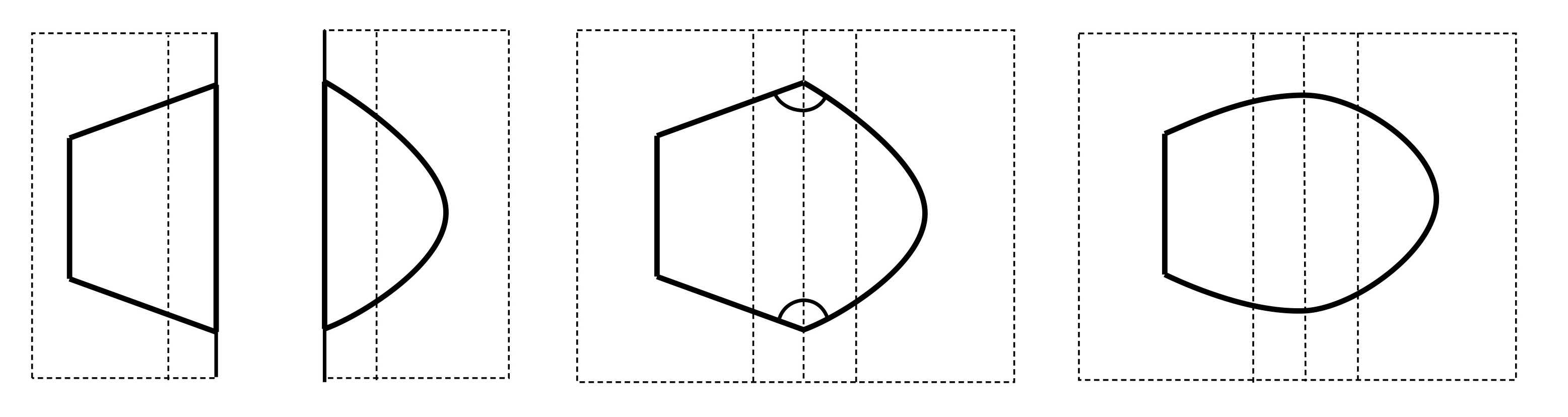}};
\node at (0,-15) {$\Ex_c\hookrightarrow \Em$};
\node at (0,5) {$\theta$};
\node at (0,-5) {$\theta$};
\node at (37,-15) {$\Ex\hookrightarrow \Em$};
\node at (-48,-15) {$\Ex_1\hookrightarrow \Em_1$};
\node at (-25,-15) {$\Ex_2\hookrightarrow \Em_2$};
\node at (-49,10) {$\Wy_1$};
\draw[->] (-50,8) -- (-48,-5);
\draw[->] (-50,8) -- (-46,7);
\node at (-25,10) {$\Wy_2$};
\draw[->] (-25,8) -- (-25,3);
\node at (-37,10) {$\widetilde{\Wy}_1$};
\draw[->] (-38,8) -- (-39.5,3);
\node at (-35,-10) {$\widetilde{\Wy}_2$};
\draw[->] (-34,-8) -- (-32.5,-3);
 \end{tikzpicture}

 \caption{An illustration of the manifold with corners $\Ex_c^n$ that is formed by a generic identification of two manifolds with boundaries and its relationship to $\Ex^n$.}
  \label{fig:cornersdefinitions}
\end{figure}
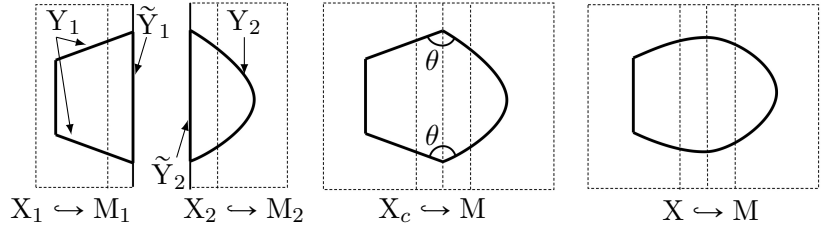

When we add to this scenario embeddings of manifolds with faces $\Ex_i^n\hookrightarrow \Em_i^n$ relative to their face $\widetilde{\Wy}_i^{n-1}$ and assume that the diffeomorphism of the boundaries restricts to a diffeomorphism $\Phi:\widetilde{\Wy}_1^{n-1}\rightarrow \widetilde{\Wy}_2^{n-1}$ of manifolds with boundary, we can ask if the gluing procedure of \cite[Lemma A.1]{Mil3} will produce the desired manifold with boundary $\Ex^n$. Consider one of the boundary components $\Zee^{n-2}_i$ of $\widetilde{\Wy}_i^{n-1}$ and the second face $\Wy^{n-1}_i$ such that $\Wy^{n-1}_i\cap \widetilde{\Wy}_i^{n-1} = \Zee_i^{n-2}$. In this case, when we identify the boundaries of $\Em_i^n$, the two faces $\Wy_i^{n-1}$ will generically intersect transversely in their common boundary $\Zee_1^{n-2}\cong \Zee_2^{n-2}$. The resulting space we should call $\Ex_c^n = \Ex_1^n \cup_{\Phi} \Ex_2^n$ as it will have corners along the gluing site as illustrated in Figure (\ref{fig:cornersdefinitions}). For an example consider Figure (\ref{fig:gluingcorners}), if we glues together the manifolds with faces $ \Disk^n\times \Disk^m$ and $\Sphere^{n-1}\times \Bee_+^{m+1}$ using arbitrary coordinates we will almost always produce the manifold with faces $\Disk^n\times \Disk^m$ rather than the desired manifold with boundary $\Disk^{n+m}$. 

We say that the smooth manifold with faces $(\Ex_c^\prime)^n$ is \emph{the result of smoothing the corner} $\Zee^{n-2}$ if it is diffeomorphic as manifold with faces to $\Ex_c^n$ outside of $\Zee^{n-2}$ such that the two faces $\Wy_i^{n-1}$ are replaced by a single face $\Wy^{n-1} = \Wy_1^{n-1}\cup_{\Zee^{n-2}} \Wy_2^{n-1}$. By \cite[Lemma A.1]{Mil3} it is always possible to smooth a corner, and by \cite[Lemma A.2]{Mil3} any two smoothings are diffeomorphic. This uniqueness statement ensures that $\Ex^n$ can be recovered from $\Ex_c^n$ by smoothing each of the corners. For instance, one can recover $\Disk^{n+m}$ from $\Disk^n\times \Disk^m$ by smoothing the corner $\Sphere^{n-1}\times \Sphere^{m-1}$. This will be the perspective we take in our version of Theorem \ref{glue}, that we can obtain an embedding of $\Ex^n\hookrightarrow \Em^n$ by smoothing the corners of $\Ex_c^n$ introduced along the gluing site. 

We are now ready to phrase our version of Theorem \ref{glue} for manifolds with corners embedded in manifolds with boundary. Keep in mind that our motivation is to produce Ricci-positive metrics on manifolds with convex boundary. 

\begin{customthm}{D}\label{gluecorners} Assume that $(\Em_i^n,g_i)$ are two Riemannian manifolds that satisfy the hypotheses of Theorem \ref{glue}. Suppose there are two manifolds with faces $\Ex_i^n$ with corners of codimension at most 2 embedded in $\Em_i^n$ relative to their face $\widetilde{\Wy}_i^{n-1}$ so that the isometry $\Phi:(\En^{n-1}_1,g_1)\rightarrow (\En_2^{n-1},g_2)$ restricts to an isometry $\Phi: (\widetilde{\Wy}_1^{n-1},g_1)\rightarrow (\widetilde{\Wy}_2^{n-1},g_2)$ of manifolds with boundaries. If the remaining faces of $\Ex_i^{n-1}$ are convex with respect to $g_i$ and if the dihedral angles of $\Ex_c^n$ made along the gluing site are less than $\pi$, then there is an embedding of the smooth manifold with boundary $\Ex^n\hookrightarrow \Em^n=\Em_1^n\cup_\Phi \Em_2^n$ that agrees with the embeddings $\Ex_i^n\hookrightarrow \Em_i^n$ outside of a small neighborhood of $\widetilde{\Wy}^{n-1}$ such that the boundary of $\Ex^n$ is convex with respect to the Ricci-positive metric of Theorem \ref{glue}. 
\end{customthm}
\noindent The additional hypothesis about the dihedral angles of $\Ex_c^n$ should be viewed in a similar light to the hypothesis of Theorem \ref{glue} about the second fundamental forms. We will see that both hypotheses play the same role respectively in the proof of the claim about boundary convexity and Ricci-positivity. 

The proof of Lemma \ref{disk} presented below in Section \ref{body:docking} will follow from an application of Theorem \ref{gluecorners} after specifying embeddings of each term on the righthand side of (\ref{maindecomp}) into the corresponding terms on the righthand side of (\ref{handle}). We state Theorem \ref{gluecorners} in this relative way precisely with this application in mind. We note however a nearly identical statement can be given without referencing an ambient manifold with boundary (see \cite[Theorem I]{BLB2} for that statement). We have also phrased Theorem \ref{gluecorners} in this way so we do not have to retread existing proofs of the conclusions of Theorem \ref{glue} in the case of manifolds with corners. As explained, after applying Theorem \ref{glue} we are left with an embedding of the manifold with faces $\Ex_c^n$ in $\Em^n=\Em_1^n\cup_\Phi \Em_2^n$. The approach we will take to proving Theorem \ref{gluecorners} below in Section \ref{corners} is to smooth each of the corners of the image of $\Ex_c^n$ introduced along the gluing site.

While Theorem \ref{gluecorners} is sufficient to prove Lemma \ref{disk}, we still need to give a fairly explicit description of the resulting metric $g_\text{disk}(R,\nu)$ restricted to the $\Sphere^{n+m-1}$ boundary of $\Disk^{n+m}\setminus \left(\Sphere^{n-1}\times \Disk^{m+1}\right)$ if we hope to prove Lemma \ref{paths}. We return now to the general notation of Theorem \ref{gluecorners}. Let $g$ be the smooth metric on $\Em^n$ produced by Theorem \ref{glue} and let $\Wy^{n-1} = \Wy_1^{n-1}\cup_\Phi \Wy_2^{n-1}$ be the smooth face of $\Ex^n$. We are precisely interested in studying the behavior of $g|_{\Wy}$. As $g|_{\Wy}$ agrees with $(g_i)|_{\Wy_i}$ away from the gluing site, it suffices to consider the behavior near the gluing site. Let $\Zee^{n-2}$ denote the corner of $\Ex_c^n$ where $\Wy_1^{n-1}$ and $\Wy_2^{n-1}$ have been smoothly identified. We may take $\partial_s$ to be the normal coordinate of $\Zee^{n-2}$ with respect to $g|_{\Wy}$ so that
$$ g|_{\Wy} = ds^2 + h(s),$$
for some smooth one parameter family of metrics $h(s)$ on $\Zee^{n-2}$. 

In Lemma \ref{paths}, we are implicit claiming that $g|_{\Wy}$ has positive Ricci curvature. Taking normal coordinates for $\Zee^{n-2}$ discussed above we have, from \cite[3.2.11]{Pet} we have
$$\K_{g|_{\Wy}}(\partial_s,-)= -(1/2) h''(s) + ((1/2)h'(s))^2.$$
Hence the assumption that $h''(s)<0$ will contribute a great deal to verifying that $g|_{\Wy}$ has positive Ricci curvature. In our situation, $g|_{\Wy}$ will be a doubly warped product metric on the sphere. In this special case, the assumption that $h''(s)<0$ will essentially guarantee that $g|_{\Wy}$ has positive Ricci curvature. It is for this reason that we will need a version of Theorem \ref{gluecorners} that also considers $h''(s)<0$.

\begin{customcor}{D}\label{glueconcave} Let everything be as in Theorem \ref{gluecorners}. Assume additional that if $(g_i)|_{\Wy_i} = ds^2+h_i(s)$ with respect to normal coordinates of $\Zee^{n-2}_i$, then $h_i''(s)<0$. Let $g_{\Wy}$ be the metric $g$ of Theorem \ref{glue} restricted to the smooth boundary $\Wy^{n-1}$ of $\Ex^n$ embedded within $\Em^n$ as in Theorem \ref{gluecorners}. If $g|_{\Wy} = ds^2 +h(s)$ with respect to the normal coordinates of $\Zee^{n-2}$, then $h''(s)<0$. 
\end{customcor}
\noindent The proof of Lemma \ref{paths} presented below in Section \ref{body:paths} will rely on the additional conclusion of Corollary \ref{glueconcave} and the fact that the metric $g|_{\Wy}$ constructed using Theorem \ref{gluecorners} is a doubly warped product metric on the sphere. The proof of Corollary \ref{glueconcave} will follow directly from the proof of Theorem \ref{gluecorners} after computing $h''(s)$ in particular coordinates carried out in Section \ref{corners:metric}. 

%%%%%%%%%%%

\subsection{Organization of Paper}\label{outline:outline} 

In section \ref{outline:top}, we have described how Theorem \ref{iterate} can be reduced to the proof of Theorem \ref{transition}, and in turn how Theorem \ref{transition} can be reduced to the proof of Theorem \ref{bendboundary}, Theorem \ref{gluecorners}, Corollary \ref{glueconcave}, Lemma \ref{disk}, Lemma \ref{paths}. As Theorems \ref{bendboundary} and \ref{gluecorners} are more general than the application we have in mind for them, we leave their proofs to two technical appendices: Appendices \ref{bend} and \ref{corners} respectively. The main body of this paper, Section \ref{body}, is dedicated to proof of Lemmas \ref{disk} and \ref{paths}, which are proven respectively in Sections \ref{body:docking} and \ref{body:paths}. As promised the proof of Lemma \ref{disk} relies on a careful analysis of Theorem \ref{docking} and Corollary \ref{glueconcave}.

%%%%%%%%%%%%%%%%%%%%%%%%

\section{The Transition Metric}\label{body}

%%%%%%%%%%%%%%%%%%%%%%%%

%%%%%%%%%%%%%%%%%%%%%%%%

\subsection{The Disk in the Sphere}\label{body:docking}

%%%%%%%%%%%%%%%%%%%%%%%%

As discussed in Section \ref{outline}, in order to prove Lemma \ref{disk} we will define an embedding of each of the terms on the righthand side of (\ref{maindecomp}) into the corresponding terms of (\ref{handle}) in such a way that we may be glue the two embeddings together to produce the desired embedding of the disk into the sphere. We start out in Section \ref{body:ball} building an embedding 
\begin{equation}\label{iotadocking} \iota_\text{docking}: \Bee^{m+1}_+\setminus \Disk^{m+1}  \hookrightarrow \Sphere^{m+1} \setminus \left( \Disk^{m+1}\sqcup \Disk^{m+1}\right),\end{equation}
so that the interior face of $\Bee^{m+1}_+\setminus \Disk^{m+1}$ has nonnegative principal curvatures with respect to $g_\text{docking}(\nu)$. This itself is a very difficult task as we will need to delve into some of the details of \cite{Per1}. Then in Section \ref{body:handle} we will described an embedding 
\begin{equation}\label{iotahandle} \iota_\text{handle}: \Disk^n\times \Disk^m \hookrightarrow \Disk^n\times \Sphere^m,\end{equation}
so that the interior face of $\Disk^n\times \Disk^m$ will have nonnegative principal curvatures with respect to $g_\text{handle}(R,\nu)$. We conclude by proving Lemma \ref{disk}. 

%%%%%%%%%%%%%%%%%%%%%%%%

\subsubsection{The Half-Ball in the Docking Station}\label{body:ball}

%%%%%%%%%%%%%%%%%%%%%%%%

In this section we will describe the particular embedding (\ref{iotadocking}). First we must give some details of the proof of Theorem \ref{docking} given in \cite{Per1}. The metric $g_\text{docking}(\nu)$ is itself built from a decomposition of $\Em^{m+1} = \Sphere^{m+1} \setminus \left(\bigsqcup_k \Disk^{m+1}\right)$ as $\Em^{m+1}=\Em_1^{m+1}\cup_{\En^m} \Em_2^{m+1}$, where
$$ \Em_1^{m+1} = \bigsqcup_k \left( [0,1]\times \Sphere^m\right),\quad  \Em_2^{m+1}= \Sphere^{m+1} \setminus \left(\bigsqcup_k \Disk^{m+1}\right)\text{, and } \En_1^m=\En_2^m= \bigsqcup_k \Sphere^m.$$
This decomposition, pictured in Figure \ref{fig:decomp}, can be thought of as removing a collar neighborhood of the boundary of $\Em^{m+1}$. We will now describe the family of Ricci-positive metrics on both $\Em_1^{m+1}$ and $\Em_2^{m+1}$ constructed in \cite{Per1} used in the construction of $g_\text{docking}(\nu)$. 

\begin{prop}\emph{\cite[Section 3]{Per1}}\label{ambient} For $m\ge 3$ and any $0<r<\pi/2$ and $0<w$ sufficiently smalll there exists an $R_0>0$ and a function $R:[0,\pi/2]\rightarrow [0,R_0]$ so that
 $$g_\text{ambient}(r,w) = \cot^2 r \left( dt^2 + \cos^2( t) dx^2 + R^2(t) ds_{m-1}^2\right),$$
is a metric defined on $\Sphere^{m+1}$ with $t\in[0,\pi/2]$ and $x\in[-\pi,\pi ]$ such that : 
\begin{enumerate}
\item \label{ambient:curvature} the sectional curvatures of $g_\text{ambient}(r,w)$ are all positive,
\item \label{ambient:boundary} the boundary of a geodesic ball of radius $r$ centered along the subspace $t=0$ is isometric to $A_1(x)dx^2 +  w^2 \cos^2( x) ds_{m-1}^2$ for $x\in [-\pi/2,\pi/2]$, 
\item \label{ambient:boundarycurvature} the sectional curvature of the boundary of such a geodesic ball are all greater than 1, 
\item \label{ambient:principalcurvatures} the principal curvatures of the boundary of such a geodesic ball are all less than 1. 
\end{enumerate}
\end{prop}

\noindent The main difficulty in proving Theorem \ref{docking} is the construction of \emph{the neck} metric, which is a Ricci-positive transition from the boundary metric (\ref{ambient:boundary}) of Proposition \ref{ambient} to the round metric.

\begin{prop}\emph{\cite[Assertion]{Per1}}\label{neck} For $m\ge 3$, $0<r<\pi/2$, $0<w$ sufficiently small and any $\rho>0$ such that $\pi \rho < r$ and $w^{n-1}< \rho^n$, there exists an $l>0$, a function $A: [0,l]\times [-\pi/2,\pi/2]\rightarrow \mathbf{R}_+$, and a function $B:[0,l]\rightarrow \mathbf{R}_+$ so that the metric
$$g_\text{neck} (r,w,\rho) =  dt^2 + A^2(t,x) dx^2 + B^2(t) \cos^2(x) ds_{m-1}^2,$$
is defined on $[0,l]\times \Sphere^m$ with $t\in[0,l]$ and $x\in [-\pi/2,\pi/2]$ such that:
\begin{enumerate}
\item \label{neck:curvature} $g_\text{neck}(r,w,\rho)$ has positive Ricci curvature,
\item \label{neck:00} $A(0,x) = B(0) = \rho/\lambda$, 
\item \label{neck:01} the principal curvatures at $t=0$ are all $-\lambda$,
\item \label{neck:10} $A(l,x) = A_1(x)$ as in (\ref{ambient:boundary}) of Proposition \ref{ambient} and $B(l)= w$,
\item \label{neck:11} the principal curvatures at $t=l$ are all greater than $1$, 
\item \label{neck:Bproperties} the function $B(t)$ is increasing and concave down. 
\end{enumerate}
\end{prop}

\noindent We can now explain the construction of $g_\text{docking}(\nu)$. 
\begin{proof}[Proof of Theorem \ref{docking}] For each $\nu>0$ and $k$, fix $\rho>0$ such that $\rho<\nu$ and $\rho< \frac{1}{k}$. We may then fix a choice of $r< \frac{\pi}{k}$ for which $\rho < r$, and similarly we can fix a $w>0$ for which $w^{n-1}<\rho^n$. Now set let $g_1= g_\text{neck}(r,w,\rho)$ and $g_2$ be the restriction of $g_\text{ambient}(r,w)$ to $\Em_2^{m+1}$ given by deleting $k$ disjoint geodesic balls of radius $r$ from $\Sphere^{m+1}$. We claim that $(\Em_1^{m+1},g_1)$ and $(\Em_2^{m+1},g_2)$ satisfy the hypotheses of Theorem \ref{glue}. By (\ref{ambient:curvature}) of Proposition \ref{ambient} and (\ref{neck:curvature}) of Proposition \ref{neck}, both of these Riemannian manifolds have positive Ricci curvature. By (\ref{ambient:boundary}) of Proposition \ref{ambient} and (\ref{neck:10}) of Proposition \ref{neck} there is an isometry of their boundaries $\Phi: (\En_1^m,g_1)\rightarrow (\En_2^m,g_2)$. By (\ref{ambient:principalcurvatures}) and (\ref{neck:11}), we have $\2_1+\Phi^* \2_2 > 1 + (-1) =0$.  By applying Theorem \ref{glue} we have a Ricci-positive metric on $\Em^{m+1}$ with boundary isometric to the remaining boundary of $(\Em_1^{m+1},g_1)$. After scaling this metric by $\lambda/ \rho$, the boundary spheres are all isometric to the round unit sphere by (\ref{neck:00}) of Proposition \ref{neck} with principal curvatures greater than $-\rho>-\nu$ by (\ref{neck:01}) of Proposition \ref{neck}. We call this metric $g_\text{docking}(\nu)$. 
\end{proof}

 \begin{figure}
\centering

\begin{subfigure}{.5\textwidth}
  \centering
  
 \begin{tikzpicture}[scale=.08]
 \node (img)  {\includegraphics[scale=0.1]{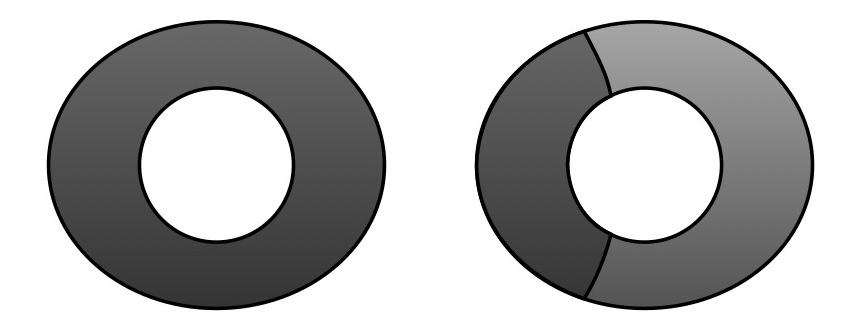}};
 \end{tikzpicture}
 
 \caption{The necks $\Em_1^{m+1}$. The manifold with faces $\Ex_1^{m+1}$ is embedded as the dark region.}
  \label{fig:eneck}
  
 \end{subfigure}%
\begin{subfigure}{.5\textwidth}
  \centering
  
 \begin{tikzpicture}[scale=.08]
 \node (img)  {\includegraphics[scale=0.1]{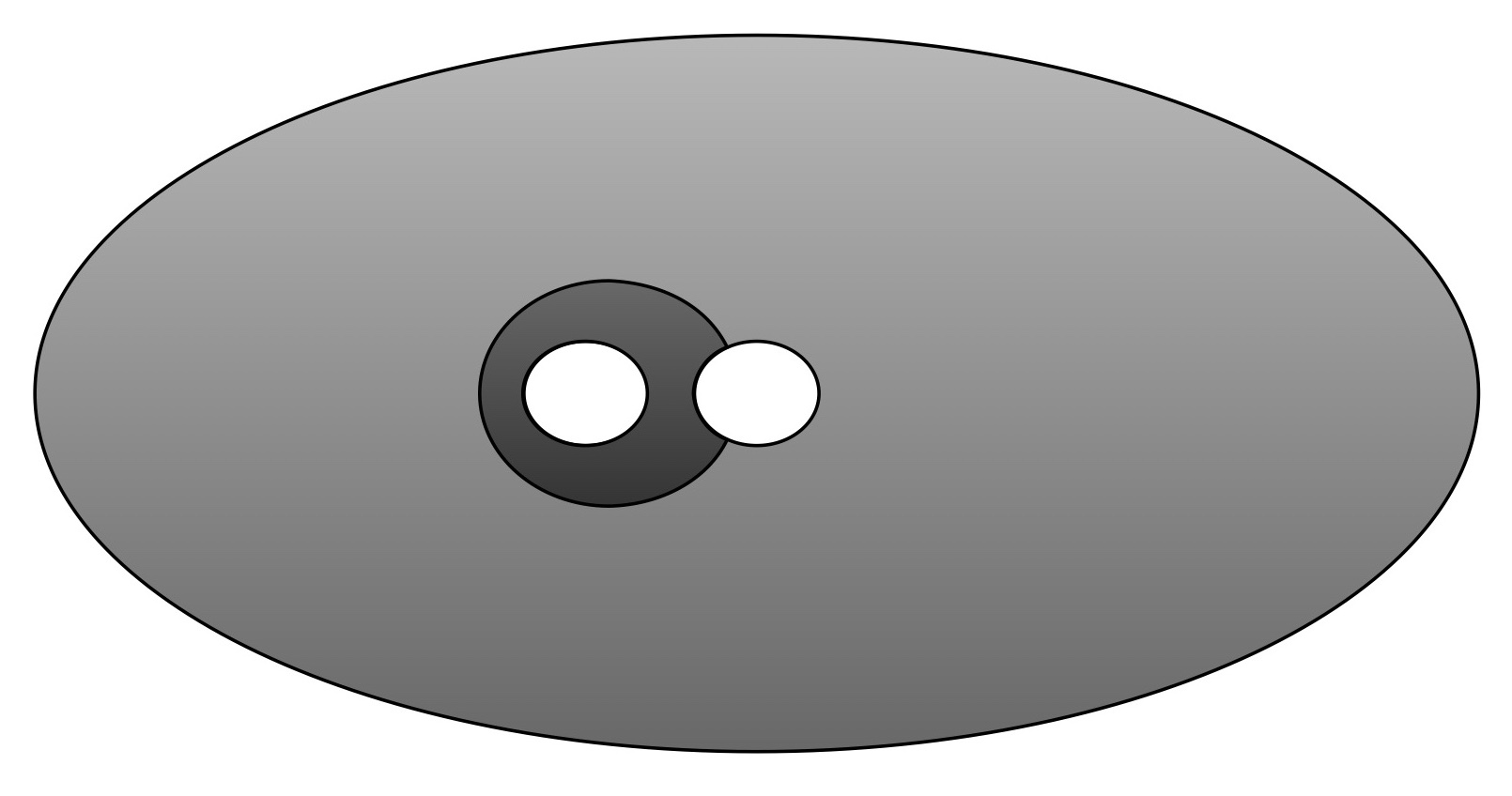}};
 \end{tikzpicture}
 
 \caption{The ambient space $\Em_2^{m+1}$. The manifold with faces $\Ex_2^{m+1}$ is embedded as the dark region.}
  \label{fig:eambient}
\end{subfigure}

 \caption{The decomposition of the docking station $\Em^{m+1}$, and the corresponding decomposition of $\Ex^{m+1}$. }
  \label{fig:decomp}
\end{figure}

As we are attempting to use (\ref{handle}) to prove Theorem \ref{transition}, we are only interested in the case $k=2$ in the above construction. We will assume in addition that $r< \pi/4$ and that the two geodesic balls deleted lie entirely within the hemisphere of $(\Sphere^{m+1},g_\text{ambient}(r,w))$ corresponding to $-\pi/2 < x < \pi/2$. At this point we must introduce a wrinkle to the outline of the proof of Theorem \ref{iterate} given in Section \ref{outline}. Rather than applying Corollary \ref{glueconcave} to the decomposition (\ref{maindecomp}), we will only need to apply Corollary \ref{glueconcave} to the decomposition of $\Ex^{m+1}=\Bee_+^{m+1}$ induced by the decomposition $\Em^{m+1} = \Em_1^{m+1}\cup_{\En^m}\Em_2^{m+1}$ in order to describe the embedding $\iota_\text{docking}:\Ex^{m+1}\hookrightarrow \Em^{m+1}$. As discussed in Section \ref{outline:corners}, this decomposes $\Ex^{m+1} = \Ex_1^{m+1}\cup_{\widetilde{\Wy}^m} \Ex_2^{m+1}$, where in the notation of Section \ref{outline:corners} we have:
\begin{align*}
\Ex_1^{m+1} & =  \left( [0, 1]\times \Disk^m \right) \sqcup  \left( [0, 1]\times \Sphere^m \right)   &\Ex_2^{m+1}  &= \Bee_+^{m+1}\setminus \Disk^{m+1}\\
\widetilde{\Wy}^m_1  &= \left( \{1\}\times \Disk^m\right) \sqcup \left( \{1\}\times \Sphere^m\right) & \widetilde{\Wy}^m_2  &= \Disk^m \sqcup \Sphere^m \\
\Wy^m_1  &= [0,1]\times \Sphere^{m-1} & \Wy_2^m&  = \Disk^m. 
\end{align*}

Our goal now is to produce two embeddings of manifolds with faces
\begin{align}
\label{iotaneck} \iota_\text{neck} & : \Ex_1^{m+1} \hookrightarrow \Em_1^{m+1},\\
\label{iotaambient} \iota_\text{ambient} &: \Ex_2^{m+1}\hookrightarrow \Em_2^{m+1},
\end{align}
relative to their faces $\widetilde{\Wy}_i^m$, to which we may apply Corollary \ref{glueconcave} to produce the embedding $\iota_\text{docking}$. We have pictured the image of (\ref{iotaneck}) in Figure \ref{fig:eneck} and the image of (\ref{iotaambient}) in Figure \ref{fig:eambient} as the darker regions. 

\begin{lemma}\label{eambient} There is an embedding (\ref{iotaambient}) such that:
\begin{enumerate}
\item\label{eambient:in0} the metric $g_2$ restricted to $\Wy^m_2$ takes the form $ds^2 + k_2^2(s)ds_m^2$, with $s\in[S_1,S_2]$ where $k_2''(s)<0$,
\item \label{eambient:in1} the principal curvatures of $\Wy^m_2$ are positive with respect to $g_2$,
\item \label{eambient:out0} $\widetilde{\Wy}^m_2$ corresponds to the portion of the boundary of a geodesic ball made up from those geodesics that make an angle $0< \theta_0\le \theta \le \pi/2$ with the $x$-axis, 
\item \label{eambient:angles} the dihedral angle made between $\widetilde{\Wy}^m_2$ and $\Wy^m_2$ is everywhere less than $\pi/2$.
\end{enumerate}
\end{lemma}
\begin{proof} Assume that the two deleted geodesic balls of $(\Sphere^{m+1},g_2)$ with radius $r$ are centered at $(0,0)$ and $(0,x_2)$ in the $(t,x)$-plane.  Consider a geodesic ball of radius $2r$ centered at $(0,x_1)$. If $r<x_1<2r$, then the boundary of this geodesic ball will intersect the boundary of the geodesic ball centered at $(0,0)$. If $|x_0-x_1| < \delta$, then the geodesic ball centered at $(0,x_2)$ will lie in the interior of this geodesic ball. Assuming this is the case, after removing the two geodesic balls of radius $r$, the interior of the geodesic ball of radius $2r$ describes an embedding $\Ex_2^{m+1}\hookrightarrow \Em_2^{m+1}$. We will show for an appropriately chosen $x_1$, $x_2$, that this embedding satisfies (\ref{eambient:in0})-(\ref{eambient:angles}).

 \begin{figure}
\centering

 \begin{tikzpicture}[scale=.1]
 \node (img)  {\includegraphics[scale=0.1]{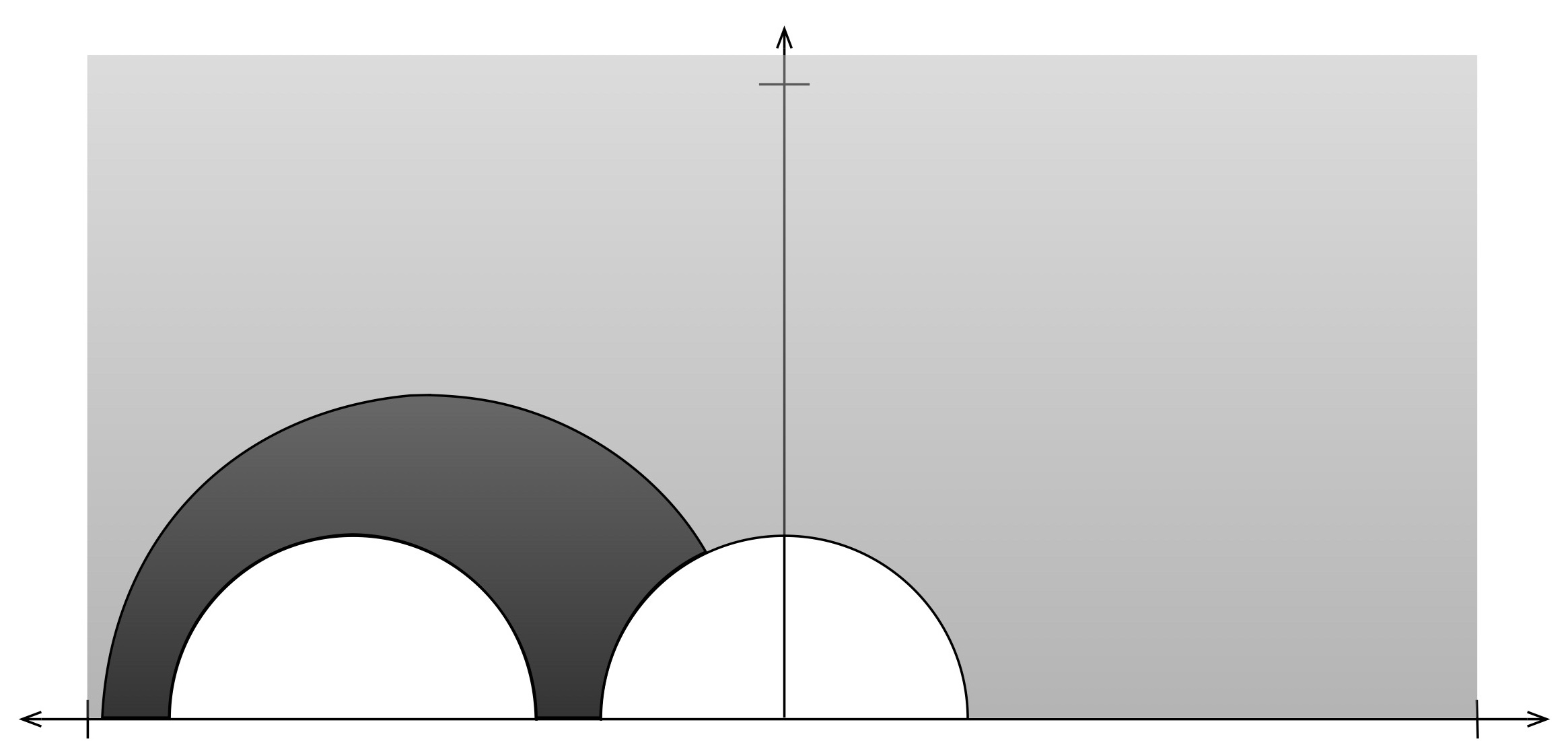}};
 \node at (-40,-20) {$x$};
\node at (-2,18) {$t$};
\node at (-36,-21) {$\frac{\pi}{2}$};
\node at (36,-21) {$-\frac{\pi}{2}$};
\node at (3,15) {$\frac{\pi}{2}$};
 \end{tikzpicture}

 \caption{The compliment of the two geodesic balls of radius $r$ is lightly shaded. This region corresopnds to $\Em_2^{m+1}$ in Figure (\ref{fig:eambient}). The remainder of the geodesic ball of radius $2r$ is shaded dark. This region corresponds to $\Ex_2^{m+1}$, which is also shaded dark in (\ref{fig:eambient}).}
  \label{fig:geodesicballs}
\end{figure}
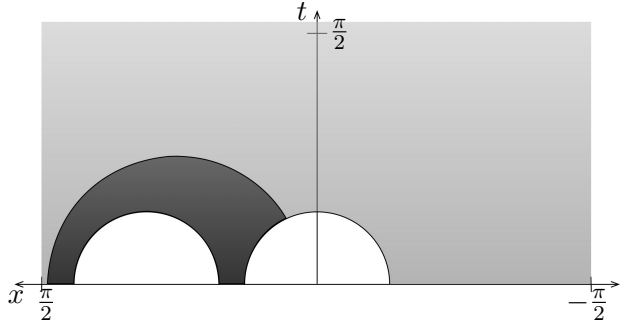

Note that (\ref{eambient:in0}) follows directly from the form of the metric $g_\text{ambient}(r,w)$ and (\ref{ambient:boundarycurvature}) in Proposition \ref{ambient}.

We will next consider (\ref{eambient:in1}). Let $\2_2$ denote the second fundamental form of $\Wy^m_2$ with respect to $g_\text{ambient}(r,w)$. Let $(t,x) = (\gamma_1(s),\gamma_2(s))$ be the arc-length parameterization of the geodesic ball in the $(t,x)$-plane. Note that $g_\text{ambient}(r,w)$ restricted to the $(t,x)$-plane is $\cot r ( dt^2 + \cos^2 t dx^2)$, which is just the round metric of radius $\cot^2 r $. Thus $\2_2(\partial_s,\partial_s)$ must agree with the principal curvatures of a geodesic ball with respect to the round metric. It is a standard computation that the principal curvatures of geodesic ball of radius $2r$ are $\cot(2r)/ \cot r$ and hence positive for $r < \pi/4$. 

To compute the remaining principal curvatures. Let $\sigma$ denote normal coordinate frame of $\Sphere^{m-1}$, so that $\sigma/R(t)$ is a unit length frame with respect to $g_\text{ambient}(r,w)$. In order to compute $\2_2(\sigma,\sigma)$, the relevant Christoffel symbols needed are
$$ \Gamma_{\sigma\sigma}^t = -R'(t) R(t) \text{ and } \Gamma_{\sigma\sigma}^x = 0.$$
The unit normal vector to the hypersurface $\{(\gamma_1(s),\gamma_2(s))\}\times \Sphere^m$ must lie in the $(t,x)$-plane and can be written as
$$N = \dfrac{ \cos^2(\gamma_1(s))\gamma_2'(s) \partial_t - \gamma_1'(s) \partial_x}{\cos(\gamma_1(s)) \sqrt{1+\sin^2(t) (\gamma_2'(s))^2}}.$$
We can now compute directly that
$$\2_2(\sigma/R,\sigma/R)= -\frac{1}{R^2}g(\nabla_{\sigma} \sigma,N) = - \frac{\Gamma_{\sigma\sigma}^t }{R^2}g( \partial_t ,N) = \dfrac{\cos(\gamma_1(s)) R'(\gamma_1(s))\gamma_2'(s)}{R(\gamma_1(s)) \sqrt{1+\sin^2(\gamma_1(s)) (\gamma_2'(s))^2}}.$$
Because $g_\text{ambient}(r,w)$ is a doubly warped product metric with positive sectional curvature by (\ref{ambient:curvature}), $R'(t)>0$ for all $t<\pi/2$ (see \cite[Sections 1.4.5 and 4.2.4]{Pet}). Note that $\gamma_2'(s)\ge 0$ and that $\gamma_2'(s)=0$ when $\gamma_1(s)=0$. When $\gamma_1(s)=0$, we see that $R(\gamma_1(s))=0$ and the denominator of $\2_2(\sigma/R,\sigma/R)$ vanishes as well. The limit can easily be shown to have the same sign as $-\gamma_2''(s)$, which is positive. This shows that (\ref{eambient:in1}) holds for any choice of $x_1$, $x_2$, and $r$. 

 \begin{figure}
\centering

 \begin{tikzpicture}[scale=.15]
 \node (img)  {\includegraphics[scale=0.15]{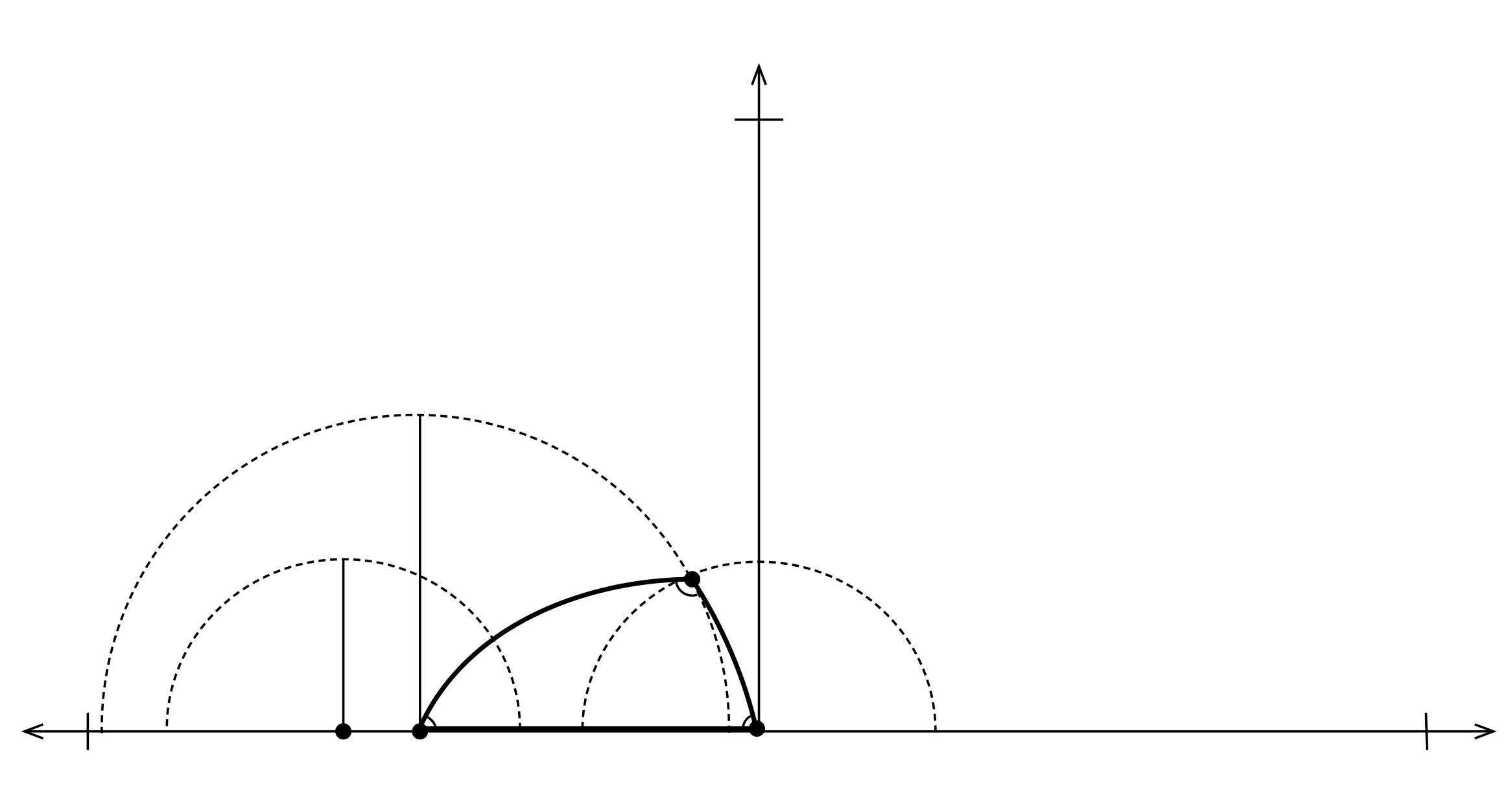}};
 \node at (-3,-16) {$\theta_0$};
  \node at (0,-20) {$0$};
   \node at (-17.5,-20) {$x_1$};
\node at (-22.5,-20) {$x_2$};
  \node at (-4.5,-12) {$\theta_r$};
 \node at (-15,-16) {$\theta_1$};
\node at (-1,-12) {$r$};
\node at (-13,-10) {$z$};
\node at (-40,-20) {$x$};
\node at (-36,-21) {$\frac{\pi}{2}$};
\node at (36,-21) {$-\frac{\pi}{2}$};
\node at (3,15) {$\frac{\pi}{2}$};
\node at (-2,18) {$t$};
 \end{tikzpicture}

 \caption{A schematic of the geodesic triangle used in the proof of Lemma \ref{eambient}. The solid lines all indicate geodesics, with the bold lines indicate those forming the triangle. The dotted lines indicate the boundary of the three geodesic balls.}
  \label{fig:triangle}
\end{figure}
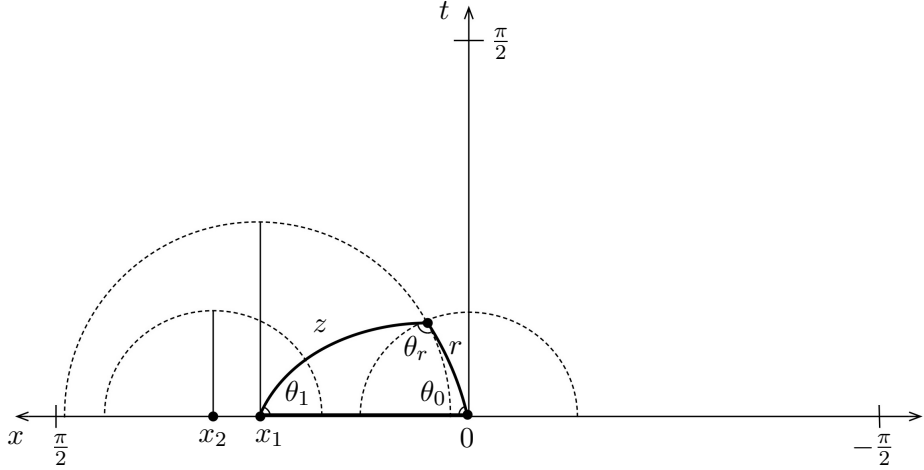

Let $0<\theta_0<\pi/2$ and $\pi/2< \theta_r < \pi$. Take a geodesic segment emanating from $(0,0)$ of radius $r$ at an angle of $\theta_0$ with the $x$-axis, and take a geodesic ray emanating from the endpoint of this geodesic making an angle of $\theta_r$. Let $z$ be the distance along the geodesic ray where it intersects the $x$-axis and let $(0,x_1)$ be the point where this intersection occurs. Let $\theta_1$ be the angle made between this geodesic and the $x$-axis. Taking these two geodesics segments together with the segment of the $x$-axis between $(0,0)$ and $(0,x_1)$ forms a geodesic triangle pictured in (\ref{fig:triangle}). 

We claim that we may choose $0<\theta_0<\pi/2$ and $\pi/2<\theta_r<\pi$ so that $z=2r$. First we must investigate the effect of varying $\theta_0$ and $\theta_r$ on $\theta_1$. Applying the spherical law of cosines to this geodesic triangle, we have that
$$\cos(\theta_1) = - \cos( \theta_r) \cos(\theta_0) + \sin(\theta_r)\sin(\theta_0) \cos(r).$$
From this we see that the limit of $\cos(\theta_1)$ as $(\theta_0,\theta_r)\rightarrow( \pi/2 , \pi)$ is $0$ and as $(\theta_0,\theta_r) \rightarrow ( \pi/2, \pi/2)$ is $\cos r$. We may now investigate the effect of varying $\theta_0$ and $\theta_r$ on $z$. Applying the spherical law of sines, we have that
$$ \sin(z) = \frac{\sin( \theta_0) \sin (r)}{\sin ( \theta_x)} .$$
The limit of $\sin(z)$ as $(\theta_0,\theta_r) \rightarrow (\pi/2,\pi)$ is $\sin(r)$ and as $(\theta_0,\theta_r)\rightarrow (\pi/2,\pi/2)$ is $1$. We conclude that there is some $\theta_0<\pi/2$ and $\pi/2< \theta_r<\pi$, for which $\sin(z) =\sin(2r)$ (as we are assuming $r<\pi/4$ to begin with). Fix such a choice of $\theta_0$ and $\theta_r$. Because $\theta_r<\pi$, the boundary of the geodesic ball of radius $2r$ centered at $(0,x_1)$ intersects the boundary of the geodesic ball of radius $r$ centered at $(0,0)$ in an angle $\theta < \pi /2$ as desired. Clearly $x_1> r$, lest the round metric violate the triangle inequality, and so we may now choose $x_2$ so that $ x_1 < x_2 < x_1+r$ so that the geodesic ball of radius $r$ centered at $(0,x_2)$ is contained in the interior of this geodesic ball. This shows that there is a choice of $\iota_\text{ambient}$ for which (\ref{eambient:out0}) and (\ref{eambient:angles}) hold as well. 
\end{proof}

Next we turn to construct the embedding (\ref{iotaneck}) pictured in Figure \ref{fig:eneck}. 

\begin{lemma}\label{eneck} There is an embedding (\ref{iotaneck}) such that:
\begin{enumerate}
\item\label{eneck:in0} the metric $g_\text{neck}(r,w,\rho)$ restricted to $\Wy^m$ takes the form $ds^2 + k_1^2(s) ds_m^2$, with $s\in[S_0,S_1]$ where $k_1''(s)<0$, 
\item\label{eneck:in1} the principal curvatures of $\Wy^m$ are nonnegative and positive on a neighborhood of $\widetilde{\Wy}_1^m$ with respect to $g_\text{neck}(r,w,\rho)$,
\item\label{eneck:straight} the manifold $\Wy_1^m$ corresponds to the set $x=b_1>0$ away from a neighborhood of $\widetilde{\Wy}_1^m$, 
\item\label{eneck:out0} $\widetilde{\Wy}^m$ corresponds to the submanifold described in (\ref{eambient:out0}) under the isometry hypothesized in (\ref{neck:10}) of Proposition \ref{neck},
\item\label{eneck:angles} the dihedral angle made between $\Wy_1^m$ and $\widetilde{\Wy}_1^m$ can be taken arbitrarily close to $\pi/2$. 
\end{enumerate}
\end{lemma} 

\begin{proof} By (\ref{neck:10}) of Proposition \ref{neck}, the metric $g_\text{neck}(r,w,\rho)$ restricted to the boundary of $\Em_1^{m+1}$ at $t=l$ is isometric to the boundary of $
\Em_2^{m+1}$ with respect to $g_\text{ambient}(r,w)$. Clearly there is a $0 < b_0<\pi/2$ so that the subspace $\{(l,b_0)\}\times \Sphere^{m-1}$ of the boundary of $\Em_1^{m+1}$ corresponds to the subspace of the boundary of $\Em_2^{m+1}$ that is the intersection of the boundary of the geodesic ball of radius $r$ centered at $(0,0)$ that is the intersection with all geodesics emanating from $(0,0)$ at an angle of $\theta\le \theta_0$ with the $x$-axis. Thus (\ref{eneck:out0}) is equivalent to embedding $\Ex_1^{m+1}$ in $\Em_1^{m+1}$ so that $\widetilde{\Wy}_1^m$ corresponds to the subspace $\{(t,x,p )\in [0,l]\times [-\pi/2 ,\pi/2] \times \Sphere^{m-1}: t=l \text{ and } x\ge b_0\}$. 

 Let $\phi_\varepsilon(a)$ be a family of functions satisfying the following.
\begin{enumerate}[(i)]
\item \label{eneckboundary:00} $\phi_\varepsilon(t) = b_0 +\varepsilon$ for $a\le l/2$,
\item \label{eneckboundary:10} $\phi_\varepsilon(l)=b_0$,
\item\label{eneckboundary:11}  $\phi_\varepsilon'(t)\le 0$ and $\phi_\varepsilon'(l)< 0$, 
\item\label{eneckboundary:12}  $\phi_\varepsilon''(t)\le0$ and $\phi_\varepsilon''(l)<0$, 
\item \label{eneckboundary:0} $\phi_\varepsilon(t)$ converges uniformly to $b_0$ as $\varepsilon\rightarrow 0$. 
\end{enumerate}

We now define the image of $\iota_\text{neck}$ to be the subset $\{(t,x,p)\in [0,l]\times [0,\pi] \times \Sphere^{m-1}: b\le \phi_\varepsilon(t)\}$ for some sufficiently small $\varepsilon>0$. The image of this embedding corresponds to the shaded region in (\ref{fig:eneck}). By (\ref{eneckboundary:10}), this embedding satisfies (\ref{eneck:out0}) for all $\varepsilon$. If we set $b_1= b_0 + \varepsilon$, then this embedding satisfies (\ref{eneck:straight}) for each $\varepsilon$ by (\ref{eneckboundary:00}). Note that (\ref{eneck:in0}) holds for $\varepsilon=0$ by the form of the metric $g_\text{neck}(r,w,\rho)$ and (\ref{neck:Bproperties}) of Proposition \ref{neck}. Thus by (\ref{eneckboundary:0}), (\ref{eneck:in0}) will hold for all $\varepsilon$ sufficiently small. Similarly we see that the dihedral angles for the boundary are identically $\pi/2$ when $\varepsilon=0$. By (\ref{eneckboundary:0}), (\ref{eneck:angles}) will hold for all $\varepsilon$ sufficiently small. All that remains to see check is (\ref{eneck:in1}) holds. This follows from Lemma \ref{cornersecond} as well as (\ref{eneckboundary:11}) and (\ref{eneckboundary:12}). 
\end{proof}

 \begin{figure}
\centering

 \begin{tikzpicture}[scale=.15]
 \node (img)  {\includegraphics[scale=0.15]{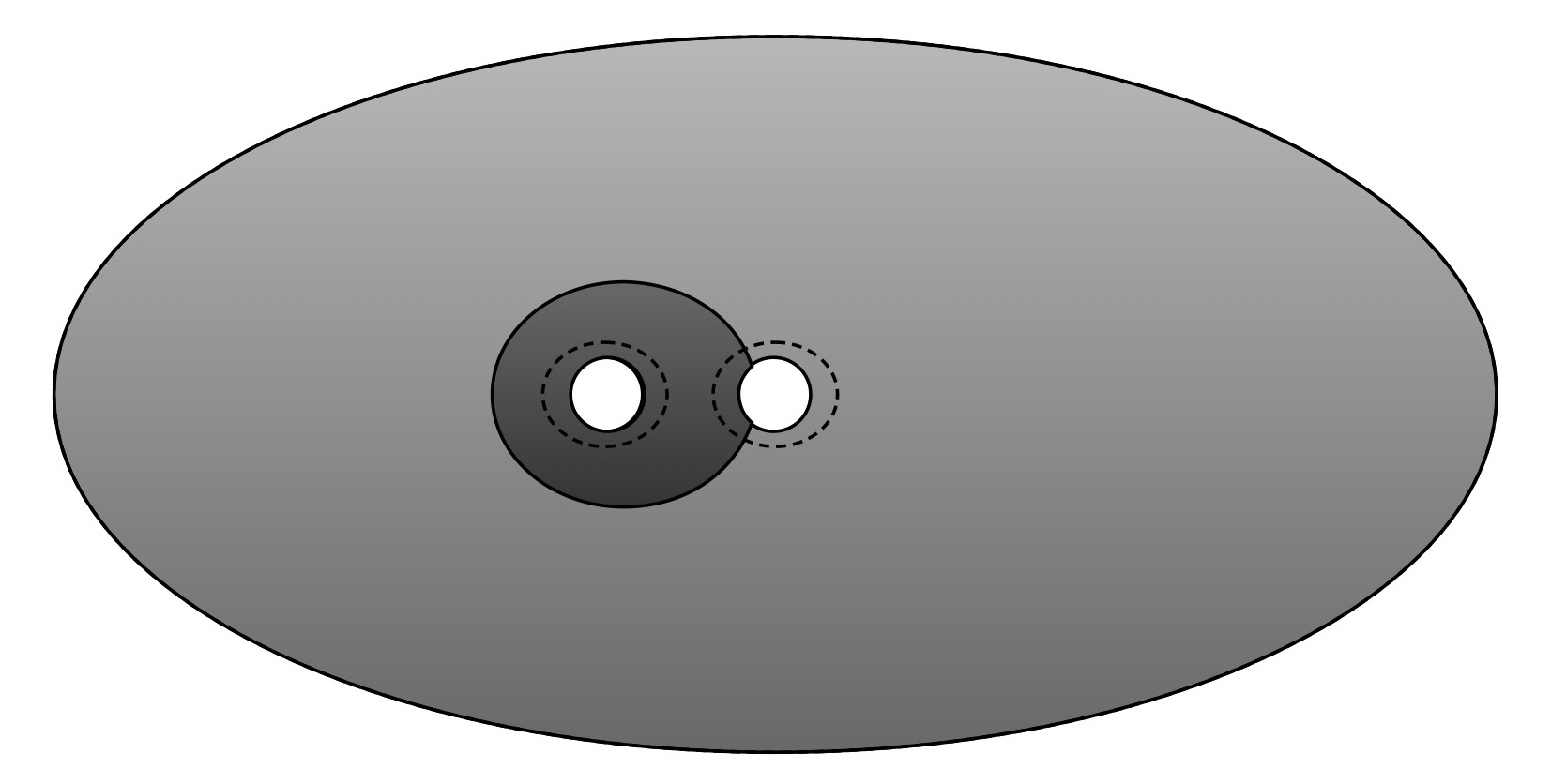}};
 \end{tikzpicture}

 \caption{The construction of the docking station $\Em^{m+1}$ using Theorem \ref{glue}, and the corresponding embedding of $\Ex^{m+1}$ constructed using Corollary \ref{glueconcave}. The dotted lines indicated where we have glued Figure \ref{fig:eneck} to Figure \ref{fig:eambient}.}
  \label{fig:edocking}
\end{figure}

Having defined our two embeddings (\ref{iotaambient}) and (\ref{iotaneck}), we are prepared to apply Corollary \ref{glueconcave} to define our embedding (\ref{iotadocking}). Note that $\Em^{m+1} = \Em_1^{m+1}\cup_{\En^m} \Em_2^{m+1}$ is itself a manifold with boundary $\En^m= \Sphere^m\sqcup \Sphere^m$. Similarly $\Ex^m=\Ex_1^{m+1}\cup_{\widetilde{\Wy}^m} \Ex_2^{m+1}$ is itself a manifold with faces. The embedding (\ref{iotadocking}) will be an embedding of a manifold with faces into a manifold with boundary $\iota_\text{docking}:\Ex^{m+1} \hookrightarrow \Em^{m+1}$ relative to its face $\widetilde{\Wy}^m \cong \Disk^m$, with interior face $\Wy^m = \Wy_1^{m}\cup_{\Zee^{m-1}} \Wy_2^m \cong \Disk^m.$

\begin{lemma}\label{edocking} For all $\nu>0$ sufficiently small there is an embedding of a manifold with faces into a manifold with boundary $\iota_\text{docking}: \Ex^{m+1}\hookrightarrow \Em^{m+1}$ relative to its face $\widetilde{\Wy}^m$ such that the principal curvatures of $\Wy^m$ are all nonnegative with respect to $g_\text{docking}(\nu)$. 

Moreover, there are normal coordinates $(t,x,p)\in [0,\varepsilon)\times [-\pi/2,\pi/2]\times \Sphere^{m-1}$ of the boundary $\En^m$ in which the metric $g_\text{docking}(\nu)$ can be written as $dt^2 + A^2(t,x) dx^2 + B^2(t) \cos^2 x ds_{m-1}^2$, and the face $\Wy^m$ corresponds to the set $x= b_1 >0$ in these coordinates. Moreover, the second fundamental form restricted to $\{(t,x)\}\times\Sphere^{m-1}$ is positive definite. 

Additionally, the metric $g_\text{docking}(\nu)$ restricted to $\Wy^m$ is isometric to 
$$ds^2 + k^2(s) ds_{m-1},$$
with $s\in [S_0,  S_2]$ such that 
\begin{enumerate}
\item \label{edocking:00} $k(0)= \cos b_1$,
\item \label{edocking:01} $0<k'(0) < \nu \cos b_1$ ,
\item \label{edocking:2} $k''(s)<0$ for $s < S_2$,
\item \label{edocking:1} $k'(S_2)=-1$ and $k^{(\text{even})}(S_2)=0$. 
\end{enumerate}
\end{lemma} 
\begin{proof}  We want to apply Corollary \ref{glueconcave} to the embeddings (\ref{iotaambient}) and (\ref{iotaneck}) constructed in Lemma \ref{eambient} and \ref{eneck} respectively.  By (\ref{eneck:out0}) of Lemma \ref{eneck} the isometry $\Phi:\En_1^m\rightarrow \En_2^n$ restricts to an isometry $\Phi: \widetilde{\Wy}^m_1 \rightarrow \widetilde{\Wy}^m_2$. By (\ref{eambient:angles}) of Lemma \ref{eambient} and (\ref{eneck:angles}) of Lemma \ref{eneck}, we may take the total dihedral angle of $\Ex_c^n$ introduced along the boundary to be less than $\pi$. We have from (\ref{eambient:in1}) of Lemma \ref{eambient} and (\ref{eneck:in1}) of Lemma \ref{eneck}, that the principal curvature are positive near the gluing site. Note that because both $g_\text{ambient}(r,w)$ and $g_\text{neck}(r,w,\rho)$ are warped product metrics, that the intrinsic concavity of $\Wy_i^m$ is equivalent to $k_i''(s)<0$, where $k_i(s)$ are the functions described in (\ref{eneck:in0}) of Lemma \ref{eneck} and (\ref{eambient:in0}) of Lemma \ref{eambient}. We may therefore apply Corollary \ref{glueconcave} to produce an embedding of a manifold with faces into a manifold with boundary $\iota_\text{docking} : \Ex^n \hookrightarrow \Em^n$ relative to its face, so that the principal curvatures of $\Wy^m$ are positive near the gluing site and agree with the principal curvatures of $\Wy_i^m$ with respect to $g_i$ away from the gluing site. 

The observation about the normal coordinates near the boundary $\En^m$ follows from Lemma \ref{eneck} and the fact that Corollary \ref{glueconcave} is a local smoothing process. 

That the metric $g_\text{docking}(\nu)$ restricted to $\Wy^m$ is a warped product metric follows from the fact that both $g_\text{ambient}(r,w)$ and $g_\text{neck}(r,w,\rho)$ will take the form $da^2 + \mu^2(a)db^2 + H^2(a,b) ds_{m-1}^2$ in normal coordinates (discussed below in Section \ref{corners:metric}) and the fact that the differential equation (\ref{ferguson}) used to define the smoothing respects such a decomposition. Conditions (\ref{edocking:00}) and (\ref{edocking:01}) follow directly Theorem \ref{docking}. Condition (\ref{edocking:2}) follows from Corollary \ref{glueconcave}. While condition (\ref{edocking:1}) follows from the fact that $g_\text{ambient}(r,w)$ is a doubly warped product metric on $\Sphere^{m+1}$ (see \cite[Section 1.4.5]{Pet}). 
\end{proof}

%%%%%%%%%%%%%%%%%%%

\subsubsection{Attaching the Handle}\label{body:handle}

%%%%%%%%%%%%%%%%%%%

 \begin{figure}
\centering

\begin{subfigure}{.5\textwidth}
  \centering
  
 \begin{tikzpicture}[scale=.15]
 \node (img)  {\includegraphics[scale=0.15]{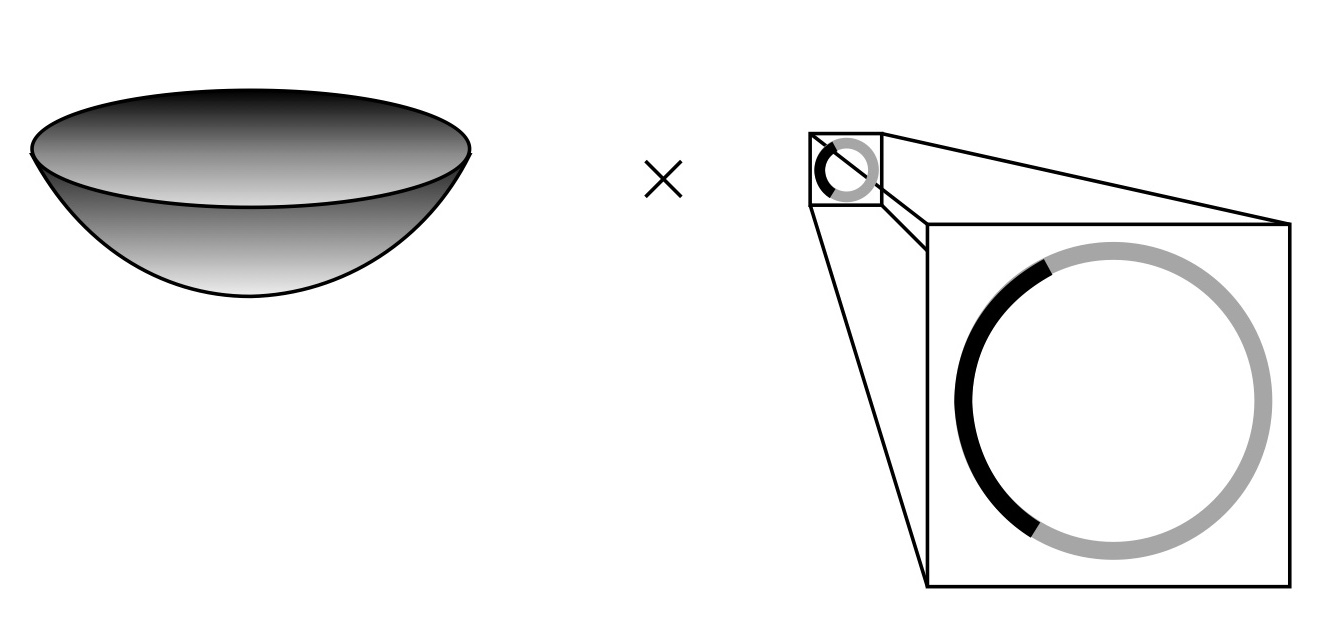}};
 \end{tikzpicture}
 
 \caption{The handle $\Em_1^{n+m}$. The manifold with faces $\Ex_1^{n+m}$ is embedded as the dark region. }
  \label{fig:ehandle}
  
 \end{subfigure}%
\begin{subfigure}{.5\textwidth}
  \centering
  
 \begin{tikzpicture}[scale=.185]
 \node (img)  {\includegraphics[scale=0.185]{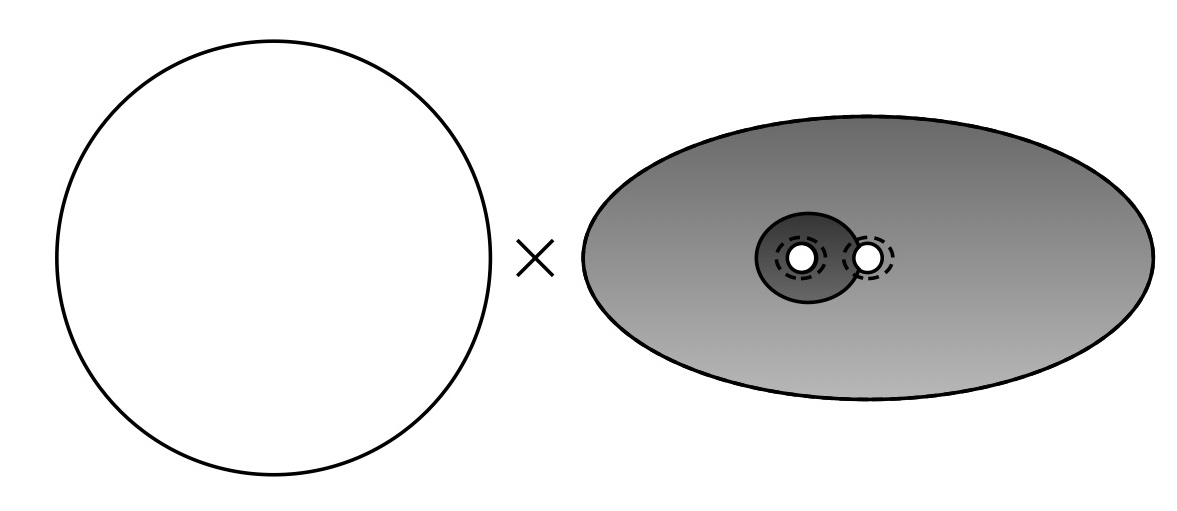}};
 \end{tikzpicture}
 
 \caption{The product of the sphere and the docking station $\Em_2^{n+m}$. The manifold with faces $\Ex_2^{n+m}$ is embedded as the dark region. }
  \label{fig:spheredocking}
\end{subfigure}

 \caption{The decomposition of $\Em^{n+m}$ and the corresponding decomposition of $\Ex^{n+m}$}
  \label{fig:esphere}
\end{figure}

In this section we will let $\Em^{n+m} = \Sphere^{n+m}\setminus \left( \Sphere^{n-1}\times \Disk^{m+1}\right)$ be a manifold with boundary $\En^{n+m} = \Sphere^{n-1}\times \Sphere^m$, that decomposes as a union $\Em^{n+m}  = \Em_1^{n+m}\cup_{\Phi} \Em_2^{n+m}$ of two manifolds with boundary, pictured in Figure \ref{fig:esphere}, where
$$\Em_1^{n+m} = \Disk^n \times \Sphere^m,\text{ and }\Em_2^{n+m} = \Sphere^{n-1}\times \left[ \Sphere^{m+1}\setminus \left(  \Disk^{m+1}\sqcup \Disk^{m+1}\right) \right],$$
are identified along a common boundary $\En_1^{n+m-1} \cong \En_2^{n+m-1} \cong \Sphere^{n-1}\times \Sphere^m$. We will let $\Ex^{n+m} = \Disk^{n+m}\setminus \left(\Sphere^{n-1}\times \Disk^{m+1}\right)$ be a manifold with boundary. In order to prove Lemma \ref{disk} we must define an embedding $\iota_\text{disk}:\Ex^{n+m}\hookrightarrow \Em^{n+m}$. The decomposition of $\Em^{n+m}$ induces a decomposition of $\Ex^{n+m} = \Ex^{n+m}_1\cup_{\widetilde{\Wy}^{n+m-1}} \Ex_2^{n+m}$ as the union of two manifolds with faces along a common face, also shown in Figure \ref{fig:esphere} as the darker regions, where in the notation of Section \ref{corners} we have:
\begin{align*}
\Ex_1^{m+1} & =  \Disk^n\times \Disk^m   &\Ex_2^{m+1}  &= \Sphere^{n-1}\times \left(\Bee_+^{m+1}\setminus \Disk^{m+1}\right)\\
\widetilde{\Wy}^m_1  &= \Sphere^{n-1}\times \Disk^m & \widetilde{\Wy}^m_2  &= \Sphere^{n-1}\times \Disk^m \\
\Wy^m_1  &=\Disk^n\times \Sphere^{m-1} & \Wy_2^m&  = \Sphere^{n-1}\times \Disk^m. 
\end{align*}

As explained in the introduction to Section \ref{body:docking}, we need to construct two embeddings $\iota_\text{handle}:\Ex_1^{n+m}\hookrightarrow \Em_1^{n+m}$ and $\text{id}\times\iota_\text{docking}:\Ex_2^{n+m}\hookrightarrow \Em_2^{n+m}$. We have already constructed the embedding $\text{id}\times \iota_\text{docking}$, pictured in (\ref{fig:spheredocking}), in Lemma \ref{edocking}. What remains is to construct the embedding (\ref{iotahandle}). 

\begin{lemma}\label{ehandle} There is an embedding $\iota_\text{handle}: \Ex_1^{n+m}\hookrightarrow \Em_1^{n+m}$ so that the principal curvatures of the inner face $\Wy^{n+m-1}_1$ are nonnegative. 

Moreover, there are normal coordinates $(t,x,p,q)\in [0,\pi R/3]\times [-\pi/2,\pi/2]\times \Sphere^{m-1}\times \Sphere^{n-1}$ of the boundary $\En_1^{n+m}$ so that the metric $g_\text{handle}(R,\nu)$ is given by $dt^2+  A^2_1(t,x) dy^2 + B^2_1(t)\cos^2 x ds_{m-1}^2 + C^2_1(t) ds_n^2$. In these coordinates $\Wy^{n+m-1}_1$ corresponds to the set $y=b_1 > 0$, and the second fundamental form of $\Wy^{n+m-1}_1$ restricted to $\{t\}\times \{b_1\}\times \Sphere^{m-1} \times \{z\}$ is positive definite. 

Additionally, the metric $g_\text{handle}(R,\nu)$ restricted to $\Wy^{n+m-1}_1$ is isometric to 
$$ ds^2 + k^2_1(s) ds_{m}^2 + h^2_1(s) ds_{n-1}^2.$$
with $s\in[0,S_0]$. The function $k_1(s)$ satisfies the following. 
\begin{enumerate}
\item\label{ehandle:k00}  $k_1(0)= \cos b_0 $,
\item\label{ehandle:k01} $k^{(\text{odd})}_1(0)=0$,
\item\label{ehandle:k11}  $k_1'(\pi R/3)> \nu \sin b_1$
\item\label{ehandle:k2}  $k_1(s)$ converges uniformly to $\cos b_1$ as $\nu \rightarrow 0$.
\end{enumerate}
And the function $h_1(s)$ satisfies the following. 
\begin{enumerate}
\item\label{ehandle:h00}  $h^{(\text{even})}_1(0)=0$,
\item\label{ehandle:h01}  $h_1'(0)=1$,
\item\label{ehandle:h2}  $h_1''(x)<0$ for $x>0$,
\item\label{ehandle:h10}  $h_1(\pi R/3) = R$, 
\item\label{ehandle:h11} $h_1'(\pi R/3) = 3 R$. 
\end{enumerate}
\end{lemma}

\begin{proof} Recall that for $t\in [0, R\pi/3]$ and $x\in [-\pi/2,\pi/2]$, that $g_\text{handle}(R,\nu)$ takes the form
$$ g_\text{handle}(R,\nu) = dt^2 + f_\nu^2(t) \left( dx^2 + \cos^2(x) ds_{m-1}^2\right)+ (2R)^2 \sin^2 (t/2R) ds_{n-1}^2 .$$
We may therefore define the embedding $\iota_\text{handle}$ to have image $\{(t,x,p,q)\in [0,\pi R/3]\times [-\pi/2,\pi/2]\times \Sphere^{m-1}\times \Sphere^{n-1}: y \ge b_1\}$. This is pictured in Figure (\ref{ehandle}. It is clear $\Wy^{n+m-1}_1$ corresponds to the set $y=b_1$. Thus the metric restricted to this boundary is 
$$ dt^2 +  f_\nu^2(t) \cos^2(b_1)ds_m^2 + (2R)^2 \sin^2(t/2R) ds_{n-1}^2 .$$
Conditions (\ref{ehandle:k00})-(\ref{ehandle:k2}) for $k_1(s)$ follow directly from Definition \ref{handlemetric}, while conditions (\ref{ehandle:h00})-(\ref{ehandle:h11}) $h_1(s)$ are clear. 

All that remains to discuss is the principal curvatures of $\Wy^{n+m-1}_1$. We can compute the second fundamental form $\2_1$ of $\Wy^{n+m-1}_1$ directly using \cite[Proposition 3.2.11]{Pet}. 
$$\2_1 =  f_\nu^2(t)\cos b_1 \sin b_1 ds_m^2.$$
The claim about principal curvatures is now clear. 
\end{proof}

Having constructed both (\ref{iotadocking}) and (\ref{iotahandle}) we are now prepared to prove Lemma \ref{disk}. 

\begin{proof}[Proof of Lemma \ref{disk}] Consider the embedding, pictured in (\ref{fig:spheredocking}), $\text{id}\times \iota_\text{docking}:\Ex_2^{n+m}\hookrightarrow \Em_2^{n+m}$. By Lemma \ref{edocking} there are normal coordinates $(t,x,p,q)\in [0,l] \times [-\pi/2,\pi/2]\times \Sphere^m\times \Sphere^{n-1}$ so that $\Wy^{n+m-1}_2$ corresponds to the set $y=b_1$. By Lemma \ref{ehandle}, we similarly have coordinates $(t,x,p,q)\in [0,\pi R/3]\times [-\pi/2,\pi/2] \times\Sphere^m\times\Sphere^{n-1}$ so that $\Wy^{n+m-1}_1$ corresponds to the set $y=b_1$. As these are already normal coordinates for the boundaries $\En_i^{n+m-1}$ with respect to $g_i$, when we apply Theorem \ref{glue} to construct $g_\text{sphere}(R,\nu)$ as in Definition \ref{spheremetric}, we do not need to worry about smoothing corners as the two faces $\Wy_1^{n+m-1}$ and $\Wy_2^{n+m-1}$ are both identified with the set $y=b_1$ near the gluing site. We may therefore define $\iota_\text{disk}$ by embedding each term on the righthand side of (\ref{maindecomp}) into the corresponding term in (\ref{handle}) using $\iota_\text{handle}$ and $\text{id}\times \iota_\text{docking}$ respectively. 
 
Because both metrics take the following form
$$g_i = dt^2 + A^2_i(t,x) dx^2 + B^2_i(t)\cos^2 x ds_m^2 + C^2_i(t) ds_n^2,$$
and because the differential equation (\ref{ferguson}) used to define the smoothing process of Theorem \ref{glue} respects this decomposition, the resulting metric $g_\text{sphere}(R,\nu)$ will also take this form. Near the gluing site, $\Wy^{n+m-1}=\Wy_1^{n+m-1}\cup_\Phi \Wy_2^{n+m-1}$ corresponds to the set $y= b_1$ and hence the second fundamental form takes the form 
$$ \2 =  B^2(x) \cos b_1 \sin b_1 ds_m^2.$$
We see immediately that the principal curvatures are all nonnegative on this neighborhood, and by Lemmas \ref{edocking} and \ref{ehandle} this is true for the entire boundary.  
  \end{proof}

%%%%%%%%%%%%%%%%%%%%%%%%

\subsection{The Ricci-positive Isotopy}\label{body:paths}

%%%%%%%%%%%%%%%%%%%%%%%%

All that remains to prove Theorem \ref{transition} and hence Theorem \ref{iterate}, is to show that the boundary $\Sphere^{n+m}$ of the embedding $\iota_\text{disk}$ constructed in the proof of Lemma \ref{disk} has a Ricci-positive metric that is connected via a path of Ricci-positive metrics to the round metric. We have already taken the first step in proving Lemma \ref{paths} by giving very careful description of the metrics $g_\text{docking}(\nu)$ and $g_\text{handle}(R,\nu)$ restricted to the faces of the manifolds with corners in (\ref{maindecomp}) that glue together to form the boundary $\Sphere^{n+m}$ of $\Disk^{n+m+1}$. Before we can prove Lemma \ref{paths} we must give a description of the metric $g_\text{handle}(R,\nu)$ constructed in the proof of Lemma \ref{disk} restricted to the boundary. Once this is established, we will rely on the fact that this metric is a doubly warped product to engineer a path of Ricci-positive metrics to the round metric. 

\begin{lemma}\label{esphere} The metric $g_\text{disk}(R,\nu)$ restricted to the boundary $\Sphere^{n+m-1}$ of $\Disk^{n+m}\setminus\left( \Sphere^{n-1}\times \Disk^{m+1}\right)$ takes the form
$$ds^2 + k^2(s) ds_m^2 + h^2(s) ds_{n-1}^2.$$
Where $s\in [0, T]$. Moreover there are $0<T_0<T_1<T_2<T$ such that $k(s)$ satisfy the following.
\begin{enumerate}
\item \label{esphere:k00} $k(s)$ satisfies (\ref{ehandle:k00}) and (\ref{ehandle:k01}) of Lemma \ref{ehandle} for $s<T_0$,
\item \label{esphere:k02}$k(s)$ converges uniformly to $\sin b_1$ for $s<T_0$ as $\nu\rightarrow 0$,
\item \label{esphere:k12} $k''(s)<0$ for $s>T_1$.
\item \label{esphere:k31} $k(s)$ satisfies (\ref{edocking:1}) of Lemma \ref{edocking} at $s=T$. 
\end{enumerate}
And the function $h(s)$ satisfy the following. 
\begin{enumerate}
\item \label{esphere:h00} $h(s)$ satisfies (\ref{ehandle:h00})-(\ref{ehandle:h2}) of Lemma \ref{ehandle} for $s<T_0$,
\item \label{esphere:h11} $-h''(s)/h(s) > 1/5R$ for $s< T_1$,
\item \label{esphere:h22} $h''(s)<0$ for $t<T_2$,
\item \label{esphere:h32} $h(s)$ is arbitrarily $C^2$-close to $R$ for $s >T_2$. 
\item \label{esphere:h30} $h(s) = R$ for $s> T_3$. 
\end{enumerate}
\end{lemma}

 \begin{figure}
\centering
  
 \begin{tikzpicture}[scale=.15]
 \node (img)  {\includegraphics[scale=0.15]{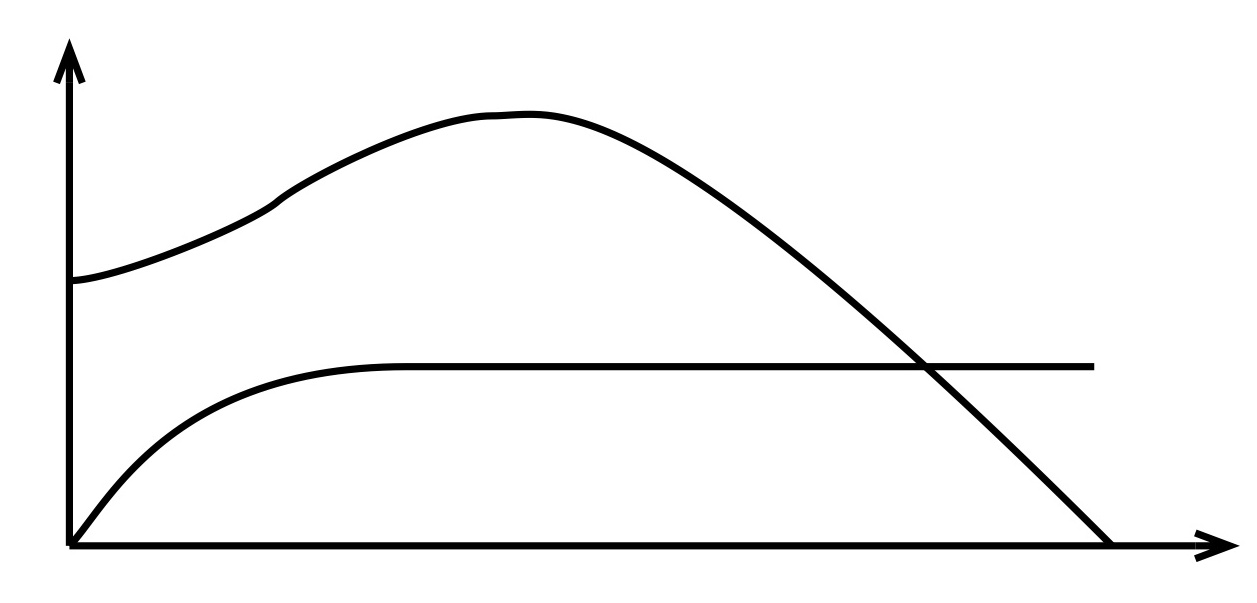}};
 \node at (-5,8.5) {$k(s)$};
  \node at (15,0) {$h(s)$};
 \end{tikzpicture}

 \caption{A rough sketch of the warping functions of $g_\text{sphere}(R,\nu)$ restricted to the boundary of $\iota_\text{disk}$ described in Lemma \ref{esphere}.}
  \label{fig:kandh}
\end{figure}

\begin{proof} We begin by setting $T_0 = \pi R/3 -\varepsilon -\delta$, $T_3 = \pi R/3 + \varepsilon +\delta$, and $T= \pi R/3 + S_2$, where $\varepsilon$ and $\delta$ are the two parameters used in the smoothing process of Theorem \ref{glue} (see the proof of Theorem \ref{gluecorners} below or the proof of \cite[Theorem 2]{BWW}). Thus for $s<T_0$ the metric agrees with $g_\text{handle}(R,\nu)$. Thus $k(s)$ satisfies (\ref{esphere:k00}) and (\ref{esphere:k02}) by Lemma \ref{ehandle} and satisfies (\ref{esphere:k12}) for $s>T_3$ and (\ref{esphere:k31}) by Lemma \ref{edocking}. Similarly $h(s)$ satisfies (\ref{esphere:h00}) by Lemma \ref{ehandle} and (\ref{esphere:h30}) by construction. 

 As already observed in the proof of Lemma \ref{disk}, the differential equation (\ref{ferguson}) that defines the smoothing process of Theorem \ref{glue} respects the decomposition of the metrics near the gluing site, and hence applying Theorem \ref{glue} to construct $g_\text{sphere}(R,\nu)$ has the effect of applying Lemmas \ref{firstorder} and \ref{secondorder} directly to the functions $A_i(t,x)$, $B_i(t)$, and $C(t)$ in the respective decompositions of $g_1$ and $g_2$. As $\Wy^{n+m-1}$ corresponds to the set $x=b_1$ on this neighborhood, we see that the metric $g_\text{sphere}(R,\nu)$ restricted to $\Wy^{n+m-1}$ takes the form $ds^2+ k^2(s) ds_{m-1}^2 + h^2(s) ds_{n-1}^2$, where $k(s)$ and $h(s)$ are also the result of applying Lemmas \ref{firstorder} and \ref{secondorder} to $k_1(s)$ and $h_1(s)$ in Lemma \ref{ehandle} and $k_2(s)=k(s)$ in Lemma \ref{edocking} and $h_2(s) = R$. 
 
Let $\overline{k}(s)$ denote the first order smoothing achieved by replacing $k_i(s)$ with $k_{1,\varepsilon}(s)$ on $[\pi R/3-\varepsilon,\pi R/3 + \varepsilon]$. We note that (\ref{edocking:01}) of Lemma \ref{edocking} and (\ref{ehandle:k11}) of Lemma \ref{ehandle} together with (\ref{first:2}) of Lemma \ref{firstorder} imply that $\overline{k}''(s) <0$ for $\varepsilon$ sufficiently small and $ s\in [\pi R/3 - \varepsilon, \pi R/3+\varepsilon]$. Let $\breve{k}(s)$ denote the second order smoothing achieved by replacing $\overline{k}(s)$ with $\overline{k}_{2,\delta}(s)$ on $[\pi R/3 \pm \varepsilon - \delta, \pi R/3 \pm \varepsilon +\delta]$. As $k''(s)<0$ for $s>T_3$ by (\ref{edocking:2}) of Lemma \ref{edocking}, (\ref{second:2}) of Lemma \ref{secondorder} implies  that $\breve{k}''(s) <0$ for all $s>T_1$ for some $T_1 > T_0$. As the smooth function $k(s)$ is arbitrarily $C^2$-close to $\breve{k}(s)$, we deduce (\ref{esphere:k12}). 

Let $\overline{h}(s)$ denote the first order smoothing achieved by replacing $h_i(s)$ with $h_{1,\varepsilon}(s)$ on $[\pi R/3 - \varepsilon , \pi R/3 + \varepsilon]$. By (\ref{ehandle:h11}) of Lemma \ref{ehandle} and (\ref{first:2}) of Lemma \ref{firstorder} we see that $\overline{h}''(s) \rightarrow -\infty $ as $\varepsilon\rightarrow 0$ for all $ s\in [\pi R/3 - \varepsilon, \pi R/3+\varepsilon]$. At the same time $\overline{h}(s)\rightarrow R$ as $\varepsilon\rightarrow 0$ for all $ s\in [\pi R/3 - \varepsilon, \pi R/3+\varepsilon]$. Thus we can assume $-\overline{h}''(s)/ \overline{h}(s) > 1/4R$ for $ s\in [\pi R/3 - \varepsilon, \pi R/3+\varepsilon]$. Note that $\overline{h}''(s)/\overline{h}(s)= 1/4R$ for $s< \pi R/3-\varepsilon$ by Lemma \ref{ehandle}. Let $\breve{h}(s)$ denote the second order smoothing achieved by replacing $\overline{h}(s)$ with $\overline{h}_{2,\delta}(s)$ on $[\pi R/3 \pm \varepsilon - \delta, \pi R/3 \pm \varepsilon +\delta]$. By (\ref{second:0}) and (\ref{second:2}) of Lemma \ref{secondorder}, we see that $-\breve{h}''(s)/\breve{h}''(s)$ will be arbitrarily close to the convex combination of $-\overline{h}''(s)/\overline{h}''(s)$ evaluated at the endpoints of the interval $[\pi R/3 - \varepsilon - \delta, \pi R/3 - \varepsilon +\delta]$. Hence we can assume that (\ref{esphere:h11}) holds for $\delta$ sufficiently small. Similarly, if we take $T_1<T_2<T_3$ sufficiently close to $T_3$, we see that $\breve{h}(s)$ is arbitrarily close to $R$. But we can also choose $\delta$ sufficiently small so that $\breve{h}''(s)<0$ for $s<T_2$. Thus (\ref{esphere:h22}) and (\ref{esphere:h32}) hold as well. 
\end{proof}

%INSERT DESCRIPTION USING PICTURE

Now that we have given a qualitative description of the metric $g_\text{sphere}(R,\nu)$ restricted the boundary $\Sphere^{n+m-1}$ of the image of $\iota_\text{disk}$ of Lemma \ref{disk}, we hope to find a path of Ricci-positive metrics connecting it to the round metric. As the metric restricted to the boundary is a doubly warped product metric, we will consider a path of doubly warped product metrics $g_\lambda$ where the warping functions are just the convex combination of the warping functions of $g_0$ and $g_1$. This path has the advantage that the sectional curvatures are given by a very simple formula \cite[Section 4.2.4]{Pet}. We will have to concatenate two such paths to connect $g_\text{sphere}(R,\nu)$ restricted to $\Wy^{n+m-1}$ to the round metric. We will now describe the first such path.

 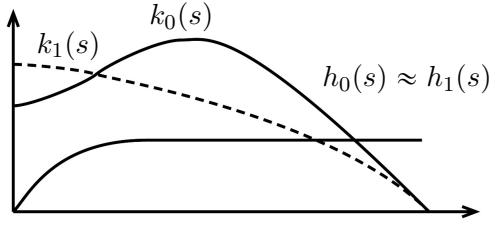
\begin{figure}
\centering
  
 \begin{tikzpicture}[scale=.15]
 \node (img)  {\includegraphics[scale=0.15]{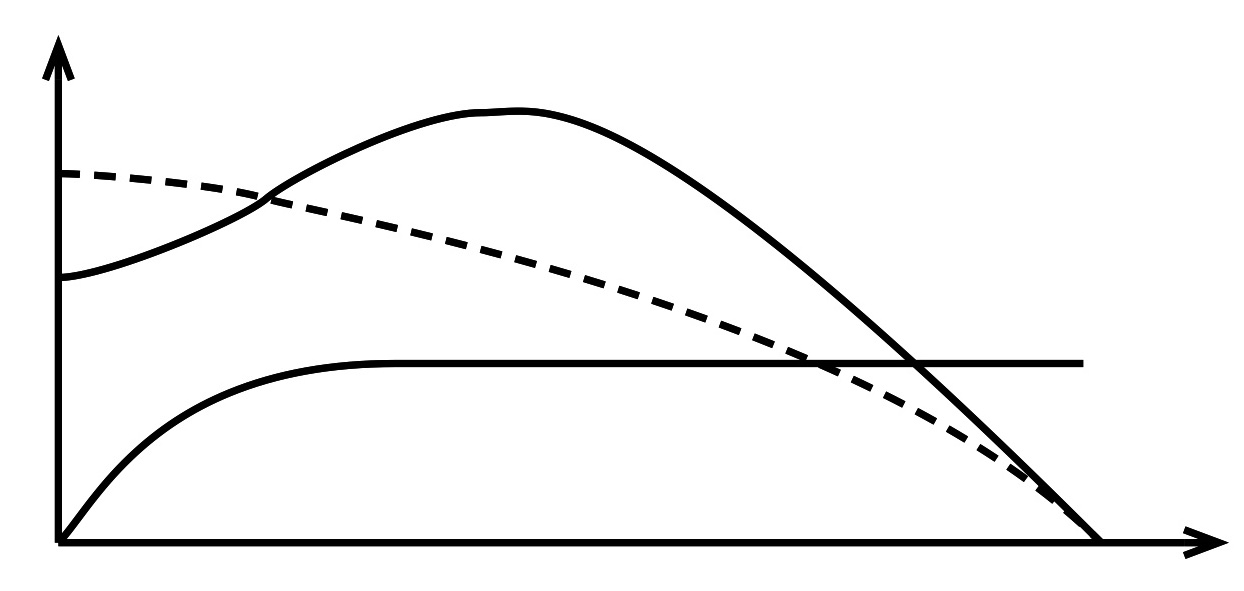}};
 \node at (-5,8.5) {$k_0(s)$};
  \node at (-15,6) {$k_1(s)$};
  \node at (15,3) {$h_0(s)\approx h_1(s)$};
 \end{tikzpicture}

 \caption{The functions of Definition \ref{kandh} pictured alongside the functions of Lemma \ref{esphere}.}
  \label{fig:k1andh1}
\end{figure}

\begin{definition}\label{kandh} Let $k_0(s) = k(s)$ and $h_0(s)=h(s)$ be any functions satisfying the conditions of Lemma \ref{esphere}. Let $k_1(s)$ be any smooth functions defined on $[0,T]$ such that
\begin{enumerate}
\item\label{k1:ends} $k_1^{(\text{odd})}(0) = k_1^{(\text{even})}(T) =0$, 
\item\label{k1:ends1} $k_1'(T)=-1$,
\item\label{k1:0}  $k_1(T_1) = k_0(T_1)$,
\item\label{k1:2}  $k_1''(s)<0$ for $s<T$,
\item\label{k1:01}  $-\nu \cos b_1 < k_1'(s)<0$ for $s\le T_2$. 
\end{enumerate}
Let $h_1(s)$ be any smooth function defined on $[0,T]$ such that
\begin{enumerate}
\item\label{h1:ends} $h_1(s) = h_0(s)$ for $s<T_0$
\item\label{h1:0}  $h_1(s)=R$ for $s\ge T_3$,
\item\label{h1:2}  $h_1''(s)<0$ for $s<T_3$,
\item\label{h1:h0}  by choosing $\delta$ sufficiently small in the construction of $g_\text{sphere}(R,\nu)$, $h_0(s)$ can be made arbitrarily close to $h_1(s)$. 
\end{enumerate}

Define the metric $g_\lambda$ on $\Sphere^{n+m-1}$ for $\lambda\in[0,1]$ as follows.
\begin{align*}
 k_\lambda(s) &= (1-\lambda)k_0(s) + \lambda k_1(s)\\
 h_\lambda(s)& = (1-\lambda)h_0(s) + \lambda h_1(s)\\
 g_\lambda &= ds^2 + k_\lambda^2(s) ds_m^2 + h_\lambda^2 ds_{n-1}^2.
\end{align*}
\end{definition} 

The functions $k_1(s)$ and $h_1(s)$ are pictured alongside $k_0(s)$ and $h_0(s)$ in Figure \ref{fig:k1andh1}. 

\begin{lemma}\label{path1} For $R>0$, if $\nu>0$ is sufficiently small then the metric $g_\lambda$ has positive Ricci curvature for all $\lambda \in [0,1]$. 
\end{lemma}
\begin{proof} By (\ref{h1:h0}) of Definition \ref{kandh}, the claim is equivalent to showing that the metric $ds^2 + k_\lambda^2(s) ds_m^2 + h_1^2(s) ds_{n-1}^2$ has positive Ricci curvature for all $\lambda\in[0,1]$ and $\nu$ sufficiently small. We will show that $g_\lambda$ has positive Ricci curvature independently on the subintervals defined by $0< T_0< T_1< T_2<T$. We will let $\theta_i$ and $\phi_j$ be a local orthonormal frame for $\Sphere^m$ and $\Sphere^{n-1}$ respectively. The sectional curvatures of $g_\lambda$ are given by the formulas in \cite[Section 4.2.4]{Pet}. 

Some special care is needed for $s= 0,T$ as the formulas of \cite[Section 4.2.4]{Pet} are not defined. At $s=0$ the affected curvatures are $\K_\lambda(\partial_s,\phi)$, $K_\lambda(\phi_i,\phi_j)$, and $K_\lambda(\theta,\phi)$. It is straightforward to check that these first two curvatures converge to $-h_1'''(s)$, while the last curvature converges to $0$. Similarly at $t=T$, the affected curvatures are $\K_\lambda(\partial_s,\theta)$, $K_\lambda(\theta_i,\theta_j)$, and $K_\lambda(\theta,\phi)$.  Again it is straightforward to check that these first two curvatures converge to $f_\lambda'''(s)$, while the last curvature converges to $0$. From the assumption that $g_\lambda$ are warped product metrics (see \cite[Section 1.4.5]{Pet}), we see that $f_\lambda''(T)= h_1''(0)=0$ while $f_\lambda''(s)<0$ for $s<T$ and $h_1''(s)<0$ for $s>0$.  Hence $f_\lambda'''(T)$ and $-h_1'''(0)$ are both negative, and the corresponding curvatures are positive. When we are arguing in that $g_\lambda$ has positive Ricci curvature in the cases $0\le s\le T_0$ and $T_2\le s\le T$ below, we will argue for $s\neq 0,T$. By our discussion here, the argument presented will automatically extend to include $s= 0,T$. 

Consider first $0\le s\le T_0$. By (\ref{h1:ends}), we have that $\K_\lambda(\partial_s, \phi ) = \K_\lambda(\phi_i,\phi_j) = \frac{1}{4R}$. By (\ref{k1:01}) and (\ref{esphere:k02}) of Lemma \ref{esphere}, we have that $\K_\lambda(\partial_s,\phi)$ and $\K_\lambda(\theta,\phi)$ both converge to $0$ as $\nu\rightarrow 0$, while $\K_\lambda(\theta_i,\theta_j)$ converges to $\frac{1}{R\cos b_1}$ as $\nu\rightarrow 0$. This shows that $\Ric_{g_\lambda}(\partial_s,\partial_s)$ and $\Ric_{g_\lambda}(\phi,\phi)$ converge to $\frac{n-1}{4R}$ while $\Ric_{g_\lambda}(\theta,\theta)$ converges to $\frac{m-2}{R \cos b_1}$ as $\nu\rightarrow 0$, and so $g_\lambda$ will have positive Ricci curvature for $\nu$ sufficiently small. 

Next we consider $T_0\le s\le T_1$. The argument is much the same as the previous case.  The only difference is that we only know that $\K_\lambda(\partial_s, \theta)> \dfrac{1}{5R}$ by (\ref{h1:ends}) and (\ref{esphere:h11}) of Lemma \ref{esphere}. We can bound $\K_{\lambda}(\theta_i,\theta_j)$ below by its value at $t=T_0$ as both $h'(s)$ and $h(s)$ are decreasing on this interval. The behavior of the remaining sectional curvatures is identical to the case $T_0 \le t \le T_1$. The argument that $g_\lambda$ has positive Ricci curvature for $\nu$ sufficiently small proceeds identically to the case $0\le s\le T_0$. 

For $T_1 <  s\le T_2$, we have that $k_\lambda''(s)\le 0$ by (\ref{k1:2}) of Definition \ref{kandh} and (\ref{esphere:k12}) of Lemma \ref{esphere} and $h_1''(s)< 0$ by (\ref{h1:2}) of Definition \ref{kandh}. This shows that $K_\lambda( \partial_s, \theta)\ge 0$ and $\K_\lambda(\partial_s,\phi)> 0$. We claim in addition that  $|k_\lambda'(s)|<1$ and $|h_1'(s)|<1$. By (\ref{esphere:k02}) of Lemma \ref{esphere}, $k_0'(T_0)>0$ can be made arbitrarily small by taking $\nu$ sufficiently small, while by (\ref{esphere:k31}) of Lemma \ref{esphere}, $k_0'(T) = -1$. As $k_0''(s)\le 0$ for $T_1<s<T$ by (\ref{esphere:k12}) of Lemma \ref{esphere}, this shows that $-1<k_0'(s)<1$. Similarly by (\ref{k1:01}) of Definition \ref{kandh}, $k_1'(0)$ can be made arbitrarily small by taking $\nu$ sufficiently small, while by (\ref{k1:ends1}) of Definition \ref{kandh} $k_1'(T) = -1$. As $k_1''(s)< 0$ by (\ref{k1:2}) of Definition \ref{kandh}, we again conclude that $-1 < k_1'(s) < 1$. Thus $|k_\lambda'(s)|<1$. By (\ref{h1:ends}) of Definition \ref{kandh} we have that $h_1'(0)=1$, while by (\ref{h1:0}) of Definition \ref{kandh}, $h_1'(T)=0$. As $h_1''(s)\le 0$ by (\ref{h1:2}) of Definition \ref{kandh}, we conclude that $|h_1'(s)|<1$. This shows that $\K_\lambda( \theta_i,\theta_j)>0$ and $\K_\lambda (\phi_i ,\phi_j)>0$. All that remains is to note that $\K_\lambda(\theta,\phi)$ converges to $0$ as $\nu\rightarrow 0$ because $k_\lambda'(s)$ converges to $0$ as $\nu \rightarrow 0$  by (\ref{k1:01}) of Definition \ref{kandh}. We conclude that 
\begin{align*}
\Ric_{g_\lambda}(\partial_s,\partial_s) & =(m-1)\K_{g_\lambda}(\partial_s,\theta) + (n-1)\K_{g_\lambda}(\partial_s,\phi) >0,\\
\Ric_{g_\lambda}(\theta,\theta) & \rightarrow \K_{g_\lambda}(\partial_s,\theta) + (m-2)\K_{g_\lambda}(\theta_i,\theta_j) >0,\\
\Ric_{g_\lambda}(\theta,\theta) & \rightarrow \K_{g_\lambda}(\partial_s,\phi) + (m-2)\K_{g_\lambda}(\phi_i,\phi_j) >0.
\end{align*}
Thus $g_\lambda$ has positive Ricci curvature for $T_1< s \le T_2$. 

Finally, consider $T_2 < s \le T$. The argument for such $s$ is nearly identical to the case of $T_1< s\le T_2$. By (\ref{esphere:k12}) of Lemma \ref{esphere} and (\ref{k1:2}) of Definition \ref{kandh}, $k_\lambda''(s)< 0$, while by (\ref{h1:2}) of Definition \ref{kandh} $h_1''(s)\le 0$. Thus $\K_{g_\lambda}(\partial_s,\theta)>0$ and $\K_{g_\lambda}(\partial_s,\phi)\ge 0$. For the same reasons as in the case $T_1< s\le T$ we have $|k_\lambda'(s)|<1$ and $|h_1'(s)|<1$, so that $\K_{g_\lambda}(\theta_i,\theta_j)>0$ and $\K_{g_\lambda}(\phi_i,\phi_j)>0$. All that remains is to show that $K_\lambda(\phi,\theta)$ can be made arbitrarily small as $h_1'(s)$ is arbitrarily close to $0$ by (\ref{h1:h0}) of Definition \ref{kandh} and (\ref{esphere:h32}) of Lemma \ref{esphere}. The argument that $g_\lambda$ has positive Ricci curvature is now identical to the case $T_0< s\le T_1$. 

We have shown that for $\nu$ taken sufficiently small depending on $R>1$, that $g_\lambda$ has positive Ricci curvature for all $\lambda\in[0,1]$ and $s\in[0,T]$. 
\end{proof} 

To conclude our proof of Lemma \ref{paths}, we note the the metric $g_1$ produced in Lemma \ref{path1} is a doubly warped product metric with  positive Ricci curvature and nonnegative sectional curvature. It is a general fact that such metrics are connected via a path of Ricci-positive metrics to the round metric. We note that having nonnegative sectional curvature is equivalent to both warping functions be weakly concave.

 \begin{figure}
\centering
  
 \begin{tikzpicture}[scale=.15]
 \node (img)  {\includegraphics[scale=0.15]{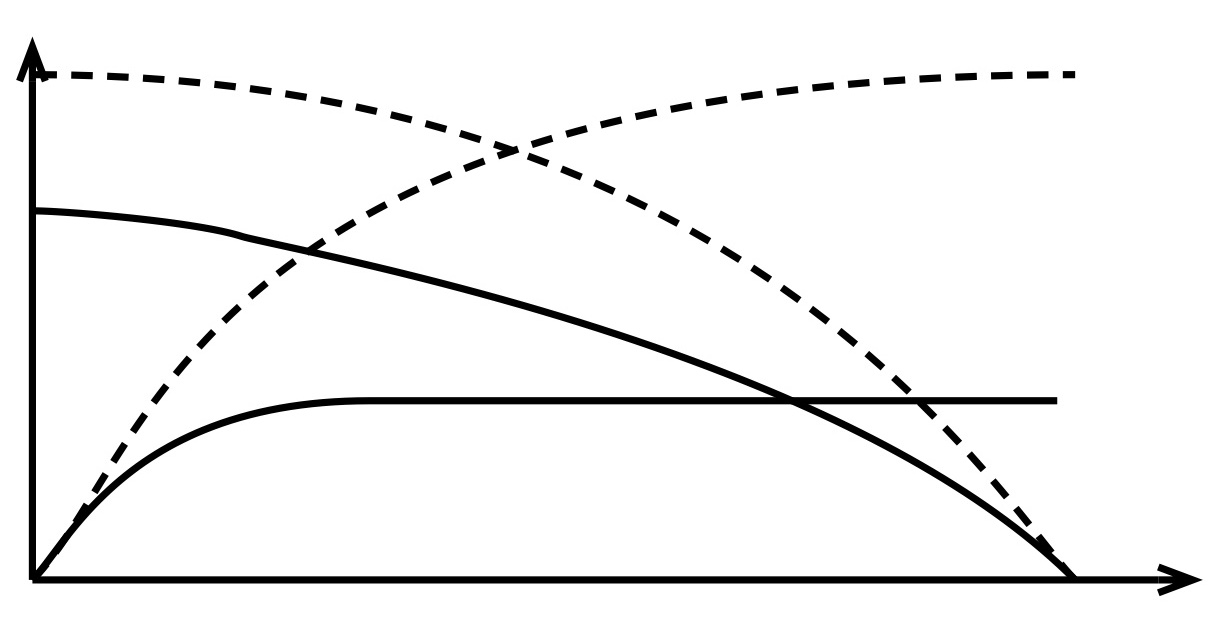}};
 \node at (-10,10) {$R\cos(s/R)$};
  \node at (10,10) {$R\sin (s/R)$};
  \node at (-15,6) {$k_1(s)$};
  \node at (15,0) {$h_1(s)$};
 \end{tikzpicture}

 \caption{The functions of Definition \ref{kandh} pictured alongside the corresponding warping functions for the round metric of radius $R$, as detailed in Definition \ref{spherepath}.}
  \label{fig:k2andh2}
\end{figure}
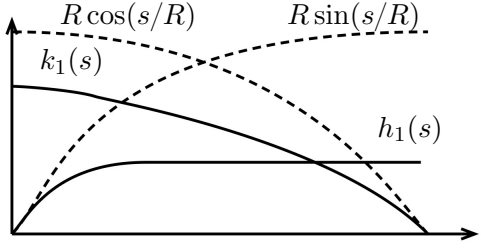

\begin{definition}\label{spherepath} Let $k_1(s)$ and $h_1(s)$ be any function defined for $s\in[0,\pi R/2]$ where $k_1''(s)\le 0$ and $h_1''(s)\le 0$. Assume additionally that 
$$ g_1 = ds^2 + k_1^2(s)ds_{m-1}^2 + h_1^2(s) ds_{n-1}^2,$$
is a metric on $\Sphere^{n+m-1}$ with positive Ricci curvature. For $\lambda\in[1,2]$, define
\begin{align*}
k_\lambda(s) &= (2-\lambda) k_1(s) + (\lambda-1) R\cos( s/ R) \\
h_\lambda(s) &= (2-\lambda) h_1(s) + (\lambda- 1) R\sin(s/R)  \\
g_\lambda & = ds^2 + k_\lambda^2(s)ds_{m-1}^2 + h_\lambda^2(s) ds_{n-1}^2.
\end{align*}
 Where $g_2$ is the round metric of radius $R$. 
 \end{definition}
The functions $k_1(s)$ and $h_1(s)$ are pictured alongside the warping functions for the round metric in Figure \ref{fig:k2andh2}.

\begin{lemma}\label{path2}  The path $g_\lambda$ has positive Ricci curvature for all $\lambda\in[1,2]$. 
\end{lemma}

\begin{proof} We will use the frame $\partial_s$. $\theta_i$, and $\phi_j$ as in the proof of Lemma \ref{path1}. The same care is needed when $s = 0, \pi R/2$ as in the proof of Lemma \ref{path1}. The argument is identical to the one presented in Lemma \ref{path1} as it is a general fact about doubly warped product metric where the warping functions are concave. We will again argue assuming that $s\neq 0,\pi R/2$, but note that the arguments will extend automatically to $s=0, \pi R/2$. 

We note that $k_\lambda''(s)< 0$ and $h_\lambda''(s)< 0$ for all $\lambda \in(1,2]$ and $s\neq 0, \pi R/2$ by construction. By \cite[Section 1.4.5]{Pet}, we have that $k_\lambda'(0) = h_\lambda'(\pi R/2) = 0$ and $-k_\lambda'(\pi R/2) = h_\lambda'(0) = 1$. We immediately conclude that $-1<k_\lambda'(s) <0$ and $0< h_\lambda'(s) < 1$. Thus $g_\lambda$ has positive sectional curvature for all $\lambda \in (1,2]$ by \cite[Section 4.2.4]{Pet}. As $g_1$ has positive Ricci curvature by assumption, we conclude that $g_\lambda$ has positive Ricci curvature for all $\lambda \in [1,2]$. 
\end{proof}

With this easy fact established, we have proven Lemma \ref{paths}

\begin{proof}[Proof of Lemma \ref{paths}] By Lemma \ref{esphere}, if we let $g_0$ denote the metric $g_\text{Sphere}(R,\nu)$ restricted to the $\Sphere^{n+m}$ boundary of the image of $\iota_\text{disk}$ of Lemma \ref{disk}, then $g_0$ satisfies the hypotheses of Lemma \ref{path1}. Hence there is a path of Ricci-positive metrics connecting it to a doubly warped product metric $g_1$ that has both its warping functions weakly concave by (\ref{k1:2}) and (\ref{h1:2}). By Lemma \ref{path2}, this metric is connected by a path of Ricci-positive metrics to the round metric. 
\end{proof}

This completes the proof of Lemma \ref{paths} and hence the proof of Theorem \ref{transition} and Theorem \ref{iterate}. 

\appendix

%%%%%%%%%%%%%%%%%%%%%%%%

\section{Bending Ricci-positive, Convex Boundaries}\label{bend}

%%%%%%%%%%%%%%%%%%%%%%%%

In this section we give a proof of Theorem \ref{bendboundary}. We begin by phrasing a different theorem from the same hypothesis that $g_0$ and $g_1$ are path connected in the space of all metrics, from which will give us a direct construction of the metric in conclusion of Theorem \ref{bendboundary}.

\begin{theorem}\label{isotopyconcordance} If $g_0$ and $g_1$ are path connected in the space of all Ricci-positive metrics on $\En^n$, then for all $1> \nu >0$ there is an $R>0$ and a Ricci-positive metric $g_\text{convex}(\nu)$ on $[0,1]\times \En^n$ such that
\begin{enumerate}
\item\label{iso:00} $g_\text{convex}(\nu)$ restricted to $\{0\}\times \En^n$ is isometric to $g_0$,
\item\label{iso:01} The principal curvatures of $\{0\}\times \En^n$ with respect to $g_\text{convex}(\nu)$ are all greater than $-\nu$,
\item\label{iso:10}  $g_\text{convex}(\nu)$ restricted to $\{1\}\times \En^n$ is isometric to $R^2g_1$,
\item \label{iso:11}  The principal curvatures of $\{1\}\times \En^n$ with respect to $g_\text{convex}(\nu)$ are all positive.
\end{enumerate}
\end{theorem}

\noindent Assuming Theorem \ref{isotopyconcordance} we can now give a short proof of Theorem \ref{bendboundary}. 

\begin{proof}[Proof of Theorem \ref{bendboundary}] Let $\nu>0$ be the infimum of the principal curvatures of the boundary of $(\Em^{n+1},g)$, then by Theorem \ref{isotopyconcordance} and Theorem \ref{glue} we can glue $(\Em^{n+1},g)$ to $([0,1]\times \En^n,g_\text{convex}(\nu))$ to produce a Ricci-positive metric on $\Em^n$ with convex boundary isometric to $R^2g_1$. Scaling this metric by $(1/R)$ produces the desired metric $\tilde{g}$. 
\end{proof}

\noindent In the author's dissertation \cite[Appendix B]{BLB2}, we discuss how the existence of a family of Ricci-positive metrics satisfying the claims of Theorem \ref{isotopyconcordance} plays the role that concordance plays for the study of positive scalar curvature (psc) metrics. In this light, Theorem \ref{isotopyconcordance} is a ``isotopy implies concordance'' theorem for metrics of positive Ricci curvature, which is a standard result for psc metrics (see \cite[Lemma II.1]{Walsh}).  Strictly speaking there is no such thing as a Ricci-positive concordance, but like psc concordances which can be glued to psc Riemannian manifolds with minimal boundaries, such a family of metrics can be used to glue to a Ricci-positive Riemannian manifold with convex boundary using Theorem \ref{glue}. 

The proof we give to Theorem \ref{isotopyconcordance} is an adapted version of the proof of \cite[Assertion]{Per1}. Both claim the existence of a Ricci-positive metric on the cylinder satisfying particular boundary conditions can be constructed from a hypothesized path of Ricci-positive metrics. While Theorem \ref{isotopyconcordance} works for any path, the hypotheses of \cite[Assertion]{Per1} ensures the existence of a particular path of metrics connecting two warped product metrics on the sphere. We are able to prove our version for any path as the conclusions in Theorem \ref{isotopyconcordance} are far weaker than those of \cite[Assertion]{Per1}. For example, in Theorem \ref{isotopyconcordance} we claim that one end is merely convex, while \cite[Assertion]{Per1} claims that the principal curvatures of one end can be bounded below by $1$. In addition to allowing an arbitrary path of metrics, the proof of Theorem \ref{isotopyconcordance} is much simpler than that of \cite[Assertion]{Per1}. Though the attentive reader will see that the main ideas of our proof of Theorem \ref{isotopyconcordance} are all present in the proof of \cite[Assertion]{Per1}. 

\subsection{A family of metrics on the cylinder}\label{bend:metric}

To begin we observe that because the space of all Ricci-positive metrics is an open subspace of the space of all symmetric $2$-tensors, we may approximate any continuous path of Ricci-positive metrics by a smooth path of Ricci-positive metrics. We will assume henceforth that $g_s$ with $s\in[0,1]$ is a smooth path of Ricci-positive metrics. We may now define a class of metrics on $[0,1]\times \En^n$ that depend on $4$-parameters. For all $0<r_0<r_1<1$ and $1<t_0<t_1$ we define 
\begin{equation}\label{concordancedef} G(t_0,t_1,r_0,r_1) = dt^2 +t^2 \rho^2(t) g_{\lambda(t)},\end{equation}
\noindent Where $\rho:[t_0,t_1]\rightarrow [r_0,r_1]$ and $\lambda:[t_0,t_1] \rightarrow [0,1]$ are the two bijections uniquely defined by 
\begin{equation}\label{Rlambda}  \alpha(t_0,t_1) \lambda'(t) = - \beta(t_0,t_1,r_0,r_1) \dfrac{\rho'(t)}{\rho(t)}  =  \Gamma(t),\end{equation}
\noindent where $\Gamma: (1,\infty) \rightarrow \mathbf{R}_+$ is the function 
\begin{equation}\label{Gamma} \Gamma(t) =  \dfrac{1}{t\ln^2 t}.
\end{equation}
\noindent Note that (\ref{Rlambda}) and (\ref{Gamma}) uniquely determines the formula for $\alpha(t_0,t_1)$ and $\beta(t_0,t_1,r)$. 
\begin{equation}\label{alphabeta} \alpha(t_0,t_1) =\frac{1}{\ln t_0}  -\frac{1}{\ln t_1} \text{ and }   \beta(t_0,t_1,r_0,r_1) =  \dfrac{ \alpha(t_0,t_1)}{\ln r_1 - \ln r_0}.
\end{equation}

We claim that this family can be used to prove Theorem \ref{isotopyconcordance}.

\begin{lemma}\label{isotopyconcordancelemma} For $0<2r_1 < \nu$ sufficiently small, there is an $0<r_0<r_1$ that depends only on $g_s$ and a $t_0>1$ sufficiently large that depends on $g_s$, $r_0$, and $r_1$ so that the metric $G(\nu) := (1/t_0 r_0)^2 G(t_0,t_0^2,r_0,r_1)$ satisfies the claims of Theorem \ref{isotopyconcordance}. 
\end{lemma} 

\noindent By construction, we have the following. 

\begin{prop}\label{iso00and10} For any choice of $t_0$, $t_1$, $r_0$, and $r_1$ the metric $G(\nu)$ satisfies (\ref{iso:00}) and (\ref{iso:10}) of Theorem \ref{isotopyconcordance} with $R=\dfrac{t_0 r_1}{ r_0}$. 
\end{prop}

\noindent In order to finish the proof of Lemma \ref{isotopyconcordancelemma}, we will need to show that for an adequately chosen $r_0$ and $r_1$ that $G(t_0,t_1,r_0,r_1)$ will have positive Ricci curvature and satisfy (\ref{iso:01}) and (\ref{iso:11}) of Theorem \ref{isotopyconcordance} for all $t_0$ sufficiently large.

\subsection{Curvature bounds}\label{bend:curvature}

In order to finish the proof of Lemma \ref{isotopyconcordancelemma} and hence Theorem \ref{isotopyconcordance}, we must provide suitable bounds for the curvature of $G(\nu)$. These bounds will be in terms of the defining function $\Gamma(t)$ of (\ref{Gamma}), the coefficients $\alpha(t_0,t_1)$ and $\beta(t_0,t_1,r_0,r_1)$ (\ref{alphabeta}), and the path of metrics $g_s$. In order for the metric to satisfy (\ref{iso:01}) and (\ref{iso:11}) of Theorem \ref{isotopyconcordance} we bound the principal curvatures of the boundary from above and below in Proposition \ref{secondbound} below. In order for the metric to have positive Ricci curvature we bound the Ricci curvatures from below in Propositions \ref{riccibound}, \ref{spacericcibound}, and \ref{mixedricci}. 

Before we move on to computing the curvature of $G(\nu)$, we must fix some notation. Take the metric $\overline{G} = ds^2 + g_s$ on $[0,1]\times \En^n$. For a fixed point $p\in \En^n$, let $X_i(s)$ be a normal coordinate frame for $\En^n$ with respect to $g_s$, so that $\partial_s$ and $\Ex_i(s)$ forms an orthonormal frame for $\overline{G}$ at $(s,p)\in [0,1]\times \En^n$ with respect to $\overline{G}$. Similarly, define $Y_i = X_i/ \rho^2(t)$ so that $\partial_t$ and $Y_i$ form an orthonormal frame for $G$ at $(t,p) \in [t_0,t_1]\times \En^n$. We will let $\nabla^G$ and ${\nabla^{\overline{G}}}$ denote the Levi-Civita connection for the respective metrics, let $\K_G$ and $\K_{\overline{G}}$ denote the sectional curvature of the respective metrics, $\Ric_G$ and $\Ric_{\overline{G}}$ will denote the Ricci tensor of the respective metrics, $\2_{G}$ and $\2_{\overline{G}}$ will be the scalar valued second fundamental forms of the level sets of $s$ and $t$ with respect to the unit normal $\partial_s$ and $\partial_t$ and the respective metrics, and $G_t = R^2(t) g_{\lambda}(t)$. 

From its definition, we natural expect to see the curvatures of $\overline{G}$ to appear in the computation of the curvatures of $G(\nu)$. To aid with the proof of Lemma \ref{isotopyconcordancelemma}, we will replace the relevant curvature of $\overline{G}$ with a single constant. 

\begin{definition}\label{Cdef} Throughout Appendix \ref{bend}, we will fix $C>0$ to be greater than the supremum of $|\2_{\overline{G}}(X,X)|$, $|\partial_s \2_{\overline{G}} (X,X)|$, and $|\Ric_{\overline{G}}(X,\partial_s)|$ over all $(s,p)\in[0,1]\times \En^n$ and all unit vectors $X$ with respect to $\overline{G}$ tangent to $\{s\}\times \En^n$. 
\end{definition}

\noindent We emphasize that this choice of $C$ depends only on the family $g_s$, and is entirely independent of our choices of $\lambda(t)$ and $\rho(t)$ as well as the various constants $t_0$, $t_1$, $r_0$, and $r_1$. 

We can now compute and bound the principal curvatures of $[t_0,t_1]\times \En^n$ with respect to $G(\nu)$ in terms of the principal curvatures of $\overline{G}$ and the functions $\Gamma(t)$, $\alpha(t_0,t_1)$, and $\beta(t_0,t_1,r_0,r_1)$ used in the definition of $\lambda(t)$ and $\rho(t)$ in (\ref{Rlambda}). 

\begin{prop}\label{secondbound} We have the following bounds. 
\begin{align}
\label{secondlower} \2_G(Y_i,Y_i) & > \left( \frac{1}{t} +  \left(\frac{1}{\beta} - \frac{C}{\alpha}\right)\Gamma(t) \right) \rho^2(t) g_s \\
\label{secondupper} |\2_G(Y_i,Y_i)| & < \frac{1}{t} +\left(  \frac{1}{\beta} + \frac{C}{\alpha} \right) \Gamma(t)  
\end{align}
\end{prop}
\begin{proof} A straightforward application of \cite[Proposition 3.2.1]{Pet} yields
\begin{equation}\label{concsecond} \2_G  = \dfrac{\rho'(t)}{\rho(t)}    \rho^2(t) g_s  + \lambda'(t)  \rho^2(t)  \2_{\overline{G}} .
\end{equation}
From Definition \ref{Cdef}, $\2_{\overline{G}}> -C g_s$. This combined with (\ref{Rlambda}) and equation (\ref{concsecond}) yields the lower bound (\ref{secondlower}). Similarly from Definition \ref{Cdef}, $|\2_{\overline{G}}|<C$. This combined with (\ref{Rlambda}) and equation (\ref{concsecond}) yields the bound (\ref{secondupper}). 
\end{proof}

We now turn to bound the Ricci curvature of $G(\nu)$ in terms of the curvatures of $\overline{G}$ and again in terms of the defining functions of (\ref{Rlambda}).  We begin by computing the Ricci curvature in the time direction using \cite[Proposition 3.2.11]{Pet}. 
 
\begin{prop}\label{riccibound} The following lower bound holds. 
\begin{equation} \label{riccitimelower}\Ric_G(\partial_t,\partial_t) > n \left(  \frac{1}{\beta}  - \frac{C}{\alpha}   \right) \left( \Gamma'(t) +\frac{2 \Gamma(t) }{t} \right) - 4n \left( \frac{1}{\beta} + \frac{C}{\alpha}\right)^2 \Gamma^2(t) .
\end{equation}
\end{prop}

\begin{proof} As we have already computed $\2_G$ in (\ref{concsecond}), we may compute and bound $\partial_t\2_G$ as follows. 
\begin{align} 
\nonumber \partial_t\2_G &=\left[  - \dfrac{1}{t^2} + \left(\dfrac{\rho'(t)}{\rho(t)}\right)'   + 2 \left( \frac{1}{t} +\frac{\rho'(t)}{\rho(t)}\right)^2 \right] ( t^2 \rho^2(t)) g_s \\
\nonumber & \qquad\qquad +  \left[ 2\left( \frac{1}{t} +\frac{\rho'(t)}{\rho(t)}\right) \lambda'(t)  +  \lambda''(t) \right] (t^2 \rho^2(t)) \2_{\overline{G}} + (\lambda'(t))^2 (t^2 \rho^2(t)) (\partial_s \2_{\overline{G}})\\
\nonumber &< \left[\frac{1}{t^2}+  \left(\left(\frac{\rho'(t)}{\rho(t)} \right)' + 4 \frac{\rho'(t)}{t\rho(t)}\right) + C \left( |\lambda''(t) |+ 2\frac{\lambda'(t)}{t} \right) + 2\left( \frac{\rho'(t)}{\rho(t)} + C\lambda'(t) \right)^2 \right](t^2\rho^2(t)) g_s.\\
\label{derivativeofsecond} &< \left[  \frac{1}{t^2}  -\frac{1}{\beta}\left(\Gamma'(t) + 4 \frac{\Gamma(t)}{t}\right) + \frac{C}{\alpha} \left( |\Gamma'(t)| +  2\frac{\Gamma(t)}{t} \right)  + 2\left(\frac{1}{\beta} + \frac{C}{\alpha} \right)^2\Gamma^2(t) \right](t^2\rho^2(t))g_s.
\end{align}
Where we have used definition (\ref{Rlambda}) repeatedly. 

Combining (\ref{derivativeofsecond}) with (\ref{secondlower}) using \cite[Proposotion 3.2.11]{Pet} yields. 
\begin{equation}\label{timesectional} \K_G(\partial_t, Y_i) >   \left(  \dfrac{1}{\beta} - \dfrac{C}{\alpha}\right) \left( \Gamma'(t)+ \frac{2\Gamma(t)}{t} \right) - 4 \left(\frac{1}{\beta}+\frac{C}{\alpha}\right)^2\Gamma^2(t). 
\end{equation}
Taking the sum  of (\ref{timesectional}) over $1\le i\le n$ yields (\ref{riccitimelower}). 
\end{proof}

\noindent Next we compute Ricci curvature in the space directions using the Gauss equation. 

\begin{prop}\label{spacericcibound} The following lower bound holds. 
\begin{equation}\label{spacericcilower} \begin{split}
\Ric_G(Y_i,Y_i) & > (n-1)\left({\frac{ K_{g_s}(X_i,X_j)}{\rho^2(t)} -1}\right)\frac{1}{t^2}\\
&  -\left[ \left(\frac{1}{\beta}+\frac{C}{\alpha} \right) \left( \Gamma'(t) +(n+4) \frac{\Gamma(t)}{t}\right) +   \left( \frac{1}{\beta} + \frac{C}{\alpha}\right)^2\Gamma^2(t)\right].\end{split}
\end{equation}
\end{prop}
\begin{proof} As we have already bound $\K_G(\partial_t,Y_i)$ above in (\ref{timesectional}), it remains to bound $\K_G(Y_i,Y_j)$. Applying the Guass equation we have:
\begin{align}
\nonumber \K_G(Y_i,Y_j) & = \K_{G_s}(Y_i,Y_j) - \2_{G}(Y_i,Y_i) \2_G(Y_j,Y_j) \\ 
\label{spacesectionallower} & > \left(\frac{1}{t^2\rho^2(t)} \K_{g_s}(X_i,X_j) - \frac{1}{t^2} \right) - C\left[\left(\frac{1}{\alpha}+\frac{1}{\beta}\right)\frac{\Gamma(t)}{t} + \left(\frac{1}{\alpha}+\frac{1}{\beta}\right)^2\Gamma^2(t)\right].
\end{align}
 \noindent Where we used (\ref{secondupper}) and the fact that $\K_{G_s}(Y_i,Y_j)= \frac{1}{t^2\rho^2(t)} \K_{g_s}(X_i,X_j)$. Taking the sum of (\ref{spacesectionallower}) over $j\neq i$ with (\ref{timesectional}) yields (\ref{spacericcilower}). 
\end{proof}

\noindent Finally, the mixed Ricci curvatures can be computed from the Codazzi equation. 

\begin{prop}\label{mixedricci} The following bound holds
\begin{equation}\label{mixedriccibound} | \Ric_G(\partial_t, Y_i)| <  \frac{C\Gamma(t)}{\alpha  t \rho(t)} 
\end{equation}
\end{prop} 
\begin{proof} Applying the Codazzi equation to (\ref{concsecond}) we have:
\begin{align}
\nonumber \R_G(Y_i,Y_j,Y_j,\partial_t)  &=  -(\nabla^G_{Y_i}\2_G)(Y_j,Y_j) + (\nabla^G_{Y_j} \2_G)(Y_i,Y_j) \\
\nonumber & =  \dfrac{\rho'(t)}{\rho(t)}\left( - (\nabla^G_{Y_i} (t^2\rho^2(t)) g_s)(Y_j,Y_j) + (\nabla^G_{Y_j} (t^2\rho^2(t)) g_s))(Y_i,Y_j)\right) \\
\nonumber&\qquad \lambda'(t) (t^2\rho^2(t)) \left(  -(\nabla^G_{Y_i} \2_{\overline{G}} )(Y_j,Y_j) +(\nabla^G_{Y_j }\2_{\overline{G}} )(Y_i,Y_j) \right).\\
\label{mixedcurvature} & =  \frac{\lambda'(t)}{t\rho(t)}\left[  - \left( X_ i \2_{\overline{G}}(X_j,X_j)-2 \2_{\overline{G}}(\nabla_{X_i}^G X_j,X_j)\right) \right.\\
\nonumber & \left. \qquad\qquad+ \left(X_j \2_{\overline{G}}(X_i,X_j)- \2_{\overline{G}}(\nabla_{X_j}^G X_i,X_j)-\2_{\overline{G}}( X_i,\nabla_{X_j}^GX_j)  \right)  \right] 
\end{align}
Where we have used the fact that $\nabla_{Y_i}^G (t^2\rho^2(t))g_s = \nabla_{Y_i}^G G =0$. By \cite[Theorem 1.159]{Besse}, we have that
$$\nabla_{X_j}^G X_i = \nabla_{X_j}^{\overline{G}} X_i + \left(\frac{1}{t} + \frac{\rho'(t)}{\rho(t)}\right) g_s(X_i,X_j) \partial_t.$$
 Combining this, with (\ref{mixedcurvature}) we have shown that
 \begin{equation}\label{mixedcurvature2} \R_{G}(Y_i,Y_j,Y_j,\partial_t) = \frac{\lambda'(t)}{t\rho(t)} \R_{\overline{G}}(X_i,X_j,X_j,\partial_s).
 \end{equation}
After taking the sum of (\ref{mixedcurvature2}) over $j\neq i$ and applying the definition (\ref{Rlambda}) yields (\ref{mixedriccibound}). 
\end{proof}

%%%%%%%%%%%%%%%%%%

\subsection{Proof of Theorem \ref{bendboundary}}\label{isotopy:proof}

%%%%%%%%%%%%%%%%%%

Now that we have bound the principal and Ricci curvatures of $G(\nu)$ we can prove Lemma \ref{isotopyconcordancelemma}, hence Theorem \ref{isotopyconcordance}, and hence Theorem \ref{bendboundary}. 

\begin{proof}[Proof of Lemma \ref{isotopyconcordancelemma}]  Fix $1>r_1>0$ so that $2r_1< \nu$, and so that $\Ric_{g_s}(X_i,X_i)>2 r_1^2$. Fix $0<r_0 < r_1$ so that $\ln r_1 - \ln r_0 > C$. We claim that $G(\nu) =\dfrac{1}{t_0^2r_1^2}G(t_0,t_1,r_0,r_1)$ will have positive Ricci curvature and satisfy (\ref{iso:00})-(\ref{iso:11}) of Theorem \ref{isotopyconcordance} for $t_1=t_0^2$ and $t_0$ sufficiently large. By Proposition \ref{iso00and10}, this will complete the proof of Lemma \ref{isotopyconcordancelemma}.

We note that at $\{t_0\}\times \En^n$ the second fundamental form of the boundary is $-\2_{G(\nu)}$ as the outward normal is $-\partial_t$ rather than $\partial_t$. Thus from (\ref{secondupper}) we have that the second fundamental form of the boundary $\{t_0\}\times \En^n$ with respect to $G(\nu)$ is
$$ -\2_{G(\nu)}|_{\left(\{t_0\}\times \En^n\right)}  > - t_0r_1 \left[ \frac{1}{t_0} + \left( (\ln r_1 - \ln r_0 ) + C\right) \frac{2 }{ t_0 \ln t_0}\right].$$
We may choose $t_0$ so large that the second term inside the brackets is less than $\frac{1}{t_0}$. This choice of $t_0$ depends only on $( \ln (r_1/r_0 )-C)$ and hence depends only $g_s$. We conclude that $-\2_{G(\nu)}> - 2r_1 > -\nu$ for such $t_0$. Thus there is a $t_0$ depending only on $g_s$ for which (\ref{iso:01}) of Theorem \ref{isotopyconcordance} holds. Similarly, by applying (\ref{secondlower}) to $\2_{G(\nu)}$ at $\{t_1\}\times \En^n$ we have
$$\2_{G(\nu)}|_{\left(\{t_1\}\times \En^n\right)}  > t_0 r_1 \left( \frac{1}{t_0^2} - \left( (\ln r_1 - \ln r_0) + C \right) \frac{1}{2 t_0^2 \ln t_0}\right).$$
Clearly if $t_0$ is chosen large enough, the large parenthetical term will be positive. This $t_0$ again clearly depends only on $(\ln(r_1/r_0)-C)$. Hence there is a $t_0$ depending only on $g_s$ for which (\ref{iso:11}) of Theorem \ref{isotopyconcordance} is satisfied. 

It remains to show that $G(\nu)$ has positive Ricci curvature for $t_0$ chosen sufficiently large in a manner that depends only on $g_s$. We note that we may disregard the scaling by $1/(t_0 r_1)^2$, and show that $G=G(t_0,t_0^2,r_0,r_1)$ has positive Ricci curvature. We begin by plugging in the definitions of $\Gamma(t)$, $\alpha(t_0,t_1)$, and $\beta(t_0,t_1,r_0,r_1)$ of (\ref{Gamma}) and (\ref{alphabeta}) into the lower bounds provided in Propositions \ref{riccibound}, \ref{spacericcibound}, and \ref{mixedricci}.
%\begin{align*}
%\frac{1}{ (\ln r_1 - \ln r_0) \beta}\Gamma''(t) = \frac{1}{\alpha} \Gamma''(t)  & =  - \frac{2\ln t_0}{t^2 \ln^2 t}   - \frac{4\ln t_0}{t^2 \ln^3 t} \\
%\frac{1}{ (\ln r_1 - \ln r_0) \beta} \frac{\Gamma'(t)}{t}  = \frac{1}{\alpha} \frac{\Gamma'(t)}{t}  & = \frac{ 2\ln t_0}{t^2 \ln^2 t} \\
%\left( \frac{1}{\alpha}+\frac{C}{\beta} \right)^2 \Gamma^2(t) & =\left( (\ln r_0 - \ln r_1  )+ C\right)^2 \frac{ 4\ln^2 t_0}{t^2 \ln^4 t}.
%\end{align*}

\begin{align}
\label{timeexplicit} \Ric_G(\partial_t,\partial_t) & > n \left( (\ln r_1 - \ln r_0 ) - C\right) \left(\frac{2\ln t_0}{t^2\ln^2 t} - \frac{4\ln t_0}{t^2\ln^3 t} \right)  - 4n ( (\ln r_1 - \ln r_0) +C)^2 \frac{ 4\ln t_0}{t^2\ln^4 t} \\
\label{spaceexplicit} \Ric_G(Y_i,Y_i)  & > \left( \frac{\Ric_{g_s}(X_i,X_i)}{\rho^2(t) } - 1\right) \frac{(n-1)}{t^2} \\
\nonumber & \qquad \qquad - ((\ln r_1 - \ln r_0 )+C) \left( \frac{2\ln t_0}{t^2\ln^2 t} - \frac{4\ln t_0}{t^2\ln^3 t}  + (\ln r +1 ) \frac{4\ln t_0}{t^2\ln^4 t}\right),\\
\label{mixedexplicit} |\Ric_G(\partial_t, Y_i) | &  < \frac{C}{r_0} \frac{2\ln t_0}{t^2\ln^2t}.
\end{align}

Consider first (\ref{timeexplicit}). Note that of those terms that depend on $t$, that $\frac{\ln t_0}{t^2 \ln^2 t}$ decays the slowest and hence its coefficient will determine the sign in the limit. We have chosen $r_0$ and  $r_1$ so that its coefficient is positive. Hence there is a $T>0$ that is independent of all other choices for which $\Ric_G(\partial_t,\partial_t)>0$ for $t_0>T$. Next consider (\ref{spaceexplicit}). Note that of thos terms that depend on $t$, that $\frac{1}{t^2}$ decays the slowest an hence its coefficient will determine the sign in the limit. Because $\rho(t)< r_1$ and $\Ric_{g_s}(X_i,X_i) > r_1^2 \ge \rho^2(t)$ by assumption, the coefficient of $\frac{1}{t^2}$ is positive. We again conclude that there is a $T>0$ that is independent of all other choices for which $\Ric_G(Y_i,Y_i)>0$ for $t_0>T$. We will assume henceforth that $t_0>T$. 

As $\Ric_G$ is not diagonalized, we must show that we can choose $t_0>T$ for which $\Ric_G(Z,Z)>0$ for all unit tangent vectors $Z$. At a fixed point in $[t_0,t_1]\times \En^n$, we can decompose $Z$ (or its negative) as $Z = \cos \theta \partial_t + \sin \theta Y$ for some $0\le \theta\le \pi/2$, where we can assume that $Y=Y_i$. We must therefore show that the following is positive
\begin{equation} \label{foiled} \Ric_G(Z,Z) = \cos^2 \theta \Ric_G(\partial_t,\partial_t) + \sin \theta\cos\theta \Ric_G(\partial_t,Y_i) + \sin^2 \theta \Ric_G( Y_i,Y_i).\end{equation} 
Note that we have no control over the sign of $\Ric_G(\partial_t,Y_i)$, and so must show that we can dominate this term by one of the others. 

Begin by fixing $\theta_0>0$ so small that for all $0\le \theta \le \theta_0$ that 
\begin{equation}\label{foiledcoefficient} \cos^2 \theta  n ((\ln r_1-\ln r_0) -C) - \sin \theta \cos \theta  \frac{ C}{r_0}  >((\ln r_1-\ln r_0) -C) .\end{equation}
Note that $\theta_0$ depends on $g_s$ and $r_0$. For such $0\le\theta\le \theta_0$ consider (\ref{foiled}), we note that $\sin^2\theta \Ric_G(Y_i,Y_i)$ is positive albeit very small. Note that of the remaining terms in (\ref{foiled}) that depend on $t$, $\frac{2\ln t_0}{t^2 \ln^2 t}$ decays the slowest. The coefficient of this term is precisely (\ref{foiledcoefficient}), which we have assumed is bounded away from zero independently of $\theta$. We may therefore find a $T>0$ that is independent of all other choices, for which $\Ric_G(Z,Z)>0$ for all $t_0>T$ and $0\le \theta\le \theta_0$. 

Next we fix $\theta_0 \le \theta \le \pi/2$. Of the terms in (\ref{foiled}) that depend on $t$, the term $\frac{1}{t^2}$ decays the slowest. We note that the coefficient of $\frac{1}{t^2}$ is bounded below by a constant that depends only on $g_s$ and $r_0$. We may therefore find a $T>0$ that depends on $g_s$ and $r_0$ so that $\Ric_G(Z,Z)>0$ for all $t_0>T$ and $\theta_0\le \theta\le \pi/2$. 

We note that $R = t_0 r_1/r_0$ depends on $t_0$ and $r_1/r_0$. Our choice of $r_1/r_0$ depends only on $g_s$, while our choice of $t_0$ depends only on $g_s$ and $r_0$. 
\end{proof}

%%%%%%%%%%%%%%%%%%%%%%%%

\section{Gluing and Smoothing for Manifolds with Corners}\label{corners}

%%%%%%%%%%%%%%%%%%%%%%%%

The approach taken to prove Theorem \ref{glue} in \cite{Per1} is to smooth the metric in normal coordinates (see additional proofs in \cite[Lemma 2.3]{Wang1} and \cite[Theorem 2]{BWW}). As we have explained in Section \ref{outline:corners}, in order to prove Theorem \ref{gluecorners} we must also smooth the corners of $\Ex_c^n$. We can build off the ideas of \cite{Per1} and use the exact same smoothing process on the corners as we do on the metric. The arguments of \cite{Per1} automatically ensure Ricci-positivity is preserved, and it falls to us to verify that boundary convexity can also be preserved. We introduce in Section \ref{corners:splines} the smoothing process used to prove Theorem \ref{glue} in \cite{Per1} and describe the behavior that will be relevant to our proof of Theorem \ref{gluecorners}. Next in Section \ref{corners:metric} we will investigate the metric in normal coordinates and compute the second fundamental form of the faces in these coordinates. We conclude in Section \ref{corners:proof} by giving the proof of Theorem \ref{gluecorners}. We will compute the quantity $\partial_s^2 (g|_{\Wy})$ in normal coordinates below in Section \ref{corners:metric} and will show how Corollary \ref{glueconcave} follows from Theorem \ref{gluecorners} in Section \ref{corners:proof}.

%%%%%%%%%%%%%%%%%%%%%%%%

\subsection{Polynomial Splines}\label{corners:splines}

%%%%%%%%%%%%%%%%%%%%%%%%

In this section we consider a general Euclidean vector bundle $\Ee \rightarrow \Em^n$ bundle metric, and suppose we are given a continuous one parameter family of sections $F(a) \in C(\mathbf{R}, \Gamma(\Ee))$. We assume that this family is smooth everywhere except at $a=0$, where we will assume $F(a)$ is only finitely differentiable. In this section we will describe how to smooth this family $F(a)$ at $a=0$. 

In our application, we will consider $\Ee$ to be one of two bundles:  $\text{Sym}^2(T^* \Em^n)$ which inherits a bundle metric from a Riemannian metric on $\Em^n$, or it will be the trivial bundle $\mathbf{R}\times \Em^n \rightarrow \Em^n$ equipped with the standard inner product on $\mathbf{R}$. This first case will correspond to Riemannian metrics in normal coordinates, and the second case will correspond to smooth functions in normal coordinates (which can be used to describe the faces of a manifold with faces). Smoothing these family of sections will correspond directly to smoothing the two metrics glued together along the boundary and smoothing the corners of $\Ex_c^n$. 

The approach to smoothing such a family originally suggested in \cite{Per1} and carried out in detail in \cite[Theorem 2]{BWW} is to replace $F(a)$ with a polynomial family of sections on a neighborhood of $a=0$ up to some desired order of differentiability. Such polynomials are typically referred to as \emph{polynomial splines}, and were first introduced in \cite{Ferg} by a Boeing engineer\footnote{The goal of \cite{Ferg} was to give an efficient program to produce a once-differentiable parametric surface in 3-space from a given array of points. Presumably this was used to design real-life airplane components. We use it here to design Riemannian manifolds.}and should rightfully be called \emph{Ferguson splines.} 

In our application, we will consider only one-parameter polynomials defined on symmetric intervals. Define $F_{k,\varepsilon}(a) \in \Gamma(\Ee)[a]$ to be the unique degree $2k+1$ polynomial defined on $[-\varepsilon,\varepsilon]$ by the system of $2k+2$ equations:
\begin{equation}\label{ferguson} F_{k,\varepsilon}^{(i)}(\pm \varepsilon) = F^{(i)}(\pm \varepsilon) \text{ such that } 0\le i \le k.\end{equation}
The beauty of these Ferguson splines is that the coefficients of are rational functions in $\varepsilon$, and the coefficients of said rational functions are \emph{linear} in terms of $F^{(i)}(\pm \varepsilon)$. Moreover, the explicit formula of these coefficients depend only $k$ and are easy to explicit compute. This allows us to reduce any conditions we wish $F_{k,\varepsilon}(a)$ to satisfy to conditions expressed entirely in terms of $F^{(i)}(\pm \varepsilon)$.

The approach suggested by \cite{Per1} was to first perform a $C^1$ smoothing on $[-\varepsilon,\varepsilon]$, followed by a $C^2$ smoothing on $[\pm \varepsilon -\delta, \pm \varepsilon + \delta]$. It turns out that this allows us to have better control on the derivatives of the spline, rather than jumping straight for a $C^2$ smoothing on $[-\varepsilon, \varepsilon]$. As we are dealing with curvature conditions in Theorem \ref{gluecorners} a $C^2$ smoothing will be sufficient. It is for this reason that we only need to describe the behavior of $F_{1,\varepsilon}(a)$ and $F_{2,\varepsilon}(a)$.

\begin{lemma}\label{firstorder} Let $F(a)$ be a continuous family of sections of a Euclidean vector bundle $\Ee\rightarrow \Em^n$ parameterized by $a\in \mathbf{R}$, which is smooth for $a\neq 0$. Let $F_{1,\varepsilon}(a)$ be the cubic polynomial in $\Gamma(\Ee)[a]$ uniquely defined by the following system of equations.
$$ F_{1,\varepsilon}^{(i)}( \pm \varepsilon) = F^{(i)}(\pm \varepsilon) \quad i\in \{0,1\}.$$
Then for all $a\in[-\varepsilon,\varepsilon]$ and any differential operator $D$ on $\Ee$ we have:
\begin{enumerate}
\item\label{first:0} $DF_{1,\varepsilon}(a) = DF(0) + O(\varepsilon)$, \\
\item\label{first:1} $DF_{1,\varepsilon}'(a) = \dfrac{\varepsilon - a}{2\varepsilon} DF'(-\varepsilon) + \dfrac{\varepsilon + a}{2\varepsilon} DF'(\varepsilon) + O(\varepsilon)$,\\
\item\label{first:2} $F_{1,\varepsilon}''(a) = \dfrac{F_+'(0)-F_-'(0)}{2\varepsilon} + O(1)$. 
\end{enumerate}
\end{lemma}

\noindent Here we mean that $F(a)-G(a) = O(p(\varepsilon))$ for two $F,G\in C(\mathbf{R},\Gamma(\Ee))$  if $|F(a) - G(a)| = O(p(\varepsilon))$ in the standard sense, where this is the norm given by the Euclidean structure on $\Ee$. The proof of Lemma \ref{firstorder} is an interesting exercise following a computation of the coefficients of $F_{1,\varepsilon}(a)$ in terms of $F^{(i)}(\pm \varepsilon)$ and $\varepsilon$. See \cite[Appendix A.2.1.1]{BLB2} or \cite[Proof of Lemma 1]{BWW}\footnote{Technically \cite{BWW} only deals with the case where $\Ee$ is the space of symmetric $2$-tensors, but the argument works exactly the same for any Euclidean vector bundle. The first paragraph of the proof of \cite[Lemma 1]{BWW} proves (\ref{first:0}) and (\ref{first:1}) of Lemma \ref{firstorder} above and the second paragraph proves (\ref{first:2}).} for the details of the proof.  

As suggested by \cite{Per1} we describe the behavior $F_{2,\varepsilon}(a)$ only in the case that $F(a)$ is already once differentiable. 
\begin{lemma}\label{secondorder} Let $F(a)$ be a once-differentiable family of sections of a Euclidean vector bundle $\Ee\rightarrow \Em^n$ parameterized by $a \in \mathbf{R}$, which is smooth for $a\neq 0$. Let $F_{2,\varepsilon}(a)$ be the quintic polynomial in $\Gamma(\Ee)[a]$ uniqeuly defined by the following system of equations. 
$$ F_{2,\varepsilon}^{(i)}( \pm \varepsilon) = F^{(i)}(\pm \varepsilon) \quad i\in \{0,1,2\}.$$
Then for all $a\in [-\varepsilon,\varepsilon]$ and any differential operator $D$ on $\Ee$ we have: 
\begin{enumerate}
\item\label{second:0} $DF_{2,\varepsilon}(a) = DF(0)+O(\varepsilon)$,\\
\item\label{second:1} $DF_{2,\varepsilon}'(a) = DF'(0) + O(\varepsilon)$,\\
\item\label{second:2} $F_{2,\varepsilon}''(a) = \dfrac{2-p(a)}{4} F''(-\varepsilon) + \dfrac{2+p(a)}{4} F''(\varepsilon) + O(\varepsilon)$, where $p(a) = (5a^3)/(4\varepsilon^3) - (3a)/(4\varepsilon)$. 
\end{enumerate}
\end{lemma}

\noindent For a proof of Lemma \ref{secondorder} see \cite[Appendix A.2.1.2]{BLB2} or \cite[Proof of Theorem 2 continued]{BWW}\footnote{Again this only covers the case of symmetric $2$-tensors.}.

%%%%%%%%%%%%%%%%%%%%%%%%

\subsection{Normal Coordinates}\label{corners:metric}

%%%%%%%%%%%%%%%%%%%%%%%%

In order to apply the polynomial splines of Section \ref{corners:splines} to smooth the metric on $\Em^n$ and the corners of $\Ex_c^n$, we must describe the behavior of normal coordinates of the boundary on a neighborhood of each corner of $\Ex_c^n$. This will allow us to apply directly the proof of Theorem \ref{glue} from \cite{BWW} to ensure Ricci-positive of the smooth metric, and what remains is to check that the faces remain convex as we smooth the corner. 

 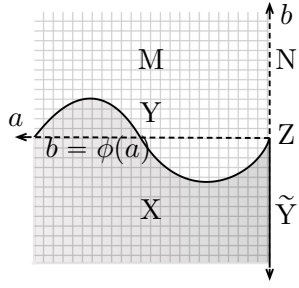
\begin{figure}
\centering

 \begin{tikzpicture}[scale=.1]
 \node (img)  {\includegraphics[scale=0.1]{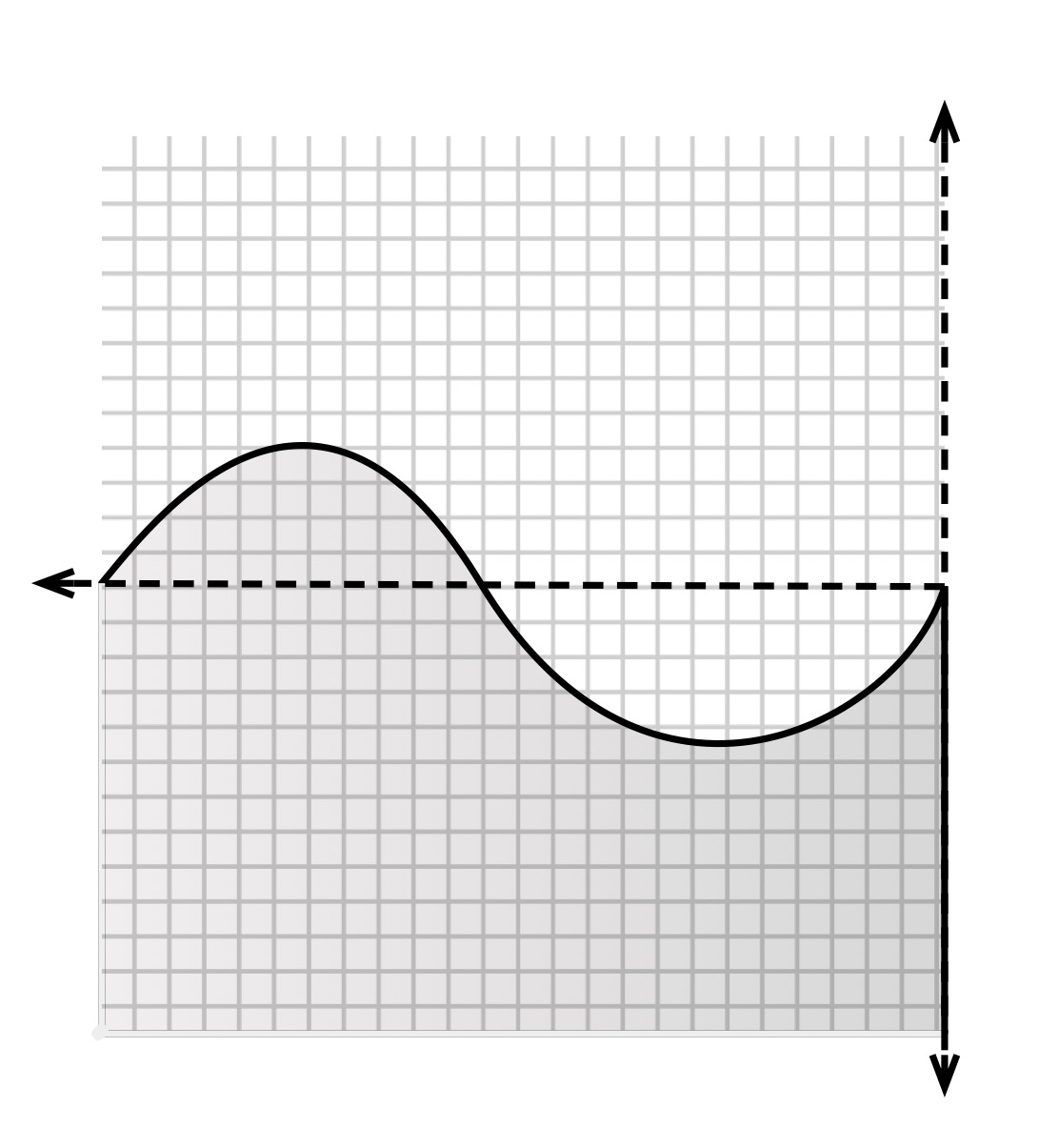}};
\node at (0,3) {$\Wy$};
\node at (18,-10) {$\widetilde{\Wy}$};
\node at (0,-10) {$\Ex$};
\node at (18,10) {$\En$};
\node at (18,0) {$\Zee$};
\node at (0,10) {$\Em$};
\node at (-18,2) {$a$};
\node at (18,16) {$b$};
\node at (-7,-2) {$b=\phi(a)$};

 \end{tikzpicture}

 \caption{An illustration of normal coordinates for $\Zee^{n-2}$ relative to $\widetilde{\Wy}^{n-1}$.}
  \label{fig:cornerchart}
\end{figure}

\begin{lemma}\label{cornernormal} Suppose we are given a manifold with faces $\Ex^n$ embedded within a Riemannian manifold $(\Em^n,g)$ with boundary $\En^{n-1}$ relative to its face $\widetilde{\Wy}^{n-1}$, designate a corner $\Zee^{n-2}$ that bounds $\widetilde{\Wy}^{n-1}$ and let $\Wy^{n-1}$ denote the other face bound by $\Zee^{n-2}$. There are normal coordinates $(a,b,z): (-a_0,0]\times (-b_0,b_0) \times \Zee^{n-2}\hookrightarrow \Em^n$ of $\Zee^{n-2}$ so that the coordinate $a\in (-a_0 ,0]$ agrees with the normal coordinates of $\En^{n-1}$ with respect to $g$. 

In these coordinates the metric $g$ splits as $da^2 + k(a)$, where $k(a)$ is a one parametr family of metrics on $\En^{n-1}$ and even further as $g=da^2+ \mu^2(a)db^2 + h(a,b),$
where $\mu(a)$ is a positive function satisfying $\mu(0)=1$ and $h(a,b)$ is a two parameter family of metrics on $\Zee^{n-2}$. 

In these coordinates $\Zee^{n-2}$ corresponds to the set $\{(a,b,z): a=b=0\}$; $\widetilde{\Wy}^{n-1}$ corresponds to the set $\{(a,b,z): a=0 \text{ and } b\le 0\}$; there exists a function $\phi(a,z)$ satisfying $\phi(0,z)=0$ such that $\Wy^{n-1}$ corresponds to the set $\{(a,b,z): a\le 0 \text{ and } b=\phi(a,z)\}$; and $\Ex^n$ corresponds to the set $\{(a,b,z): a\le 0 \text{ and } b\le \phi(a,z)\}$. 
\end{lemma} 
\begin{proof} Let $a\in (-a_0,0]$ denote the normal coordinates of $\En^{n-1}$ with respect to $g$, then $g=da^2 + k(a)$ where $k(a)$ is a metric on $\{a\}  \times \En^{n-1}$. For each $a$, these normal coordinates specify an embeddings $\{a\}\times \Zee^{n-2}\hookrightarrow \{a\}\times \En^{n-1}$. For each $a$, let $b(a)$ denote normal coordinates of $\{a\}\times \Zee^{n-2}$ inside of $\{a\} \times \En^{n-2}$ with respect to $k(a)$. Then $k(a) = db^2(a)  + h(a,b)= (b'(a))^2 db^2 + h(a,b)$. After a reparameterization we can assume that $b'(0)=1$, and we have $\mu(a) = b'(a)$ as desired. 

In these coordinates $\En^{n-1}$ corresponds to $a=0$ and $\Zee^{n-2}$ corresponds to $a=b=0$ by definition. It is clear that $\widetilde{\Wy}^{n-1}$ corresponds to the set $\{(a,b,z): a=0\text{ and } b\le 0\}$. As $\Wy^{n-1}$ is transverse to $\En^{n-1}$ and $\Wy^{n-1}\cap \En^{n-1}= \Zee^{n-2}$, $\Wy^{n-1}$ must also be transverse to the sets $\{a\}\times(-b_0,b_0)\times \Zee^{n-2}$ for all $a$ sufficiently small. Moreover $\Wy^{n-1} \cap\{a\}\times(-b_0,b_0)\times \Zee^{n-2} \cong \Zee^{n-2}$. After possibly choosing $a_0$ even smaller, we can assume that $\Wy^{n-1}$ for each $a\in (-a_0,0]$ and $z\in \Zee^{n-2}$ that there is a unique $\phi(a,z)\in [-b_0,b_0]$ such that $(a,\phi(a,z),z) \in \Wy^{n-1}$ and that $\Wy^{n-1}$ corresponds to the set $\{(a,b,z): b=\phi(a,z)\}$. It is now obvious that $\Ex^n$ corresponds to the set $\{(a,b,z): b\le \phi(a,z)\}$. 
\end{proof}

\noindent We refer to the coordinates of Lemma \ref{cornernormal} as \emph{normal coordinates for} $\Zee^{n-2}$ \emph{relative to the face} $\widetilde{\Wy}^{n-1}$.

For a manifold with faces, each face has a separate second fundamental form. In normal coordinates for $\Zee^{n-2}$ relative to the face $\widetilde{\Wy}^{n-1}$, the second fundamental form of $\widetilde{\Wy}^{n-1}$ agrees with the second fundamental form of $\En^{n-1}$ and hence can be computed as $(1/2)\partial_a k(a)$. The second fundamental form $\2$ of $\Wy^{n-1}$ can clearly be computed in terms of $h(a,b)$, $\mu(a)$, and $\phi(a,z)$ in these normal coordinates, and in order to prove Theorem \ref{gluecorners} we will need to have the exact formulas for $\2$ so we can maintain the convexity of the boundary as we smooth the corners of $\Ex_c^n$. 

\begin{lemma}\label{cornersecond} Let $(a,b,z):(-a_0,0]\times (-b_0,b_0)\times \Zee^{n-2}\hookrightarrow \Em^n$ be normal coordinates for $\Zee^{n-2}$ relative to a face $\widetilde{\Wy}^{n-1}$. Note in these coordinates that $T\Wy^{n-1} = \langle \tau \rangle \oplus T\Zee^{n-2}$, where $\tau \in T\left((-a_0 ,0]\times \mathbf{R}\right)$. 
\begin{align}
\2(\tau,\tau) & = \frac{-\mu \phi_{aa} -\phi_a\mu_a \left( \mu^2\phi_a^2+2\right)}{ \left( 1+ (\mu\phi_a)^2 \right) \sqrt{1+\mu^2(\phi_a^2 + \sum_i \phi_i^2)}}\\
\2|_{T\Zee^{n-2}} &= \frac{ -\phi_a  h_a\mu^2 + h_b }{2\mu \sqrt{1+ \mu^2\left(\phi_a^2 + \sum_i \phi_i^2\right)}}.
\end{align}
\end{lemma}
\begin{proof} First note that $\tau = \tau_a \partial_a + \tau_b \partial_b$, where
$$\tau_a = \frac{1}{\sqrt{1+\mu^2 \phi_a^2}} \text{ and } \tau_b = \frac{\mu^2}{\sqrt{1+\mu^2 \phi_a^2}}.$$
 Let $\partial_i$ with $1\le i \le n-2$ denote normal coordinates for $(\Zee^{n-2},h(a,b))$. Finally let $\nu_0=(\nu_a\partial_a + \nu_b \partial_b +\sum_i \nu_i\partial_i)$, so that $\nu = \nu_0/\|\nu_0\|$ is the unit normal of $\Wy^{n-1}$ with respect to $g$. Noting that $\tau$ and $\partial_i + \phi_i \partial_b$ spans $T\Wy^{n-1}$, it is straightforward to check that $\nu$ defined by the following is normal to $\Wy^{n-1}$ with respect to $g$. 
 $$\nu_a = -\mu^2 \phi_a, \quad \nu_b = 1 ,\text{ and } \nu_i = \phi_i/h_{ii} .$$

We begin by computing $\2(\tau,\tau) = -g(\nabla_\tau \tau,\nu)$. In order to do so we must compute the Christoffel symbols of $g$ with respect to the decomposition of Lemma \ref{cornernormal}. Those Christoffel symbols relevant to this computation are those of the form $\Gamma_{xy}^z$ where $x,y\in \{a,b\}$. The only nonzero symbols (up to symmetry) are the following:
 $$\Gamma_{ab}^b = \frac{\mu_a}{\mu}  , \quad \Gamma_{bb}^a = -\mu_a\mu .$$
 We begin by computing $g(\nabla_\tau \tau, \nu_0)$. 
\begin{equation}\label{secondtauexpand}
 g(\nabla_\tau \tau, \nu)=\nu_a(\tau_a(\partial_a \tau_a) + \tau_b (\partial_b \tau_a)) + \mu^2 \nu_b(\tau_a(\partial_a \tau_b)+ \tau_b(\partial_b \tau_b))   + \tau_a\tau_b \mu^2 \nu_b ( \Gamma_{ab}^b + \Gamma_{ba}^b) - \tau_b^2\nu_a \Gamma_{bb}^a .
\end{equation}

\noindent Note that $\partial_b \tau_a = \partial_b \tau_b =0$, and that
$$\partial_a \tau_a  = - \dfrac{ \mu^2\phi_a\phi_{aa} + \mu \mu_a \phi_a^2}{(1+\mu^2\phi_a^2)^{3/2}} \text{ and } \partial_a \tau_b =  \dfrac{\phi_{aa} - \mu\mu_a\phi_a^3}{(1+\mu^2 \phi_a^2)^{3/2}}.$$
Combining this computation into (\ref{secondtauexpand}) we have
\begin{align*}
g(\nabla_\tau \tau,\nu) & =  \frac{ \mu^2\phi_a  (\mu^2 \phi_a\phi_{aa} + \mu\mu_a \phi_a^2)}{(1+\mu^2\phi_a^2)^2} + \dfrac{\mu^2 (\phi_{aa} - \mu\mu_a\phi_a^3)}{(1+\mu^2\phi_a^2)^2}+\dfrac{2 \phi_a \mu \mu_a}{(1+\mu^2\phi_a^2)} +\frac{\phi_a^2 \mu^2\phi_a\mu_a\mu}{(1+\mu^2 \phi_a^2)},\\
& = \dfrac{ \phi_{aa}^2\mu^2 + \phi_a\mu_a\mu\left(\phi_a^2\mu^2 + 2\right)}{ 1+\mu^2\phi_a^2}.
\end{align*}
\noindent From which we conclude that 
$$\2(\tau,\tau) = -g(\nabla_\tau \tau,\nu_0)/\|\nu_0\| =  \dfrac{- \phi_{aa}^2\mu^2 - \phi_a\mu_a\mu\left(\phi_a^2\mu^2 + 2\right)}{  (1+\mu^2\phi_a^2)\sqrt{1+\mu^2\left(\phi_a^2 + \sum_i \phi_i^2\right)}} .$$

Next we compute $\2|_{T\Zee^{n-2}}$. In order to do so we must compute those Christoffel symbols of the form $\Gamma_{ij}^z$. The only nonzero symbols of this form (up to symmetry) are the following:
$$\Gamma_{ii}^a = -\frac{1}{2}\partial_a h_{ii} ,\quad \Gamma_{ii}^b = -\frac{1}{2} \frac{\partial_b h_{ii}}{\mu^2}.$$
We may now compute $\2(\partial_i,\partial_i)$
$$\2(\partial_i,\partial_i) = -g(\nabla_{\partial_i}\partial_i, \nu) = -\left( \nu_a \Gamma_{ii}^a + \mu^2\nu_b \Gamma_{ii}^b\right)/\|\nu_0\| = \dfrac{ - \phi_a\mu^2 \partial_a h_{ii} + \partial_b h_{ii}}{2\mu \sqrt{1+\mu^2\left(\phi_a^2 +\sum_i \phi_i^2\right)}}.$$
\end{proof}

In normal coordinates for $\Zee^{n-2}$ relative to the face $\widetilde{\Wy}^{n-1}$, we can consider the metric $g$ restricted to the other face $\Wy^{n-1}$, which by Lemma \ref{cornernormal} will take the form
$$g|_{\Wy} = (1+ \mu^2(a) \phi_a^2(a) )da^2 + h(a,\phi(a)) = ds^2 + h(a(s),b(s)),$$
where $a(s)$ is the arclength parameterization of $a$. In order to prove Corollary \ref{glueconcave}, we must compute $\partial_s^2 h(a(s),b(s))$ in terms of $\mu$, $h$, and $\phi$. 

\begin{lemma}\label{secondsderivative} Let $(a,b,z):(-a_0,0]\times (-b_0,b_0)\times \Zee^{n-2}\hookrightarrow \Em^n$ be normal coordinates for $\Zee^{n-2}$ relative to a face $\widetilde{\Wy}^{n-1}$. Then $g|_{\Wy^{n-1}} = ds^2 + h(a(s),b(s))$ and
$$\partial_s^2 h(a(s),b(s)) =      \frac{ \phi_a^2 h_{aa}}{1+\mu^2\phi_a^2} + \frac{ \left(-\phi_a\mu^2 h_a + h_b\right) \phi_{aa} }{(1+\mu^2\phi_a^2)^2} + D(h,\mu,\phi), $$
where $D$ is some differential operator that has order 1 with respect to $\partial_a$. 
\end{lemma}
\begin{proof} Clearly we have
$$
\partial_s^2 h(a(s),b(s)) = a'' h_a + b''h_b + (a')^2 h_{aa} + 2a'b' h_{ab} + (b')^2 h_{bb}.$$
To compute the derivatives of $a(s)$ and $b(s)$ we let $\phi(a)$ denote the arclength parameterization of $(a,\phi(a))$ wit respect to the metric $da^2 + \mu^2(a) db^2$. It is straightforward to compute 
\begin{align*}
\psi'(a) & = \sqrt{ 1+\mu^2 \phi_a^2}\\
\psi''(a) & = \dfrac{\mu\mu_a \phi_a^2 + \mu^2\phi_a\phi_{aa}}{\sqrt{1+\mu^2 \phi_a^2}}.
\end{align*}
As $a(s) = \psi^{-1}(s)$ and $b(s) = \phi(\psi^{-1}(s))$ we have:
\begin{align*}
a'(s) & = \dfrac{ 1}{\psi'(a(s))}\\
a''(s) & = - \dfrac{ \psi''(a(s))}{(\psi'(a(s)))^3}\\
b'(s) & =\dfrac{ \phi_a(a(s))}{\psi'(a(s))}\\
b''(s) & = \dfrac{\phi_{aa}(a(s))}{(\psi'(a(s)))^2} - \dfrac{\phi_a(a(s)) \psi''(a(s))}{(\psi'(a(s)))^3}.
\end{align*}
From this, we see immediately that the coefficient of $h_{aa}$ in $\partial_s h(a(s),b(s))$ is as claimed. We note that $\phi_{aa}$ only appears as a term in $a''(s)$ and $b''(s)$, and by examining the formulas we see the desired coefficient. 
\end{proof}

%%%%%%%%%%%%%%%%%%%%%%%%

\subsection{Proof of Theorem \ref{gluecorners} and its Corollary}\label{corners:proof}

%%%%%%%%%%%%%%%%%%%%%%%%

In this section we will prove Theorem \ref{gluecorners}. As promised we will use the smoothing process outline in Section \ref{corners:splines} to smooth both the metric on $\Em^n$ and the corners of the natural embedding $\Ex_c^n\hookrightarrow \Em^n$ to produce a smooth embedding $\Ex^n\hookrightarrow \Em^n$. Thanks to the proof of Theorem \ref{glue} in \cite[Theorem 2]{BWW}, we need only to pay attention to the boundary convexity during this smoothing process. If we let $\2$ denote the second fundamental form of the boundary of $\Ex_c^n$, it suffices to show that $\2(\tau,\tau)>0$ and that $\2|_{T\Zee^{n-2}}$ is positive definite.The formulas for these quantities given in Lemma \ref{cornersecond} are unnecessarily complicated by their denominators so we will instead work with the following quantities.
\begin{align}
\label{tauclear}\2(\tau,\tau)  \|\nu_0^2\| (1+\mu^2\phi_a^2)& =  -\phi_{aa}\mu - \phi_a\mu_a (\mu^2\phi_a^2 + 2)\\
\label{zedclear} \2|_{T\Zee^{n-2}} \|\nu_0^2\| & =  -\phi_ah_a \mu^2 +h_b.
\end{align}
In order to show that $\2$ is positive definite, it now suffices to show that (\ref{tauclear}) is positive and (\ref{zedclear}) is positive definite. 

 \begin{figure}
\centering

 \begin{tikzpicture}[scale=.1]
 \node (img)  {\includegraphics[scale=0.1]{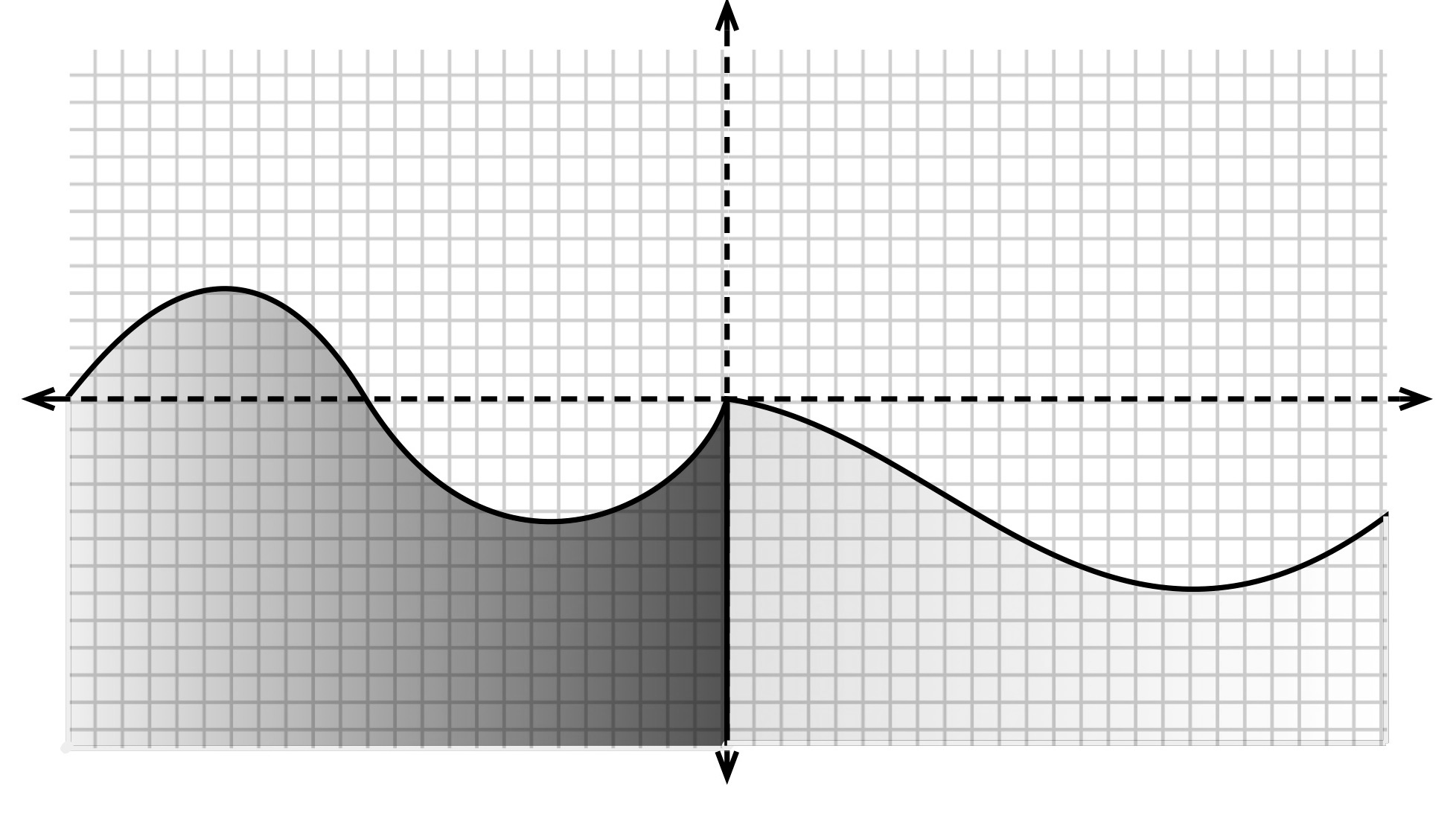}};
\node at (-15,5) {$\Wy_1$};
\node at (12,-1) {$\Wy_2$};
\node at (-25,15) {$\Em_1$};
\node at (27,15) {$\Em_2$};
\node at (-15,-13) {$\Ex_1$};
\node at (17,-13) {$\Ex_2$};
\node at (2,-8) {$\widetilde{\Wy}$};
\node at (2,4) {$\Zee$};
\node at (2,12) {$\En$};
\node at (25,-3) {$b=\phi_1(-a)$};
\node at (-23,-3) {$b=\phi_2(a)$};
\node at (30,3) {$a$};
\node at (2,17) {$b$};
 \end{tikzpicture}

 \caption{An illustration of the combined normal coordinates of $\Zee^{n-2}$. Note that $\Ex_c^n$ is the shaded region. }
  \label{fig:doublecharts}
\end{figure}
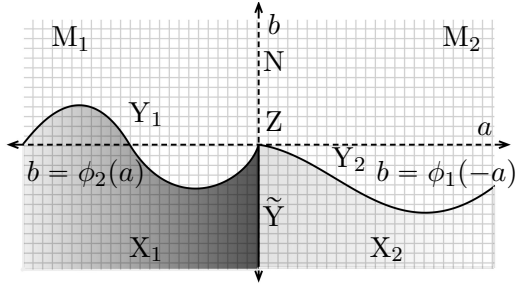

\begin{proof}[Proof of Theorem \ref{gluecorners}] For $a\in[-a_0,a_0]$ let $(-1)^{i+1}a$ denote normal coordinates for $\En^{n-1}_i$ with respect to $g_i$ for $(-1)^i a\ge 0$. By \cite[Lemma A.1]{Mil3} we can form a smooth manifold $\Em^n = \Em_1^n\cup_\Phi \Em_2^n$ using these coordinates, which determines an embedding $\Ex_c^n \hookrightarrow \Em^n$. Fix one of the corners $\Zee^{n-2}$ introduced along $\widetilde{\Wy}^{n-1}$. In these coordinates we can apply Lemma \ref{cornernormal} on $\Zee_{i}^{n-2}$ with respect to $(\Em_i^n,g_i)$ to define a function $\mu(a) = \mu_i\left((-1)^{i+1} a\right)$, a family $h(a,b)=h_i\left((-1)^{i+1}a,b\right)$ of metrics on $\Zee^{n-2}$, and a function $\phi(a,z)=\phi_i\left((-1)^{i+1} a,z\right)$. This combined normal coordinates is pictured in (\ref{fig:doublecharts}). Note that $\mu(a)$, $h(a,b)$, and $\phi(a,z)$ are smooth everywhere except at $a=0$ where they are continuous, and that $g= da^2 +\mu^2(a)db^2 + h(a,b)$ is smooth everywhere except at $a=0$ where it is continuous. 

Note that we may consider $h(a,b)$, $\mu(a)$, and $\phi(a,z)$ as families of sections of certain Euclidean vector bundles parameterized by $a\in [-a_0,a_0]$. For $h(a,b)$ the relevant bundle is $\pi^* \text{Sym}^2\left(T^*\Zee^{n-2}\right) \rightarrow [-b_0,b_0]\times \Zee^{n-2}$, where $\pi: [-b_0,b_0]\times \Zee^{n-2}\rightarrow \Zee^{n-2}$ is the projection given the natural bundle metric induced by $h(0,0)$; for $\mu(a)$ the relevant bundle is $\mathbf{R}\rightarrow \text{pt}$ with the standard metric; and for $\phi(a,z)$ the relevant bundle is $\mathbf{R}\times \Zee^{n-2}\rightarrow \Zee^{n-2}$ again with the standard metric. We may therefore define Ferguson splines $h_{k,\varepsilon}(a,b)$, $\mu_{k,\varepsilon}(a)$, and $\phi_{k,\varepsilon}(a,z)$ via (\ref{ferguson}) of Section \ref{corners:splines}.  

Define $\overline{h}(a,b)$, $\overline{\mu}(a)$, and $\overline{\phi}(a,z)$ by setting them equal to $h(a,b)$, $\mu(a)$, and $\phi(a,z)$ for $a\notin [-\varepsilon,\varepsilon]$ and setting them equal to $h_{1,\varepsilon}(a,b)$, $\mu_{1,\varepsilon}(a)$, and $\phi_{1,\varepsilon}(a,z)$ for $a\in[-\varepsilon,\varepsilon]$. The metric $\overline{g}$ defined by $\overline{g} =da^2 +\overline{\mu}^2(a)db^2 + \overline{h}(a,b)$ is $C^1$ by definition, and is smooth for $a\neq \pm \varepsilon$. Similarly the the manifold $\overline{\Ex}^n$ described by $\{(a,b,z): b\le \overline{\phi}(a,z) \}$ is a $C^1$ manifold with boundary that is smooth away from the boundary. By assumption the Ricci-curvature of $\overline{g}$ is positive for $a\notin [-\varepsilon,\varepsilon]$, and by \cite[Lemma 1]{BWW} $\overline{g}$ has positive Ricci curvature for $a\in(-\varepsilon,\varepsilon)$ for all $\varepsilon$ sufficiently small. If we let $\overline{\2}$ denote the second fundamental form of the boundary $\overline{\Wy}^{n-1} = \{(a,b,z): b=\overline{\phi}(a,z)\}$, then $\overline{\2}$ is positive definite for $a\notin[-\varepsilon,\varepsilon]$. What remains to be shown is that $\overline{\2}$ is positive definite for $a\in(-\varepsilon,\varepsilon)$ for all $\varepsilon$ sufficiently small. 

As explained, it suffices to show that (\ref{tauclear}) is positive and (\ref{zedclear}) is positive definite. First let us consider (\ref{tauclear}). By (\ref{first:0}) of Lemma \ref{firstorder}, $\overline{\mu}(a)$ can be bounded in terms of $\mu(0)$ for $a\in(-\varepsilon,\varepsilon)$. By (\ref{first:1}) of Lemma \ref{firstorder}, $\phi_a$ and $\mu_a$ can be bounded in terms of $\phi_a(\pm\varepsilon,z)$ and $\mu_a(\pm \varepsilon)$ for $a\in (-\varepsilon,\varepsilon)$. We conclude that (\ref{tauclear}) can be replaced with the following equation. 
\begin{equation}\label{taudom} -\overline{ \phi}_{aa}(a,z) \mu(a) +O(1)
\end{equation}

We now apply (\ref{first:2}) of Lemma \ref{firstorder} to study $\overline{\phi}_{aa}(a,z)$. 
$$\overline{\phi}_{aa}(a,z) = \dfrac{\partial_a{\phi}_2(0,z) - \partial_a{\phi}_1(0,z) }{2\varepsilon} + O(1).$$
We claim $\overline{\phi}_{aa}(a,z)\rightarrow -\infty$ for $a\in (-\varepsilon,\varepsilon)$ as $\varepsilon\rightarrow 0$. It suffices to show that $\partial_a{\phi}_2(0,z) - \partial_a{\phi}_1(0,z) <0$. We claim the negativity of this term is equivalent to the total dihedral angle along $\Zee^{n-2}$ not exceeding $\pi$. Let $\tau_i$ denote the following tangent vector of $\Wy^{n-1}_i$ that is tangent to $\Zee^{n-2}_i$ with respect to $g_i$
$$(-1)^{i+1}\tau_i = \partial_a + (\partial_a \phi_i(0,z))\partial_b.$$
Let $\nu_2$ be the following vector that is normal to $\tau_2$ in $\Em^n_2$ but is still normal to $\Zee_2^{n-2}$ with respect to $g_2$
$$\nu_2 = -(\partial_a \phi_2(0,z))\partial_a + \partial_b.$$
By assumption, the angle between $\tau_1$ and $\tau_2$ lies between $0$ and $\pi$. Equivalently, the angle between $\tau_1$ and $\nu_2$ lies between $-\pi/2$ and $\pi/2$. If we call this latter angle $\alpha$ then 
$$\|\tau_1\| \|\nu_2\| \cos \alpha = g( \tau_1,\nu_2) = \partial_a \phi_1(0,z) - \partial_a\phi_2(0,z).$$
Because $ -\pi/2 < \alpha < \pi/2$, this implies that $\partial_a \phi_1(0,z) - \partial_a\phi_2(0,z)>0$ as desired. We conclude that $\overline{\phi}_{aa}(a,z)\rightarrow -\infty$ for $a\in (-\varepsilon,\varepsilon)$ as $\varepsilon\rightarrow 0$. It follows that (\ref{taudom}) blows up to $+\infty$ for $a\in(-\varepsilon,\varepsilon)$ as $\varepsilon\rightarrow0$, and in particular is positive for some $\varepsilon$ sufficiently small. 

Next let us consider (\ref{zedclear}). By (\ref{first:0}) we have that $\overline{\mu}(a)=1+O(\varepsilon)$, $\overline{h}_b(a,\phi(a,z))= h_b(0,0)+O(\varepsilon)$, and $\overline{h}_a(a,\phi(a,z))= \overline{h}(a,0) + O(\varepsilon)$ for $a\in(-\varepsilon,\varepsilon)$. Substituting these into (\ref{zedclear}) we have for $a\in (-\varepsilon,\varepsilon)$,
\begin{equation}\label{zeddom} - \overline{\phi}_a(a,z) \overline{h}_a(a,0) + h_b(0,0) + O(\varepsilon).
\end{equation}
By applying (\ref{first:1}) of Lemma \ref{firstorder}, we have that $\overline{\phi}_a(a,z)$ and $\overline{h}_a(a,0)$ are arbitrarily close to linear interpolation between their boundary values on $[-\varepsilon,\varepsilon]$. 

We claim that for all $z$ and all tangent vectors $v\in T\Zee^{n-2}$ that the linear interpolation for $\overline{\phi}_a(a,z)$ and $\overline{h}_a(a,0)(v,v)$ have negative slopes. As already explained above, we have for $\varepsilon$ sufficiently small that
$$ \overline{\phi}_a(-\varepsilon,z) = \partial_a \phi_1(-\varepsilon,z) > \partial_a\phi_2(\varepsilon,z) = \overline{\phi}_a(\varepsilon,z).$$
We conclude that the linear interpolation for $\overline{\phi}_a(a,z)$ has negative slope. For $\overline{h}_a(a,0)$ we note that 
$$(-1)^{i+1}\widetilde{\2}_i = ( \partial_a\mu_i)(0)\mu_i(0) db^2 + \partial_a h_i(0,b).$$
By assumption $\widetilde{\2}_1(v,v)+\widetilde{\2}_2(v,v)>0$, which if we restrict to $T\Zee^{n-2}$ implies that $\partial_ah_1(0,b)(v,v) - \partial_ah_2(0,b)(v,v)>0$. We may choose $\varepsilon$ sufficiently small that $\partial_ah_1(-\varepsilon,b)(v,v) - \partial_ah_2(\varepsilon,b)(v,v)>0$ so that the linear interpolation for $\overline{h}_a(a,0)(v,v)$ has negative slope as well. 

It is an elementary fact that the product of two linear interpolations with negative slope is a concave up quadratic interpolation. We conclude that for $a\in (-\varepsilon,\varepsilon)$ that
\begin{equation}\label{concaveup} \overline{\phi}_a(a,z) \overline{h}_a(a,0)(v,v) < \frac{\varepsilon - a}{2\varepsilon} \overline{\phi}_a(-\varepsilon,z) \overline{h}_a(-\varepsilon,0)(v,v) + \frac{\varepsilon + a}{2\varepsilon} \overline{\phi}_a(-\varepsilon,z) \overline{h}_a(-\varepsilon,0) (v,v)+ O(\varepsilon).\end{equation}
Applying (\ref{concaveup}) to (\ref{zeddom}) we can bound (\ref{zeddom}) below by the following quantity. 
\begin{equation}\label{lowerbound} \frac{\varepsilon - a}{2\varepsilon}\left( - \overline{\phi}_a(-\varepsilon,z) \overline{h}_a(-\varepsilon,0)+h_b(0,0)\right)(v,v) + \frac{\varepsilon + a}{2\varepsilon}\left( \overline{\phi}_a(-\varepsilon,z) \overline{h}_a(-\varepsilon,0) +h_b(0,0)\right)(v,v)+ O(\varepsilon)
\end{equation}
By assumption (\ref{zeddom}) is positive definite for $a=\pm \varepsilon$ as $\overline{\phi}_a(\pm\varepsilon,z) = \phi_a(\pm \varepsilon,z)$ and $\overline{h}_a(\pm\varepsilon,b)=h_a(\pm\varepsilon,b)$. It is now clear that (\ref{lowerbound}) is approximately a convex combination of two positive quantities and hence is positive. We conclude that (\ref{zeddom}) is positive definite and hence (\ref{zedclear}) is positive definite if $\varepsilon$ is chosen small enough. 

Fix $\varepsilon$ sufficiently small so that $\overline{g}$ has positive Ricci curvature by \cite[Lemma 1]{BWW} and so that (\ref{tauclear}) is positive and (\ref{zedclear}) is positive definite and hence $\overline{\2}$ is positive definite. We have produce a $C^1$ metric on $\Em^n$ with positive Ricci curvature and an embedding of a $C^1$ manifold with boundary $\overline{\Ex}^n \hookrightarrow \Em^n$ with convex boundary. 

Define $\breve{h}(a,b)$, $\breve{\mu}(a)$, and $\breve{\phi}(a,z)$ by setting them equal to $\overline{h}(a,b)$, $\overline{\mu}(a)$, and $\overline{\phi}(a,z)$ for $a\notin [\pm \varepsilon - \delta,\pm \varepsilon +\delta]$ and setting them equal to $\overline{h}_{2,\delta}(a,b)$, $\overline{\mu}_{2,\delta}(a)$, and $\overline{\phi}_{2,\varepsilon}(a,z)$ for $a\in[\pm\varepsilon-\delta,\pm\varepsilon+\delta]$. The metric $\breve{g}$ defined by $\breve{g} =da^2 +\breve{\mu}^2(a)db^2 + \breve{h}(a,b)$ is $C^2$ by definition. Similarly the the manifold $\breve{\Ex}^n$ described by $\{(a,b,z): b\le \breve{\phi}(a,z) \}$ is a $C^2$ manifold with boundary. By assumption the Ricci-curvature of $\breve{g}$ is positive for $a\notin [\pm \varepsilon -\delta,\pm \varepsilon + \delta]$, and by \cite[Lemma 2]{BWW} $\breve{g}$ has positive Ricci curvature for $a\in(\pm\varepsilon -\delta, \pm \varepsilon + \delta)$ for all $\delta$ sufficiently small. If we let $\breve{\2}$ denote the second fundamental form of the boundary $\breve{\Wy}^{n-1} = \{(a,b,z): b=\breve{\phi}(a,z)\}$, then $\breve{\2}$ is positive definite for $a\notin[\pm \varepsilon -\delta,\pm \varepsilon +\delta]$. What remains to be shown is that $\breve{\2}$ is positive definite for $a\in(\pm \varepsilon - \delta, \pm \varepsilon + \delta)$ for all $\delta$ sufficiently small. 

We may find a shift of $\mathbf{R}$ sending either $[\pm\varepsilon-\delta,\pm\varepsilon+\delta]$ to $[-\delta,\delta]$. We may then argue that $\breve{\2}$ is positive definite for $a\in [-\delta,\delta]$ under the assumption that $\breve{h}(a,b)$, $\breve{\mu}(a)$, and $\breve{\phi}(a,z)$ are all defined via (\ref{ferguson}) on $[-\delta,\delta]$ using $C^1$ families of metrics and $\breve{\2}$ is positive definite for $a\notin [-\delta,\delta]$. This allows us to apply Lemma \ref{secondorder} directly and to avoid unnecessary repetition as both cases are entirely symmetric. 

We consider first showing that $\breve{\2}(\tau,\tau)$ is positive. Note that by (\ref{second:0}) and (\ref{second:1}) of Lemma \ref{secondorder} that $\breve{\mu}(a) = \overline{\mu}(0)+O(\varepsilon)$, $\breve{\mu}_a(a)= \overline{\mu}_a(0)+O(\varepsilon)$, $\breve{\phi}_a(a,z) = \overline{\phi}_a(0,z)+O(\varepsilon)$. These together show that (\ref{tauclear}) can be replaced by 
\begin{equation}\label{tauseconddom}  \breve{\phi}_{aa}(a,z) \overline{\mu}(0) + \overline{\phi}_a(0,z)\overline{\mu}_a(0)(\overline{\mu}^2(0)\overline{\phi}_a^2(0,z) +2) +O(\varepsilon) .\end{equation}
Applying (\ref{second:2}) of Lemma \ref{secondorder} to $\breve{\phi}_{aa}(a,z)$ we have that equation (\ref{tauseconddom}) is for each $a$ approximately some convex combination of its boundary values on $[-\delta,\delta]$. By assumption, these boundary values are positive as $\breve{\phi}_{aa}(\pm \delta,z) = \overline{\phi}_{aa}(\pm \delta, z)$. We conclude that (\ref{tauseconddom}) and hence (\ref{tauclear}) will be positive for all $a$ if $\delta$ is chosen sufficiently small. 

Next to show that $\breve{\2}|_{T\Zee^{n-2}}$ is positive definite, we note that all of the terms in (\ref{zedclear}) are arbitrarily close to the terms for $\overline{\2}|_{T\Zee^{n-2}}$ by (\ref{second:0}) and (\ref{second:1}) of Lemma \ref{secondorder} and hence $\breve{\2}|_{T\Zee^{n-2}}$ must be positive definite for all $\delta$ sufficiently small. 

Fix $\varepsilon$ sufficiently small so that $\breve{g}$ has positive Ricci curvature by \cite[Lemma 2]{BWW} and so that (\ref{tauclear}) is positive and (\ref{zedclear}) is positive definite and hence $\breve{\2}$ is positive definite. We have produce a $C^2$ metric on $\Em^n$ with positive Ricci curvature and an embedding of a $C^2$ manifold with boundary $\breve{\Ex}^n \hookrightarrow \Em^n$ with convex boundary. To conclude we note that we can find a smooth metric arbitrarily $C^2$-close to $\breve{g}$ so that it still has positive Ricci-curvature, and we can find a smooth function arbitrarily $C^2$-close to $\breve{\phi}(a,z)$ so that the smooth embedding of $\Ex^n\hookrightarrow \Em^n$ by this function still has convex boundary. 
\end{proof}

Having established Theorem \ref{gluecorners}, we now can prove Corollary \ref{glueconcave} by following the same smoothing process while keeping track of the quantity $\partial_s^2h(a(s),b(s))$ in Lemma \ref{secondsderivative}.

\begin{proof}[Proof of Corollary \ref{glueconcave}] Using the same notation as in the proof of Theorem \ref{gluecorners} we endeavor to show that with a perhaps smaller choice of $\varepsilon$ and $\delta$, we can also ensure that $\partial_s^2 h(a(s),b(s))$ is negative definite. 

First consider the formula for $\partial_s^2\overline{h}(\overline{a}(s),\overline{b}(s))$ given by Lemma \ref{secondsderivative}. By (\ref{first:0}) and (\ref{first:1}) of Lemma \ref{firstorder} we have that can bound $D(\overline{h},\overline{\mu},\overline{\phi})$ in terms of the boundary values of $D(h,\mu,\phi)$. We claim that the other two terms in Lemma \ref{secondsderivative} have eigenvalues that tend to $-\infty$ as $\varepsilon \rightarrow 0$. To see this is the case for $\overline{h}_{aa}$ we apply (\ref{first:2}) of Lemma \ref{firstorder} to see that
$$\overline{h}_{aa} = \dfrac{ \partial_a h_2(0,0) - \partial_a h_2(0,0)}{2\varepsilon} + O(1).$$
Note that $(-1)^{i+1}\2_i = \mu_i' \mu_i db^2 +(1/2) \partial_ah_i(0,b)$ and hence $\2_1+\Phi^*\2_2>0$ implies that $\partial_a h_2(0,0) - \partial_ah_2(0,0)$ is negative definite. We immediately conclude that the eigenvalues all tend to $-\infty$ as $\varepsilon\rightarrow 0$. As already observed in the proof of Theorem \ref{gluecorners} that $\overline{\phi}_{aa} \rightarrow -\infty$ as $\varepsilon\rightarrow 0$, so it suffices to show the rest of this term in Lemma \ref{secondsderivative} is positive definite. But this term equals (\ref{zedclear}), which was demonstrated to be positive definite for $\varepsilon$ sufficiently small in the proof of Theorem \ref{gluecorners}. We conclude that $\partial_s^2 \overline{h}(\overline{a}(s),\overline{b}(s))$ will be negative definite for $\varepsilon$ sufficiently small. 

The argument that $\partial_s^2\breve{h}(\breve{a}(s),\breve{b}(s))$ is negative definite for $\delta$ sufficiently small is entirely similar to the proof that $\breve{\2}(\tau,\tau)$ is positive above in the proof of Theorem \ref{gluecorners}. By (\ref{second:0}) and (\ref{second:1}) of Lemma \ref{secondorder} all the terms in the formula in Lemma \ref{secondsderivative} are arbitrarily close to the same terms for $\partial_s^2\overline{h}(\overline{a}(s),\overline{b}(s))$ except for $\breve{h}_{aa}$ and $\breve{\phi}_{aa}$. But by (\ref{second:2}) of Lemma \ref{secondorder} both $\breve{h}_{aa}$ and $\breve{\phi}_{aa}$ are arbitrarily close to the same convex combination of their boundary values. It follows that $\partial_s^2\breve{h}(\breve{a}(s),\breve{b}(s))$ is itself arbitrarily close to a convex combination of its boundary values, and hence is negative definite for $\delta$ sufficiently small. 

Again because $\partial_s^2 h(a(s),b(s))$ is only a second order quantity, we may also assume the smooth metric and function chosen in the proof of Theorem \ref{gluecorners} can also be chosen so that $\partial_s^2 h(a(s),b(s))$ is still negative definite. 
\end{proof}

%%%%%%%%%%%%%%%%%%%%%

\bibliographystyle{abbrv} 
\bibliography{references}

\end{document}